\documentclass[11pt]{amsart}
\usepackage{aliascnt}
\usepackage{amsmath,amsfonts,amssymb,amsthm}
\usepackage{latexsym,bm,graphicx}
\usepackage[msc-links, lite, alphabetic]{amsrefs}
\usepackage{mathrsfs}
\usepackage{color}
\usepackage{hyperref}
\usepackage[all]{xy}

\title[Pinching  and collapsing]{Fibrations, the First Betti Number, and Almost Nonnegative Ricci Curvature}
\author{Hongzhi Huang, Xian-Tao Huang, Jikang Wang, and Xingyu Zhu}
\address{Hongzhi Huang \\ College of Information Science and Technology \\ Jinan University\\ Guangzhou 510275}
\email{\href{mailto:hyyqsaax@163.com}{hyyqsaax@163.com}
}

\address{Xian-Tao Huang\\School of Mathematics\\  Sun Yat-sen University\\ Guangzhou 510275}
\email{\href{mailto:hxiant@mail2.sysu.edu.cn}{hxiant@mail2.sysu.edu.cn}
}

\address{Jikang Wang\\Department of Mathematics\\ University of California, Berkeley\\ Berkeley 94720}
\email{\href{mailto:jikangwang1117@gmail.com}{jikangwang1117@gmail.com}
}

\address{Xingyu Zhu \\ Department of Mathematics\\ Michigan State University\\ Lansing 48824}
\email{\href{mailto:zhuxing3@msu.edu}{zhuxing3@msu.edu}
}

\newtheorem{theorem}{Theorem}[section]
\newtheorem{thm}{Theorem}[section]
\newtheorem{prop}[theorem]{Proposition}
\newtheorem{lem}[theorem]{Lemma}

\newtheorem{defn}[theorem]{Definition}
\newtheorem{rem}[theorem]{Remark}

\newtheorem{cor}[theorem]{Corollary}

\newtheorem{prob}[theorem]{Problem}
\newtheorem{claim}[theorem]{Claim}

\newtheorem{mainthm}{Theorem}[section] 
\newaliascnt{myMainThm}{mainthm}        
\newtheorem*{myMainThm*}{Theorem \Alph{mainthm}} 

\newenvironment{maintheorem}[1][]
{\begin{mainthm}[#1]}
	{\end{mainthm}}

\theoremstyle{definition}
\theoremstyle{remark}

\numberwithin{equation}{section}

\newcommand {\diam }{\mathrm{diam}}
\newcommand {\Isom }{\mathrm{Isom}}

\newcommand {\rank}{\mathrm{rank}}

\newcommand{\leb}{\mathcal{L}}

\newcommand {\supp }{\mathrm{supp}}

\newcommand {\id}{\mathrm{id}}

\newcommand {\Z}{\mathbb{Z}}

\newcommand {\R}{\mathbb{R}}

\newcommand {\GH}{\xrightarrow{\mathrm{GH}}}

\newcommand {\pGH}{\xrightarrow{\mathrm{pGH}}}

\newcommand {\N}{\mathbb{N}}

\newcommand {\Rm}{\mathrm{Rm}}
\newcommand {\inj}{\mathrm{inj}}
\newcommand {\Lip}{\mathrm{Lip}}

\newcommand {\Ric}{\mathrm{Ric}}

\newcommand {\RCD}{\mathrm{RCD}}

\newcommand{\meas}{\mathfrak{m}}
\newcommand{\defeq}{\mathrel{\mathop:}=}

\newcommand{\mb}{\mathbb}
\newcommand{\mc}{\mathcal}

\newcommand\tbbint{{-\mkern -16mu\int}}

\newcommand\dbbint{{-\mkern -19mu\int}}

\newcommand\bbint{
	{\mathchoice{\dbbint}{\tbbint}{\tbbint}{\tbbint}}
}

\newcommand{\XXint}[3]{{
		\setbox0=\hbox{$#1{#2#3}{\int}$}
		\vcenter{\hbox{$#2#3$}}\kern-.5\wd0}}

\begin{document}

\maketitle
\begin{abstract}

	In this paper, we prove fibration theorems for manifolds with almost nonnegative Ricci curvature and certain extra regularity assumptions. We show that a closed $n$-manifold $M$ satisfying
	\begin{align*}
	\diam(M)^2\sec_M \geq -\kappa \quad \text{and} \quad \diam(M)^2\Ric_M \geq -\delta,
	\end{align*}
	where $\delta>0$ is sufficiently small depending only on $n$ and $\kappa$, fibers over a $b_1(M)$-torus. This removes the upper sectional curvature bound required in the earlier result of Yamaguchi \cite{Y88}. As a corollary, we obtain a refinement of Yamaguchi's smooth fibration theorem (\cite{Y91}), showing that the fiber itself (rather than a finite cover of it) fibers over a $b_1$-torus. Our results extend to manifolds satisfying a generalized Reifenberg condition introduced in \cite{HH24}, which encompasses both a lower bound on sectional curvature and the local rewinding Reifenberg condition. In the nonsmooth setting, a similar result also holds for a non-collapsed $\mathrm{RCD}(-\epsilon(D,r,n),n)$ space whose diameter is bounded by $D$ and which satisfies the $(r,\delta(n))$-local rewinding Reifenberg condition. The proofs rely on an equivariant regularity theorem for almost submetries under a lower Ricci curvature bound.
    In addition, we study the stability of rank of Abelian actions along equivariant Gromov-Hausdorff convergence in this paper.


	\vspace*{10pt}
	\noindent {\it 2010 Mathematics Subject Classification}: 	53C21, 53C23, 53B21.
	

\end{abstract}
\tableofcontents

\section{Introduction}  

A closed manifold $M$ is said to have \emph{almost nonnegative Ricci curvature} if there exists a small constant $\delta(n) > 0$, depending only on the dimension $n$ of $M$, such that its diameter and the minimal eigenvalue of its Ricci tensor satisfy
\begin{align}\label{pinching-ricci}
	\diam(M)^2 \Ric_M \geq -\delta(n).
\end{align}
Gromov~\cite{Gr81-II} (see also Gallot~\cite{Ga83}) proved that if $M$ satisfies~\eqref{pinching-ricci} for a sufficiently small $\delta(n)$, then the first Betti number $b_1(M) \leq n$.
Furthermore, Gromov conjectured that the equality $b_1(M) = n$ implies that $M$ is diffeomorphic to the flat $n$-torus, a conjecture that was finally resolved by Colding~\cite{Col97}.

A natural question arises: under the pinching condition \eqref{pinching-ricci}, what specific geometric structure is guaranteed when $b_1(M) < n$? Given an upper bound on sectional curvature, Yamaguchi proved the following fibration theorem.

\begin{theorem}[\cite{Y88}]\label{thm-pinching-Yamaguchi-1}
	For any $n\in\mathbb{Z}^{+}$ and $\kappa>0$, there exists $\delta(n,\kappa)>0$ such that the following holds. Let $M$ be an $n$-dimensional closed Riemannian manifold satisfying \eqref{pinching-ricci} and
	\begin{align*}
	\diam(M)^2\,\sec_M \le \kappa.
	\end{align*}
	Then $M$ fibers over a $b_{1}(M)$-dimensional torus.
\end{theorem}

Yamaguchi (\cite{Y88}) further conjectured that the conclusion remains true even without the upper bound on sectional curvature. However, Anderson (\cite{An92-I}) constructed a counterexample as follows: for $n \geq 4$, there exists a sequence of closed $n$-manifolds $\{M_i\}$ such that $|\Ric_{M_i}| \to 0$, $b_1(M_i) = k$ (where $k\leq n-1$), and $M_i$ converges in the Gromov--Hausdorff sense to a flat torus $T^k$, but no cover of $M_i$ fibers over $S^1$.

When the pinching condition \eqref{pinching-ricci} is strengthened to the case of almost nonnegative sectional curvature, Yamaguchi established the following result.

\begin{theorem}[\cite{Y91}]\label{thm-pinching-Yamaguchi}
	There exists a constant $\delta(n) > 0$, depending only on $n$, such that any compact $n$-manifold $M$ satisfying
	\begin{align}\label{pinching-1.2}
		\diam(M)^2 \sec_M \geq -\delta(n)
	\end{align}
	has the following properties:
	\begin{enumerate}
		\item[(a)] A finite cover of $M$ fibers over a torus of dimension $b_1(M)$;
		\item[(b)] If $b_1(M) = n$, then $M$ is diffeomorphic to the $n$-torus.
	\end{enumerate}
\end{theorem}
In fact, this theorem is a special case of Yamaguchi's general smooth fibration theorem in \cite{Y91}. The general theorem states that if a manifold with uniformly lower bound of sectional curvature collapses in the Gromov-Hausdorff sense to a lower-dimensional smooth manifold, then the collapsing yields a smooth fibration whose fiber satisfies the topological properties described above.

Note that under the hypotheses of Theorem \ref{thm-pinching-Yamaguchi-1}, the manifold $M$ admits a uniform lower bound on its sectional curvature, which depends only on $n$ and $\kappa$. In \cite{Y88}, this two-sided bound on sectional curvature is crucial for establishing regularity estimates for exponential maps and for carrying out a certain smoothing procedure of Riemannian metrics.

Recent advances in metric geometry under lower Ricci curvature bounds allow us to remove the upper sectional curvature assumption in Theorem~\ref{thm-pinching-Yamaguchi-1} and to avoid passing to a finite cover in Theorem~\ref{thm-pinching-Yamaguchi}. A main result of this paper unifies Theorems~\ref{thm-pinching-Yamaguchi-1} and~\ref{thm-pinching-Yamaguchi} as follows.

\begin{maintheorem}\label{mainthm-mix-curv-pinching}
	For any $n\in\mathbb{Z}^{+}$, $\kappa>0$, there exists $\delta(n,\kappa)>0$, such that the following holds. Let $M$ be an $n$-dimensional closed Riemannian manifold satisfying (\ref{pinching-ricci}), and
	\begin{equation}\label{sec}
			\diam(M)^2\,\sec_M \ge-\kappa.
	\end{equation}
Then $M$ fibers over a $b_{1}(M)$-torus with connected fibers.
\end{maintheorem}

By \cite{KW}, manifolds satisfying (\ref{pinching-ricci}) have virtually nilpotent fundamental groups.
The following is an immediate corollary of the above theorem.

\begin{cor}
Suppose $M$ satisfies assumptions of Theorem \ref{mainthm-mix-curv-pinching}. If $\pi_1(M)$ is infinite, then the Euler characteristic number $M$ is $0$.
\end{cor}

Furthermore, we can adapt Yamaguchi's smooth fibration theorem (\cite{Y91}) as follows.

\begin{theorem}\label{mainthm-fiber-sec}
	Given $n \geq 2$ and $r_0 > 0$, there exists a $\delta_0$ depending only on $n$, $r_{0}$  such that the following holds for every $\delta\in (0,\delta_0)$.
    Suppose $M$ and $N$ are compact Riemannian manifolds of dimensions $n$ and $k$ ($k \leq n$) respectively, and
    \begin{align}
    \sec_{M}\geq -1,
    \end{align}
	\begin{align}\label{5.2-001}
    \Ric_N \geq -1, \quad and \quad d_{\mathrm{GH}}(B_{r_{0}}(q),B_{r_{0}}(0^{k})) < \delta \quad \text{ for any } q\in N,
    \end{align}
		\begin{align}
	 d_{\mathrm{GH}}(M,N) <\delta.
	\end{align}
	Then there exists a smooth fibration $f \colon M \to N$ which is a $\Psi(\delta|n,r_0)$-Gromov-Hausdorff approximation and the fibers $F$ satisfy:
	\begin{enumerate}
		\item $b_1(F) \leq n - k$;
		
		\item $F$ fibers over the torus $T^{b_1(F)}$ with connected fibers.
	\end{enumerate}
\end{theorem}
Compared to the original form in \cite{Y91}, the conclusion (2) here no longer requires passing to a finite cover. In the rigidity case $b_1(F) = n - k$, it follows immediately from (2) that $F$ is diffeomorphic to a torus.
In addition, compared to \cite{Y91}, the assumption on the lower dimensional manifold $N$ is weaken to (\ref{5.2-001}).

\begin{rem}
	Suppose we have a fibration $p:M\rightarrow T^{k}$ with connected fiber $F$, then we have a long exact sequence
	\begin{align}
		\cdots\rightarrow \pi_{n}(F)\rightarrow \pi_{n}(M)\rightarrow \pi_{n}(T^{k})\rightarrow\pi_{n-1}(F)\rightarrow\cdots \rightarrow\pi_{0}(E)\rightarrow0.
	\end{align}
	In particular, we have
	\begin{align}
		0=\pi_{2}(T^{k})\rightarrow \pi_{1}(F)\rightarrow \pi_{1}(M)\rightarrow \pi_{1}(T^{k})\rightarrow\pi_{0}(F)=0.
	\end{align}
	Thus $b_{1}(M)\geq b_{1}(T^{k})=k$.
	As a conclusion, $M$ in Theorem \ref{mainthm-mix-curv-pinching} (or $F$ in Theorem \ref{mainthm-fiber-sec}) cannot fiber over a $T^{b_1(M)+1}$ ($T^{b_1(F)+1}$ respectively).
\end{rem}

\begin{rem}
Indeed, the assumption of a lower bound on sectional curvature in Theorems \ref{mainthm-mix-curv-pinching} and \ref{mainthm-fiber-sec} can be weakened to a generalized Reifenberg condition, which is introduced in \cite{HH24}. This condition encompasses both the case of sectional curvature bounded below and that of manifolds with local rewinding Reifenberg condition (see \cite{HH24}*{Proposition 5.6}).
    See Remark \ref{rem5.4} for a discussion.
\end{rem}

In the non-smooth setting, Yamaguchi \cites{Y96} conjectures that a fibration theorem similar to \cite{Y91} still holds when $M$ is a compact Alexandrov space, see \cite{Fujio21} etc.  for recent development in this direction. Consequently, we propose the following problem.

\begin{prob}\label{conj-nonsmooth}
	Can one relax the smoothness assumption on $M$ in Theorems \ref{mainthm-mix-curv-pinching} or \ref{mainthm-fiber-sec}, and replace the corresponding curvature hypotheses by a suitable synthetic condition?
\end{prob}

Now we turn to fibration theorems in the $\RCD$ setting related to first Betti number. There are several results in the maximal first Betti number case, see \cites{MMP22,HKMPR,ZamoraZhu}. The reader can also refer to \cite{Wang2024} for some other extensions of fibration theorems to $\RCD$ spaces.

Our next result also contributes to Problem \ref{conj-nonsmooth}. Recall that a non-collapsed $\mathrm{RCD}(K,n)$ space $(X,d,\mathcal{H}^n)$ satisfies the $(r,\delta)$-local rewinding Reifenberg condition if for every $x\in X$, the universal cover $(\widetilde {B_r(x)},\tilde x)\to (B_r(x),x)$ satisfies
\begin{align*}
d_{\mathrm{GH}}\bigl(B_t(\tilde x),\,B_t(0^n)\bigr)\le t\delta,
\end{align*}
for all $0<t\le r/3$, where $B_t(0^n)\subset\mathbb{R}^n$.
\begin{rem}
	Thanks to the work of Cheeger-Fukaya-Gromov \cite{CFG}, any $n$-manifold with $|\sec| \le 1$ satisfies the $(r,\delta)$-local rewinding Reifenberg condition for suitable $r,\delta$. Philosophically, this condition can therefore be viewed as a synthetic analogue of bounded sectional curvature. However, it is important to note that, logically, it does not imply a lower curvature bound in the Alexandrov sense.
\end{rem}
Our next result is as follows.

\begin{maintheorem}\label{mainthm_B}
	Let $n\in\mathbb{N}$ and $r,D>0$. There exist constants $\delta:=\delta(n)$ and $\epsilon:=\epsilon(D,r,n)$ such that if $(X,d,\mathcal{H}^n)$ is an $\mathrm{RCD}(-\epsilon,n)$ space with $\operatorname{diam}(X)\le D$ and $X$ satisfies the $(r,\delta)$-local rewinding Reifenberg condition, then $X$ is homeomorphic to a fiber bundle over $T^{b_1(X)}$.
\end{maintheorem}

Another contribution in this paper is the following observation on the stability of rank of Abelian actions along equivariant Gromov--Hausdorff convergence.

\begin{theorem}\label{thm:rank-stability}
Let $k\in \mathbb{N}$, $N \in [1,\infty)$ with $k \le N$, and $D>0$.
Assume that $\{(X_i,p_i,d_i,\mathfrak{m}_i)\}$ is a sequence of $\mathrm{RCD}(-1,N)$ spaces satisfying
\begin{align*}
b_1(X_i) = k \quad \text{and} \quad \operatorname{diam}(X_i) \le D.
\end{align*}
Let $(\hat{X}_i, \hat{p}_i, \hat{d}_i, H_i)$ be the normal covering space of $(X_i,p_i,d_i)$, where the transformation group is the first homology group
\begin{align*}
H_i = \pi_1(X_i,p_i) \big/ [\pi_1(X_i,p_i), \pi_1(X_i,p_i)].
\end{align*}

If
\begin{align*}
(\hat{X}_i, \hat{p}_i, \hat{d}_i, H_i) \xrightarrow{\mathrm{GH}}
(Y, \hat{p}, d, H = \mathbb{R}^{k_1} \times \mathbb{Z}^{k_2} \times C),
\end{align*}
where $C$ is a compact group, then
\begin{align*}
k = k_1 + k_2.
\end{align*}
\end{theorem}

\begin{rem}\label{rem:rank-stability}
The same conclusion of Theorem  \ref{thm:rank-stability} holds if the condition $$H_i = \pi_1(X_i,p_i) \big/ [\pi_1(X_i,p_i), \pi_1(X_i,p_i)]$$ is replaced by $$H_i = (\pi_1(X_i,p_i) \big/ [\pi_1(X_i,p_i), \pi_1(X_i,p_i)])/T_i,$$ where $T_i$ is the torsion subgroup of $\pi_1(X_i,p_i) \big/ [\pi_1(X_i,p_i), \pi_1(X_i,p_i)]$. See Corollary \ref{cor: compact torsion}.
\end{rem}
One application of Theorem \ref{thm:rank-stability} is to give a new proof of the fact that, under pointed measured Gromov--Hausdorff convergence of $\mathrm{RCD}$ spaces, the first Betti number cannot drop more than the rectifiable dimension does-a result established by Zamora \cite{Zamora2022} and Zamora-Santos-Rodriguez \cite{SantosZamoraPi1}. The difference in our argument is that we do not need to analyze the vanishing of small loops in the limit.

In the following we outlines the main ideas and techniques in the proof of Theorems \ref{mainthm-mix-curv-pinching} and \ref{mainthm_B}.

A technical result in the proofs of Theorems \ref{mainthm-mix-curv-pinching} and \ref{mainthm_B} is an equivariant, ``regular'' almost submetry-a tool that may admit further applications.

We call a map (not necessarily continuous) $f \colon (X,d_X) \to (Y,d_Y)$ a $\delta$-\emph{almost submetry} at scale $r_0$ if for every $p \in X$ and every $r \in(0, r_0)$,
\begin{align*}
f\big(B_{r-\delta}(p)\big) \subset B_{r}(f(p)) \subset T_{\delta}\!\big(f(B_{r+\delta}(p))\big),
\end{align*}
where for a set $A\subset Y$, $T_{r}(A)=\{y\in Y\mid d(y,a)< r,\ a\in A\}$ denotes its open $r$-neighborhood.
It is elementary that if a map $f \colon (X,d_X) \to (Y,d_Y)$ is a $\delta$-almost submetry, then it is $\delta$-almost onto.
In addition, if a map $f \colon (X,d_X) \to (Y,d_Y)$ is a $\frac{\delta}{2}$-Gromov-Hausdorff approximation, then it is a $\delta$-almost submetry.

With the above notion, the following technical result can be regarded as an effective, equivalent ``fibration''.

\begin{theorem}\label{thm:eq-submetry}
	Given $k \leq n$, $\epsilon > 0$, and $(N,h)$ a compact $k$-dimensional Riemannian manifold. Let $H$ be a closed subgroup of  $\Isom(N)$. Then there exists $\delta = \delta(n, N, H, \epsilon) > 0$ such that the following holds.
	
	Suppose $(X,d,\mathcal{H}^n)$ is a compact $\mathrm{RCD}(-1,n)$-space and $G$ is a finite subgroup of $\Isom(X)$. Assume there exist a map $f \colon X \to N$ and a map $\phi \colon G \to H$ such that:
	\begin{enumerate}
		\item[(i)] $f$ is a $\delta$-almost submetry at scale $1$.
		\item[(ii)] For all $\gamma \in G$ and all $x \in M$,
		\begin{align} \label{1.34341}
			d_N\big(f(\gamma(x)),\, \phi(\gamma)(f(x))\big) \leq \delta.
		\end{align}
	\end{enumerate}
	Then there exist a homomorphism $\psi \colon G \to H$ and a Lipschitz map $F \colon X \to N$ such that:
	\begin{enumerate}
		\item[(a)] $\sup_{x \in X} d_N(F(x), f(x)) < \epsilon$;
		\item[(b)] $F \circ \gamma = \psi(\gamma) \circ F$ for every $\gamma \in G$;
		\item[(c)] For any $x\in X$, there exists a normal coordinate chart $\Theta_{y}:B_{10\sqrt{\delta}}(y)\rightarrow V \subset \mathbb{R}^{k}$ for $y=F(x)$,
		such that $\Theta_{y}$ is an $\epsilon$-Gromov-Hausdorff approximation onto its image, and $\bar{F}=\Theta_{y}\circ F|_{B_{5\sqrt{\delta}}(x)}$ satisfies
		\begin{itemize}
			\item [(c1)] $\Lip \bar{F}\leq C(n)$,
			\item [(c2)]$|\Delta \bar{F}|_{L^{\infty}(B_{5\sqrt{\delta}}(x))}\leq \frac{C(n)}{\sqrt{\delta}}$,
			\item [(c3)] $\bbint_{B_{5\sqrt{\delta}}(x)}|\langle\nabla \bar{F}^{i}, \nabla \bar{F}^{j} \rangle- \delta_{ij}|\leq \epsilon$.
		\end{itemize}

	\end{enumerate}
\end{theorem}

\begin{rem}
	If we assume that $X$ is smooth and satisfies a generalized Reifenberg condition in Theorem \ref{thm:eq-submetry}, then according to the non-degeneracy theorem in \cite{HH24} and Ehresmann's lemma, the resulting map $F$ is a smooth fiber bundle. Hence, Theorem \ref{thm:eq-submetry} is an effective version of \cite{HuangHzh}*{Corollary 1.9}.
    In Section \ref{sec-6}, we will prove a fiber version of Theorem \ref{thm:eq-submetry} (see Proposition \ref{prop-eq-fibration}), which will be used in the proof of Theorem \ref{mainthm-fiber-sec}.
\end{rem}

The proof of Theorem \ref{mainthm-mix-curv-pinching} is by contradiction.
Suppose we have a contradicting sequence $M_i$ with $\mathrm{diam}(M_i)=1$ and $b_{1}(M_i)=b$. We consider the corresponding sequence of normal covering $P_i:\hat{M}_i\rightarrow M_{i}$ whose deck transformation group $H_i$ is the torsion-free part of the homology group of $M_i$.
We consider the equivariant Gromov-Hausdorff convergence $(\hat{M}_i,H_i)\GH (\R^s \times \hat{Y},H)$, where $\hat{Y}$ is a compact metric space, and $H$ is Abelian.
According to the upper semi-continuity of the rank of Abelian group action (Lemma \ref{lem:rank non-dec}), we have $s \geq b$.
It is easy to see that for any discrete Abelian cocompact isometric group action $H_0$ on $\R^s \times \hat{Y}$, $\R^s \times \hat{Y}/H_0$ admits an submetry over some $s$-dimensional torus.
Since $H$ may contain some nontrivial $\R$-factor, to ensure a submetry a torus of $s$-dimension, we take a discrete cocompact subgroup $H_0$ of $H$.
According to a technical lemma on local Abelian group action (see Lemma \ref{LocalAction}), we known that there exists a sequence of finite covers $\tau_i:\check{M}_i\rightarrow M_i$ with deck transformation group $K_i$, such that $\check{M}_i$ convergence to $\R^s \times \hat{Y}/H_0$.
In particular, each $\check{M}_i$ admits an almost-submetry $f_i:\check{M}_i\rightarrow T^s$ which almost communicates with the group action $K_i$ (i.e. (\ref{1.34341}) holds).
Thus we can apply Theorem \ref{thm:eq-submetry} to modify $f_i$ into an almost-submetry $F_i$ which communicates with the group action $K_i$ (i.e. (b) in Theorem \ref{thm:eq-submetry} holds).
By the smoothness of $\check{M}_i$ and the uniform lower bound of sectional curvature, each $F_i$ is a fibration over $s$-torus. The quotient map of $F_i$ gives a fibration $F'_i:M_i\rightarrow T^s$, and we get a contradiction.

The proof of Theorem~\ref{mainthm_B} is a minor modification of Theorem~\ref{mainthm-mix-curv-pinching}. In fact, due to the lack of a smooth structure, we cannot immediately conclude the non-degeneratcy of $dF_i$.
In contrast, by the transformation theorems as well as some topological results, we can conclude that $F_{i}$ is a topological fiber bundle map, see Theorem~\ref{thm:bundlemap}.

In the following, we give an outline of this paper. Section \ref{sec2} contains some notions and background results that are used in this paper. In Section \ref{sec-3} we prove a technical lemma on local Abelian group action.
Assuming Proposition \ref{prop-eq-fibration}, which ensures the existence of sufficiently regular equivariant maps, we will prove the pinching theorems \ref{mainthm-mix-curv-pinching} and \ref{mainthm_B} as well as the smooth fibration theorem \ref{mainthm-fiber-sec} in Sections \ref{sec-13} and \ref{sec-4}.
In Section \ref{sec-5}, we give two methods to prove Theorem \ref{thm:rank-stability}, and give a new proof of a theorem  established in \cite{Zamora2022} and \cite{SantosZamoraPi1} by applying Theorem \ref{thm:rank-stability}.
In Section \ref{sec-6}, we use the technique of center of mass to prove various theorems about the existence of sufficiently regular maps, including Proposition \ref{prop-eq-fibration}.
Finally, in Appendix \ref{sec-8}, we collect some facts on normal coordinates, which are used in Section \ref{sec-6}.

\vspace*{20pt}

\noindent\textbf{Acknowledgments.}

The authors would like to thank Profs. B.-L. Chen, X. Rong, H.-C. Zhang and X.-P. Zhu for constant encouragements.
The authors would like to thank Prof T. Yamaguchi for comments on this paper.
H. Huang is partially supported by National Natural Science Foundation of China 12571052 and Guang-dong Natural Science Foundation 2026A1515012296.
X.-T. Huang is partially supported by National Natural Science Foundation of China (Nos. 12271531, 12426202).
X.Z is supported by AMS-Simons travel grant.

\section{Preliminaries}\label{sec2}

In this paper, we denote by $\Psi(a_1,\ldots, a_l|b_1,\ldots, b_m)$ the nonnegative error function such that $\Psi(a_1,\ldots, a_l|b_1,\ldots, b_m)\to 0$ when $a_1,\ldots, a_l\to 0$ and $b_1,\ldots,b_m$ fixed. The $\Psi$ may change from line to line.



\subsection{Equivariant Gromov--Hausdorff convergence and isometry groups}

We review the notion of equivariant Gromov--Hausdorff convergence introduced by Fukaya and Yamaguchi \cites{Fukaya1986,FukayaYamaguchi1992}.

Let $(X,p)$ and $(Y,q)$ be two pointed metric spaces. Let $H$ and $K$ be closed subgroups of $\mathrm{Isom}(X)$ and $\mathrm{Isom}(Y)$, respectively. For any $r>0$, define the sets
$$H(r)=\{h \in H | d(hp,p) < r \}, \ K(r)=\{k \in K | d(kq,q) < r \}.$$

For $\epsilon>0$, a pointed $\epsilon$-equivariant Gromov--Hausdorff approximation (or simply an $\epsilon$-eGHA) is a triple of maps $(f,\phi,\psi)$, where
$$f:B_{\frac{1}{\epsilon}}(p) \to B_{\frac{1}{\epsilon}+ \epsilon}(q),\quad \phi:H\left(\frac{1}{\epsilon}\right) \to K\left(\frac{1}{\epsilon}\right),\quad \psi:K\left(\frac{1}{\epsilon}\right) \to H\left(\frac{1}{\epsilon}\right)$$
satisfy the following conditions:
\begin{enumerate}
    \item\label{item:AlmostIsom} $f(p)=q$, $f(B_{\frac{1}{\epsilon}}(p))$ is $2\epsilon$-dense in $B_{\frac{1}{\epsilon} + \epsilon}(q)$ and $|d(f(x_1),f(x_2))-d(x_1,x_2)|\le\epsilon$ for all $x_1,x_2\in B_{\frac{1}{\epsilon}}(p)$;
    \item $d(\phi(h)f(x),f(hx)) < \epsilon$ for all $h \in H(\frac{1}{\epsilon})$ and $x \in B_{\frac{1}{\epsilon}}(p)$;
    \item $d(kf(x),f(\psi(k)x))<\epsilon$ for all $k \in K(\frac{1}{\epsilon})$ and $x \in B_{\frac{1}{\epsilon}}(p)$.
\end{enumerate}

The equivariant Gromov--Hausdorff (eGH in short) distance $d_{GH}((X_i,p_i,G_i),(X,p,G))$ is defined as the infimum of $\epsilon$ so that there exists an $\epsilon$-eGHA.
We say a sequence of metric spaces with isometric actions $(X_i,p_i,G_i)$ converges to a limit space $(X,p,G)$, if $d_{\mathrm{GH}}((X_i,p_i,G_i),(X,p,G))\to 0$. We denote the convergence by $(X_i,p_i,G_i)\xrightarrow{\mathrm{GH}}(X,p,G)$ for short.


We have the following pre-compactness theorem for equivariant Gromov--Hausdorff convergence (see \cites{Fukaya1986,FukayaYamaguchi1992}).
\begin{theorem}[\cites{Fukaya1986,FukayaYamaguchi1992}] \label{eGH}
	Let $(X_i,p_i)$ be a sequence of metric spaces converging to a limit space $(X,p)$ in the pointed Gromov--Hausdorff sense. For each $i$, let $G_i$ be a closed subgroup of $\mathrm{Isom}(X_i)$. Then passing to a subsequence if necessary,
	$$(X_i,p_i,G_i)\xrightarrow{\mathrm{GH}} (X,p,G),$$
	where $G$ is a closed subgroup of $\mathrm{Isom}(X)$. Moreover, the quotient spaces $(X_i/G_i, \bar{p}_i)$ pointed Gromov--Hausdorff converge to $(X/G,p)$.
\end{theorem}

\begin{defn}\label{defn-k-euc}
Given a metric space $X$, we say $B_{r}(p)\subset X$ is $(\delta,k)$-Euclidean, if there exists a metric space $(Z,d_{Z})$ such that
\begin{align}
d_{\mathrm{GH}}(B_{r}(p),B_{r}((0^{k},z))\leq\delta r,
\end{align}
where $(0^{k},z)\in \mathbb{R}^{k}\times Z$.
\end{defn}

\subsection{$\RCD (K,N)$ spaces}\label{sec:RCDgeo}

$\textmd{RCD}(K,N)$ spaces, where $K\in \R$ and $N\in [1,\infty)$, are metric measure spaces $(X,d,m)$ with generalized Ricci curvature bounded from below by $K$ and dimension bounded from above by $N$, see \cites{AGMR15,EKS15,Gig13} etc.
$\textmd{RCD}(K,N)$ spaces includes all $n$-manifolds with $\Ric \geq K, n\leq N$ and their Ricci limit spaces.
If $N\in \mathbb{Z}^{+}$, and $m=\mathcal{H}^{N}$, then we say $(X,d,m)$ is a non-collapsed $\textmd{RCD}(K,N)$ space, see \cite{DePGig18} etc.
We assume that the reader is familiar with the basics of $\RCD(K,N)$ spaces.

Throughout this paper, when we say a group $G$ acting on a $\RCD$ space $(X,d,m)$, we require that each element of $G$ is a measure-preserving isometry of $(X,d,m)$.
A fact that we will frequently use is that the measure-preserving isometry group of an $\RCD(K,N)$ space is a Lie group (see \cites{Sosa2018,GuSan2019}).



\begin{defn}
Let $(X,d,\meas)$ be an $\RCD(K,N)$ space. Given $k \in \mathbb{N}$, we define the $k$-dimensional regular set as the set of points with unique tangent cone $\R^k$. More precisely,
\begin{align*}
\mc R^k(X)\defeq\left\{x\in X : (X,x,r^{-1}d,\meas(B_r(x))^{-1}\meas)\overset{\mathrm{GH}}\longrightarrow(\R^k,0^k, |\cdot|, \leb^k), \ \forall r\to 0^+.\right\}
\end{align*}
 There exists $n \le N$ such that $\meas(X \setminus \mc{R}^n(X))=0$ by Bru\'e--Semola \cite{BS_EssentialDim}. We call this $n$ the rectifiable dimension of $(X,d,\meas)$.
\end{defn}


The notion of $\delta$-splitting maps on manifolds or Ricci-limit spaces, which originates from Cheeger and Colding's works in \cites{CC96,CC97} etc., was introduced in \cite{CN}.
The notion of $\delta$-splitting maps can be generalized to $\RCD$ spaces, see e.g. \cite{BPS19} \cite{BNS}.

\begin{defn}\label{def-harm-split}
Let $(X,d,m)$ be an $\mathrm{RCD}(-1,N)$ space, $p\in X$ and $\delta > 0$ be fixed.
We say that $u:=(u_{1},\ldots,u_{k}) : B_{r}(p)\to \mathbb{R}^{k}$ is a $(\delta,k)$-splitting map if it belongs to the domain of the local Laplacian on $B_{r}(p)$, and
\begin{description}
  \item[(i)] $\mathrm{Lip} u \leq  C(N)$;
  \item[(ii)] $|\Delta u^a|_{L^{\infty}(B_{r}(p))}<\delta$;
  \item[(iii)] $\bbint_{B_{r}(p)}|\langle\nabla u^{a},\nabla u^{b}\rangle-\delta_{ab}|dm <\delta$ for any $a,b\in\{1,\ldots,k\}$.
\end{description}
If  (ii) is strengthened to $\Delta u=0$, then we say that $u$ is a $(\delta,k)$-splitting harmonic map.
We use $u : (B_{r}(p),p)\to (\mathbb{R}^{k},0^{k})$ to denote a map $u : B_{r}(p)\to \mathbb{R}^{k}$ with $u(p)=0^{k}\in \R^{k}$.
\end{defn}





The existence of an almost splitting function is equivalent to pointed Gromov-Hausdorff  closeness to a space that splits off an Euclidean factor $\R$, see \cites{BNS,CC97,CheegerColding2000a,CJN21}.

\begin{theorem}\label{thm-split-GHisom}
Let $1\leq N<\infty$ be fixed.
For every positive number $\delta \ll 1$, there exists $\epsilon(N,\delta)>0$ such that any $\epsilon\in(0,\epsilon(N,\delta))$ satisfies the followings.
If $(X,d,m)$ is an $\mathrm{RCD}(-\epsilon,N)$ space, $p\in X$, then
\begin{description}
  \item[(1)] if $d_{\mathrm{mGH}}(B_{4}(p),B_{4}^{\mathbb{R}^{k}\times Z}(0^{k},z))\leq \epsilon$ for some integer $k$ and some pointed metric measure space $(Z,z,d_{Z},m_{Z})$, then there exists a $(\delta,k)$-splitting map $u=(u^{1},\ldots,u^{k}) : B_{1}(p)\rightarrow\mathbb{R}^{k}$;
  \item[(2)] if there exists an $(\epsilon,k)$-splitting map $u : B_{4}(p)\rightarrow\mathbb{R}^{k}$ for some integer $k$, then $d_{\mathrm{mGH}}(B_{1}(p),B_{1}^{\mathbb{R}^{k}\times Z}(0^{k},z))\leq \delta$ for some pointed metric measure space $(Z,z,d_{Z},m_{Z})$; moreover, there exists $f:B_{1}(p)\rightarrow Z$ such that $(u-u(p),f):B_{1}(p)\rightarrow B_{1}^{\mathbb{R}^{k}\times Z}(0^{k},z)$ is a $\delta$-Gromov-Hausdorff isometry.
\end{description}
\end{theorem}



\subsection{Transformation theorems and the generalized Reifenberg condition}


Since \cites{CN,CJN21}, there have been various versions of transformation theorems, see e.g. \cites{BNS22,WZh21,HondaPeng2023,HH24,HuangHzh22}.
The reader can refer to these papers for many interesting applications of the transformation theorems.

The following transformation theorem follows from \cite{HuangHzh22}*{Theorem 4.1}.
Even though in the statement of \cite{HuangHzh22}*{Theorem 4.1}, it holds for manifolds, but its proof works on $\RCD$-space.
Note that we do not require any non-collapsing assumption in Theorem \ref{thmA.1}. See also \cite{HH24}*{Theorem 1.7}.

\begin{theorem}\label{thmA.1}
Given $\epsilon,\eta> 0$, $n\in \mathbb{Z}^{+}$ and $L\geq 0$,
there exits $\delta_{0}=\delta_{0}(n,\epsilon,\eta,L)$, so that for any $\delta\in(0,\delta_{0})$ the following holds.
Suppose $(X, d,m)$ is an $\RCD(-\delta,n)$-space and $\overline{B_{2}(p)}$ is compact in $X$, and there exits $s\in(0, 1)$ such that, for any $r\in [s, 1]$, $B_{r}(p)$ is $(\delta,K)$-Euclidean but not $(\eta,K+1)$-Euclidean.
Let $u:B_{2}(p)\rightarrow \R^k$ be a Lipschitz map which belonging to the
domain of the local Laplacian on $B_2(p)$, where $1\leq k\leq K\leq n$, and suppose
\begin{align}\label{2.11-1}
|\Delta u^{\alpha}|_{L^{\infty}(B_{2}(p))}\leq L,
\end{align}
\begin{align}\label{2.11-2}
\bbint_{B_{2}(p)}|\langle \nabla u^{\alpha}, \nabla u^{\beta}\rangle-\delta_{\alpha\beta}|\leq \delta
\end{align}
for $1\leq \alpha,\beta\leq k$.
Then for every $r\in [s,1]$, there exists a $k\times k$ lower triangle matrix $T_{r}$ with positive diagonal entries
so that the followings hold:
\begin{description}
  \item[(1)] $\bbint_{B_{r}(p)}\langle\nabla(T_{r}u)^{\alpha},\nabla(T_{r}u)^{\beta}\rangle=\delta^{\alpha\beta}$;
  \item[(2)] $\bbint_{B_{r}(p)}|\langle\nabla(T_{r}u)^{\alpha},\nabla(T_{r}u)^{\beta}\rangle-\delta^{\alpha\beta}|\leq \epsilon$;
  \item[(3)] $|\Delta (T_{r}u)^{\alpha}|_{L^{\infty}(B_{1}(p))}\leq C(n,L)r^{-\epsilon}$;
  \item[(4)] $|\nabla (T_{r}u)^{\alpha}|_{L^{\infty}(B_{r}(p))}\leq C(n,L)$;
  \item[(5)] for $s\leq r\leq t\leq 1$, $|T_{r}T_{t}^{-1}|+|T_{t}T_{r}^{-1}|\leq 2\bigl(\frac{t}{r}\bigr)^{\epsilon}$, here $|\cdot|$ means the $L^{\infty}$-norm of a matrix.
\end{description}

\end{theorem}

The following definition of the generalized Reifenberg condition is equivalent to the one given in \cite{HH24}*{Definition 1.4}.

\begin{defn}\label{defn:GRC}
Given a metric space $(X, d)$ and $p\in X$, if there exist $\delta\geq 0$, $k\in \mathbb{Z}^{+}$, $r_{0}>0$, and a function $\Phi:\R^+\to \R^+$ with $\lim_{\delta'\to0}\Phi(\delta')=0$, so that if some geodesic ball $B_{r}(p)$ is $(\delta',k')$-Euclidean, where $r\leq r_{0}$, $k'\geq k$, $\delta'\geq \delta$, then for every $s\leq r$, $B_{s}(p)$ is $(\Phi(\delta'),k')$-Euclidean, then we say $p$ satisfies the {$(\Phi, r_{0}; k,\delta)$-generalized Reifenberg condition}.

If for every $p\in U\subset X$, $p$ satisfies the {$(\Phi, r_{0}; k,\delta)$-generalized Reifenberg condition}, then we say $U$ satisfies the {$(\Phi,r_0; k,\delta)$-generalized Reifenberg condition}.

We will use the $(k,\delta)$-generalized Reifenberg condition if there is no ambiguity on $\Phi$ and $r_0$.
\end{defn}

The following theorem generalizes \cite{HH24}*{Theorem 4.6}, and their proofs are similar.

\begin{theorem}\label{thm-trans-GRC-111}
Given $\epsilon>0$, $n, k\in \mathbb{Z}^{+}$ with $k\leq n$, $L\geq 0$ and a function $\Phi_{0}:\R^+\to \R^+$ with $\lim_{\delta\to0}\Phi_{0}(\delta)=0$, there exists $\delta_{0}>0$ depending on $n$, $\epsilon$, $L$, and $\Phi_{0}$ such that the following holds for every $\delta\in(0,\delta_{0})$.
Suppose $(X, d,m)$ is an $\RCD(-\delta,n)$-space and $\overline{B_{2}(p)}$ is compact in $X$, and
\begin{equation}\label{4.8888888888888}
\Theta^{(M,g)}_{p,k'}(s)\le \Phi_{0}(\max\{\delta,\theta^{(M,g)}_{p,k'}(s)\})
\end{equation}
holds for every $s\in(0,1)$ and every integer $k\leq k'\leq n$.
Suppose in addition $B_{2}(p)$ is $(\delta,k)$-Euclidean, and $u:(B_{2}(p),p))\rightarrow (\R^{k_{1}},0^{k_{1}})$ is a Lipschitz map which belongs to the domain of the local Laplacian on $B_2(p)$, where $k_{1}\leq k$, and suppose (\ref{2.11-1}) (\ref{2.11-2}) hold
for $1\leq \alpha,\beta\leq k_{1}$, then for every $r\in(0,1]$, there exists a $k_{1}\times k_{1}$ lower triangle matrix $T_{r}$ with positive diagonal entries such that (1)-(5) in Theorem \ref{thmA.1} hold.
\end{theorem}

Based on Theorem \ref{thm-trans-GRC-111}, we can prove the following non-degeneracy theorem, see \cite{HH24}*{Theorem 1.5}.

\begin{theorem}\label{thm-trans-GRC}
Given $n\in \mathbb{Z}^{+}$, $L\geq 0$ and a function $\Phi_{0}:\R^+ \to \R^+$ with $\lim_{\delta\to0}\Phi_{0}(\delta)=0$, there exists $\delta_{0}>0$ depending on $n$, $L$ and $\Phi_{0}$ so that the following holds for every $\delta\in(0,\delta_{0})$.
Suppose $(M,g,p)$ is an $n$-dimensional manifold so that $\Ric\geq -\delta$, $\overline{B_{2}(p)}$ is compact in $M$, $B_{2}(p)$ is $(\delta,k)$-Euclidean, and there exists a $C^1$ map $u : B_{2}(p)\rightarrow\mathbb{R}^{k}$ (where $1\leq k\leq n$) which belongs to the domain of the local Laplacian on $B_2(p)$ and satisfies (\ref{2.11-1}) (\ref{2.11-2})
for $1\leq \alpha,\beta\leq k$.
Suppose in addition that $B_{1}(p)$ satisfies the $(\Phi_{0}, 1; k,\delta)$-generalized Reifenberg condition,
then
\begin{enumerate}
  \item for any $x\in B_{1}(p)$, $du : T_{x} M\rightarrow\mathbb{R}^{k}$ is non-degenerate;
  \item for any $p_{1},p_{2}\in B_{\frac{1}{2}}(p)$, it holds
      \begin{equation}\label{BiHolder-app}
	(1-\Psi(\delta))d(p_{1},u^{-1}(u(p_{2})))^{1+\Psi(\delta)}\le |u(p_{1})-u(p_{2})|\le (1+\Psi(\delta))d(p_{1},u^{-1}(u(p_{2}))),
\end{equation}

\item $B_{\frac{1}{2}}(0^k) \subset u(B_{\frac{1}{2}+\Psi(\delta)}(p))$, where  $\Psi(\delta):=\Psi(\delta|n, L, \Phi_{0})$.
\end{enumerate}

\end{theorem}



If $(X, d,\mathcal{H}^n)$ is a non-collapsed $\RCD(-\delta,n)$-space and $d_{\mathrm{GH}}(B_2(p), B_2(0^n))\leq \delta$, then by \cite{KM20}, $B_{1}(p)$ satisfies the $(\Phi_0;1,\delta)$-generalized Reifenberg condition for suitable $\Phi_0$. In particular, Theorem \ref{thm-trans-GRC-111} holds for $(X, d,\mathcal{H}^n)$.
Because of the lack of smooth structure, Theorem \ref{thm-trans-GRC} does not hold in this case.
However, based on the transformation theorem, one can prove the following Canonical Reifenberg theorem.

\begin{theorem}[\cites{CJN21,BNS,HondaPeng2023}] \label{Rei}
Suppose $(X,d,\mathcal{H}^n)$ is an $\RCD(-\delta,n)$ space and $u: B_2(p) \to \mathbb{R}^n$ is a Lipschitz map which belonging to the domain of the local Laplacian on $B_2(p)$ and satisfying (\ref{2.11-1}) (\ref{2.11-2}) for $1\leq \alpha,\beta\leq n$. Then for any $x,y \in B_1(p)$ we have
$$(1-\Phi(\delta))d(x,y)^{1+\Phi(\delta)} \le d(f(x),f(y)) \le (1+\Phi(\delta))d(x,y),$$
where  $\Psi(\delta):=\Psi(\delta|n, L)$.
\end{theorem}


\section{Local Abelian groups actions}\label{sec-3}

In the following, we prove a lemma about local Abelian group actions on geodesic balls.

\begin{lem}\label{LocalAction}
Suppose $(X_{i}, p_{i}, d_{i}, G_{i})\pGH(X,p,d,G)$, where each $X_i$  and  $X$ are locally compact metric spaces.
If $G_{i}, G$ are all Abelian and $\Gamma$ is a finitely generated free Abelian subgroup of $G$ acting discontinuously on $X$, then
\begin{enumerate}
  \item There exists $R_{0}>0$ depending only on $\Gamma$ and $X$, so that for every $R\geq R_{0}$, there exists a couple of $\delta_{i}$-equivariant Gromov-Hausdorff approximations $f_{i}:(B_{2R}(p_{i}),p_{i})\rightarrow (B_{2R+\delta_{i}}(p),p)$, $\varphi_{i}:\Gamma(2R)\rightarrow G_{i}(2R+\delta_{i})$, such that $\varphi_{i}$ preserves the product and reverse operations, i.e., if $\gamma, \omega, \gamma\omega\in \Gamma(2R)$, then $\varphi_{i}(\gamma)\varphi_{i}(\omega)=\varphi_{i}(\gamma\omega)$ and $\varphi_{i}(\gamma^{-1})=\varphi_{i}(\gamma)^{-1}$.
  \item For any large $i$ depending on $R$, the local $\varphi_{i}(\Gamma(2R))$-action on $B_{R}(p_{i})$ induces a canonical equivalent relation: for each $x, y\in B_{R}(p_{i})$, $x\sim y$ holds if and only if there exists $\gamma\in\Gamma(2R)$ such that $x =\varphi_{i}(\gamma)(y)$.
      Similarly, the local $\Gamma(2R)$-action on $B_{R}(p)$ induces a canonical equivalent relation.
  \item Assume $X$ is a length space. Equipping the quotient metric, we have the following commutative diagram:
  \begin{align}\label{diag9.1}
\xymatrix@C=2.5cm{
  (B_{R}(p_{i}), {p}_{i}, \varphi_{i}(\Gamma(2R))) \ar[d]_{\sigma_{i}} \ar[r]^{\mathrm{GH}} & (B_{R}(p), p, \Gamma(2R)) \ar[d]^{\sigma} \\
  (B_{R}(p_{i})/\varphi_{i}(\Gamma(2R)), \bar{p}_{i}) \ar[r]^{\mathrm{GH}} & (B_{R}(p)/\Gamma(2R),\bar{p}),  }
\end{align}
  where $\sigma_{i}:B_{R}(p_{i})\rightarrow B_{R}(p_{i})/\varphi_{i}(\Gamma(2R))$, $\sigma:B_{R}(p)\rightarrow B_{R}(p)/\Gamma(2R)$ are the natural projection maps given by the equivalent relation in (2).
  \item If the $G_{i}$-action on $X_{i}$ is properly discontinuous and free, then the projection map $\sigma_{i}:B_{R}(p_{i})\rightarrow B_{R}(p_{i})/\varphi_{i}(\Gamma(2R))$ is a local homeomorphism.
      Moreover, there exists a surjective local homeomorphism  $\pi_{i}:B_{R}(p_{i})/\varphi_{i}(\Gamma(2R))\rightarrow B_{R}(p_{i})/G_{i}$ such that $\pi_{i}\circ\sigma_{i} = P_{i}$, where $P_{i}: B_{R}(p_{i})\rightarrow B_{R}(p_{i})/G_{i}$ is the standard projection.
  \item If each $X_{i}$ is a complete Riemannian manifold and the $G_{i}$-action on $X_{i}$ is properly discontinuous and free, then $B_{R}(p_{i})/\varphi_{i}(\Gamma(2R))$ is a manifold without boundary,  which admits a unique smooth structure so that $\sigma_{i}:B_{R}(p_{i})\rightarrow B_{R}(p_{i})/\varphi_{i}(\Gamma(2R))$ is a local diffeomorphism, and $\pi_{i}:B_{R}(p_{i})/\varphi_{i}(\Gamma(2R))\rightarrow B_{R}(p_{i})/G_{i}$ is also a surjective local diffeomorphism.
  \item There exists $R_{i}\rightarrow\infty$ such that the following commutative diagram holds:
  \begin{align}\label{diag9.2}
\xymatrix@C=2.5cm{
  (B_{R_{i}}(p_{i}), {p}_{i}, \varphi_{i}(\Gamma(2R_{i}))) \ar[d]_{\sigma_{i}} \ar[r]^{\mathrm{GH}} & (X, p, \Gamma) \ar[d]^{\sigma} \\
  (B_{R_{i}}(p_{i})/\varphi_{i}(\Gamma(2R_{i})), \bar{p}_{i}) \ar[r]^{\mathrm{GH}} & (X/\Gamma,\bar{p}).  }
\end{align}
\item Suppose $X_{i}$, $G_{i}$ as in (4), and $G/\Gamma$ is compact.
In (\ref{diag9.2}), for each large $i$, there exists a discrete closed Abelian subgroup $K_{i}<  \mathrm{Isom}(\pi_{i}^{-1}(B_{0.2R_{i}}(\underline{p}_{i})))$, where $\underline{p}_{i}=P_{i}(p_{i})\in X_{i}/G_{i}$, such that $\pi_{i}^{-1}(B_{0.2R_{i}}(\underline{p}_{i}))/K_{i}$ is isometric to $B_{0.2R_{i}}(\underline{p}_{i})$.
Further, up to a subsequence, we have the following commutative diagram:
  \begin{align}\label{diag9.3}
\xymatrix@C=2.5cm{
  (\pi_{i}^{-1}(B_{0.2R_{i}}(\underline{p}_{i})), \bar{p}_{i}, K_{i}) \ar[d]_{\pi_{i}} \ar[r]^{\mathrm{GH}} & (X/\Gamma,\bar{p}, K) \ar[d]^{\pi} \\
  (B_{0.2R_{i}}(\underline{p}_{i}), \underline{p}_{i}) \ar[r]^{\mathrm{GH}} & (X/G,\underline{p}).  }
\end{align}
where the limit group $K< \mathrm{Isom}(X/\Gamma)$ coincides with the image of the inclusion $G/\Gamma\rightarrow \mathrm{Isom}(X/\Gamma)$.
\end{enumerate}
\end{lem}

\begin{proof}
\textbf{Proof of (1).} Let $\{\gamma_{1},\ldots,\gamma_{k}\}$ be a linearly independent generating subset of $\Gamma$, then every $\gamma$ can be uniquely presented by $\gamma_{1}^{l_{1}}\ldots\gamma_{k}^{l_{k}}$ with $l_{1},\ldots,l_{k}\in \mathbb{Z}$.
For every $R> \max\{d(\gamma_{1}(p),p),\ldots,d(\gamma_{k}(p),p)\}$, since $\Gamma(2R)$ is finite, there exists $C(R) > 10$ such that for every $\gamma\in \Gamma(2R)$, the corresponding $l_{1},\ldots,l_{k}$ must satisfy $|l_{1}|+\ldots+|l_{k}|<C(R)$.
Let $\phi_{i}: G(10C(R)R)\rightarrow G_{i}(10C(R)R+\delta_{i}')$ and $h_{i}:(B_{10C(R)R}(p_{i}),p_{i})\rightarrow(B_{10C(R)R+\delta_{i}'}(p),p)$ be a couple of $\delta_{i}'$-equivariant Gromov-Hausdorff approximations, where $\delta_{i}'\downarrow0$.

For every $\gamma=\gamma_{1}^{l_{1}}\ldots\gamma_{k}^{l_{k}}\in \Gamma(2R)$, define
\begin{align}
\varphi_{i}(\gamma):=\phi_{i}(\gamma_{1})^{l_{1}}\ldots\phi_{i}(\gamma_{k})^{l_{k}}.\nonumber
\end{align}
Obviously $\varphi_{i}$ preserves the product and reverse operations.
Denote by $\bar{\gamma}=\gamma_{1}^{l_{1}-1}\ldots\gamma_{k}^{l_{k}}$, then for any $x\in B_{2R}(p_{i})$,
\begin{align}\label{9.3}
&d(h_{i}(\varphi_{i}(\gamma)(x)),\gamma(h_{i}(x)))\\
=&d(h_{i}(\phi_{i}(\gamma_{1})\varphi_{i}(\bar{\gamma})(x)),\gamma(h_{i}(x)))\nonumber\\
\leq&d(h_{i}(\phi_{i}(\gamma_{1})(\varphi_{i}(\bar{\gamma})(x))),\gamma_{1}(h_{i}(\varphi_{i}(\bar{\gamma})(x))))+ d(\gamma_{1}(h_{i}(\varphi_{i}(\bar{\gamma})(x))),\gamma_{1}\bar{\gamma}(h_{i}(x))) \nonumber\\
\leq &\delta_{i}'+d(h_{i}(\varphi_{i}(\bar{\gamma})(x)),\bar{\gamma}(h_{i}(x)))\nonumber\\
\leq &C(R)\delta_{i}'.\nonumber
\end{align}
In addition,
\begin{align}\label{9.4}
d(\varphi_{i}(\gamma)(p_{i}),p_{i}) &\leq d(h_{i}(\varphi_{i}(\gamma)(p_{i})),p)+\delta_{i}'\\
&\leq
d(\gamma(p),p)+d(h_{i}(\varphi_{i}(\gamma)(p_{i})),\gamma(h_{i}(p_{i})))+\delta_{i}'\nonumber\\
&\leq 2R+2C(R)\delta_{i}'.\nonumber
\end{align}
Take $\delta_{i}:=10C(R)\delta_{i}'$, then (\ref{9.4}) just means $\varphi_{i}(\gamma)\in G_{i}(2R+\delta_{i})$ for every $\gamma\in \Gamma(2R)$, and (\ref{9.3}) implies the map $f_{i}:=h_{i}|_{B_{2R}(p_{i})}$ and  $\varphi_{i}$ are a couple of $\delta_{i}$-equivariant Gromov-Hausdorff approximations.
This proves (1).

\textbf{Proof of (2).} For each $x, y\in B_{R}(p_{i})$, denote by $x\sim y$ if there exists $\gamma\in\Gamma(2R)$ such that $x =\varphi_{i}(\gamma)(y)$.
We verify $\sim$ is an equivalent relation as follows.
The reflexivity and symmetry are obvious.

Suppose $x\sim y$ and $y\sim z$ for $x, y, z \in B_{R}(p_{i})$, thus there exist $\gamma, \omega \in\Gamma(2R)$ such that $x =\varphi_{i}(\gamma)(y)$ and $y =\varphi_{i}(\omega)(z)$.
By (\ref{9.3}) we have $d(\gamma(f_{i}(y)), f_{i}(x))\leq \delta_{i}$ and $d(\omega(f_{i}(z)), f_{i}(y))\leq \delta_{i}$, then
\begin{align*}
d(\gamma\omega(f_{i}(z)), p)\leq d(\gamma\omega(f_{i}(z)), \gamma(f_{i}(y))) + d(\gamma(f_{i}(y)), f_{i}(x)) + d(f_{i}(x), f_{i}(p_{i}))\leq 3\delta_{i} + R,
\end{align*}
and hence
\begin{align}
d(\gamma\omega(p), p)\leq d(\gamma\omega(p), \gamma\omega(f_{i}(z))) + d(\gamma\omega(f_{i}(z)), p)\leq 2R + 4\delta_{i}.
\end{align}
Since $\Gamma(p)$ is discrete, for large $i$, we have $d(\gamma\omega(p), p)\leq 2R$, i.e. $\gamma\omega\in\Gamma(2R)$.
Then by (1), we have $x =\varphi_{i}(\gamma)\varphi_{i}(\omega)(z)=\varphi_{i}(\gamma\omega)(z)$, i.e., $x\sim z$.
Thus, $\sim$ has transitivity, hence is an equivalent relation.

Similarly, the local $\Gamma(2R)$-action on $B_{R}(p)$ induces an equivalent relation.

\textbf{Proof of (3).} Since $X$ is a length space, we could modify $h_{i}$ in (1) slightly so that $h_{i}(B_{R}(p_{i}))\subset B_{R}(p)$ is still a $\delta_{i}'$-Gromov-Hausdorff approximation and (\ref{9.3}) holds for some slightly larger $\delta_{i}'\downarrow 0$.
We define a map
$\bar{f}_{i}: (B_{R}(p_{i})/\varphi_{i}(\Gamma(2R)),\bar{p}_{i})\rightarrow (B_{R}(p)/\Gamma(2R),\bar{p})$ as follows:
for any $\bar{x}\in B_{R}(p_{i})/\varphi_{i}(\Gamma(2R))$, we choose some $x\in \sigma_{i}^{-1}(\bar{x})$, then define $\bar{f}_{i}(\bar{x})=\sigma(f_{i}(x))$.
Similar to the global case, one can verify that $\bar{f}_{i}$ is a $3\delta_{i}$-Gromov-Hausdorff approximation (the detailed proof is omitted here), and this proves (3).

\textbf{Proof of (4) and (5).} Since the $\varphi_i(\Gamma(2R))$ action is properly discontinuous and free, and it gives an equivalent relation on $B_{R}(p_{i})$ as in (2), the proofs of (4) and (5) are similar to the cases of global action, we omit them here. 

\textbf{Proof of (6).} It is by a standard diagonal argument basing on (3).

\textbf{Proof of (7).}  Since $G/\Gamma$ is compact, we have $D :=\mathrm{diam}((G/\Gamma)(\bar{p}))<\infty$.
Without loss of generality, we assume $R_{i}>100D$ for every $i$.

Firstly, we prove some elementary facts, which will be used several times below:

\textbf{Fact 1. } If $g\in G_{i}(r)$ for $r\leq 1.5 R_{i}$, then there exists $\gamma \in \Gamma(r + D + 2\delta_{i})$ such that $\varphi_{i}(\gamma)g \in G_{i}(D + 4\delta_{i})$.

\textbf{Proof of Fact 1:}
Note that there exists $\gamma^{-1}\in \Gamma$ such that $d(\gamma^{-1}(p), f_{i}(g(p_{i}))) < D + \delta_{i}$.
Then
\begin{align}
d(\gamma^{-1}(p),p)\leq d(\gamma^{-1}(p),f_{i}(g(p_{i})))+d(f_{i}(g(p_{i})),f_{i}(p_{i}))\leq D +r+2\delta_{i},\nonumber
\end{align}
i.e. $\gamma^{-1}\in \Gamma(r+D+2\delta_{i})\subset \Gamma(2R_{i})$.
So $d(\varphi_{i}(\gamma^{-1})(p_{i}), g(p_{i})) < D + 4\delta_{i}$, and hence $\varphi_{i}(\gamma)g \in G_{i}(D + 4\delta_{i})$.

\textbf{Fact 2. }
\begin{align}\label{9.6}
\pi_{i}^{-1}(B_{0.2 R_{i}}(\underline{p}_{i}))\subset \sigma_{i}(B_{0.25R_{i}}(p_{i})).
\end{align}

\textbf{Proof of Fact 2:} Given any $\bar{x}\in \pi_{i}^{-1}(B_{0.2R_{i}}(\underline{p}_{i}))$, choose $x\in B_{R_{i}}(p_{i})$ such that, $\sigma_{i}(x) =\bar{x}$.
Since $P_{i}(x) =\pi_{i}(\bar{x})\in B_{0.2R_{i}}(\underline{p}_{i})$, there exists $g\in G_{i}$ such that $d(x, g(p_{i})) < 0.2R_{i}$.
Then $g\in G_{i}(1.2R_{i})$.
By Fact 1, there exists $\gamma \in \Gamma(2R_{i})$ such that $\varphi_{i}(\gamma)g \in G_{i}(D + 4\delta_{i})$. Thus
\begin{align}
&d(\varphi_{i}(\gamma)(x), p_{i})=d(x, \varphi_{i}(\gamma^{-1})(p_{i}))\leq d(x, g(p_{i}))+d(g(p_{i}), \varphi_{i}(\gamma^{-1})(p_{i}))\\
 < &0.2R_{i} +D + 4\delta_{i}<0.25R_{i}.\nonumber
\end{align}
Since $\sigma_{i}(\varphi_{i}(\gamma)(x)) = \bar{x}$, (\ref{9.6}) follows.

We construct a map $\tau_{i}: G_{i}(0.25R_{i})\rightarrow \mathrm{Isom}(\pi_{i}^{-1}(B_{0.2R_{i}}(\underline{p}_{i})))$ as follows.
For any $g\in G_{i}(0.25R_{i})$, $\bar{x} \in \pi_{i}^{-1}(B_{0.2R_{i}}(\underline{p}_{i}))$, we take $x\in B_{0.75R_{i}}(p_{i})$ such that $\sigma_{i}(x) =\bar{x}$ (note that by (\ref{9.6}), such an $x$ always exists).
Note that $g(x) \in B_{R_{i}}(p_{i})$, thus $\sigma_{i}$ is defined at $g(x)$.
We define $\tau_{i}(g)(\bar{x}):= \sigma_{i}(g(x))$.

Now we verify the well-definedness of $\tau_{i}(g)(\bar{x})$.
For any $x, y \in B_{0.75R_{i}}(p_{i})$ satisfying $\sigma_{i}(x) = \sigma_{i}(y) =\bar{x}$, there exists $\gamma\in \Gamma(2R_{i})$ such that $x =\varphi_{i}(\gamma)(y)$.
So $g(x) = g(\varphi_{i}(\gamma)(y)) = \varphi_{i}(\gamma)(g(y))$, where the latter equality is by Abelianness of $G_{i}$.
Hence $\sigma_{i}(g(x)) =\sigma_{i}(g(y))$.

Next we verify $\tau_{i}(g)\in \mathrm{Isom}(\pi_{i}^{-1}(B_{0.2R_{i}}(\underline{p}_{i})))$.
For any $\bar{x}, \bar{y}\in \pi_{i}^{-1}(B_{0.2R_{i}}(\underline{p}_{i}))$,
$$
d(\tau_{i}(g)(\bar{x}), \tau_{i}(g)(\bar{y}))=d(\sigma_{i}(g(x)), \sigma_{i}(g(y))),
$$
where $x\in  \sigma_{i}^{-1}(\bar{x})\cap B_{0.75R_{i}}(p_{i})$, $y\in \sigma_{i}^{-1}(\bar{y})\cap B_{0.75R_{i}}(p_{i})$.
On the other hand,
\begin{align}\label{9.8}
d(\sigma_{i}(g(x)), \sigma_{i}(g(y))) = \min_{x',y'} d(x',y'),
\end{align}
where the minimum is taken among all $B_{R_{i}}(p_{i})\ni x'\sim g(x)$, $B_{R_{i}}(p_{i})\ni y'\sim g(y)$.
Combining the Abelianness of $G_{i}$, the right hand side of (\ref{9.8}) equals to
\begin{align}
\min_{\gamma_{1},\gamma_{2}\in \Gamma(2R_{i})} d(\varphi_{i}(\gamma_{1})(x), \varphi_{i}(\gamma_{2})(y)) = d(\bar{x}, \bar{y}).\nonumber
\end{align}
In conclusion, the map $\tau_{i}: G_{i}(0.25R_{i})\rightarrow \mathrm{Isom}(\pi_{i}^{-1}(B_{0.2R_{i}}(\underline{p}_{i})))$ is well-defined.

Denote the range of $\tau_{i}$ by $K_{i}$.
In the following, we verify that $K_{i}$ is a discrete closed Abelian subgroup of $\mathrm{Isom}(\pi_{i}^{-1}(B_{0.2R_{i}}(\underline{p}_{i})))$.

For any $g, h\in G_{i}(0.25R_{i})$, by fact 1, there exists $\gamma\in \Gamma(0.6R_{i})$ such that $\varphi_{i}(\gamma)gh \in G_{i}(D +4\delta_{i})\subset G_{i}(0.25R_{i})$.
In the following, we prove that
\begin{align}\label{9.9}
\tau_{i}(g)\tau_{i}(h) = \tau_{i}(\varphi_{i}(\gamma)gh).
\end{align}

For any $\bar{x}\in \pi_{i}^{-1}(B_{0.2R_{i}}(\underline{p}_{i}))$, by (\ref{9.6}), we choose $x\in B_{0.25R_{i}}(p_{i})$ such that $\sigma_{i}(x) = \bar{x}$.
By fact 1, there exists $\omega\in \Gamma(0.3R_{i})$ such that $\varphi_{i}(\omega)h\in G_{i}(D+4\delta_{i})$.
Then it is easy to check that
\begin{align}
&d(\varphi_{i}(\omega)h(x), p_{i}) \leq  0.3R_{i},\nonumber
\end{align}
and hence $\varphi_{i}(\omega)gh(x)=g\varphi_{i}(\omega)h(x)\in B_{0.55R_{i}}(p_{i})$.

Note that
\begin{align}
\varphi_{i}(\gamma)gh(x) = \varphi_{i}(\gamma)\varphi_{i}(\omega^{-1})g\varphi_{i}(\omega)h(x) = \varphi_{i}(\gamma\omega^{-1})g\varphi_{i}(\omega)h(x),
\end{align}
where we use the Abelianness of $G_{i}$ in the first equality, and use $\gamma\omega^{-1}\in \Gamma(2R_{i})$ in the latter equality.
So $\varphi_{i}(\gamma)gh(x)\sim g\varphi_{i}(\omega)h(x)$.

Since $\tau_{i}(h)(\bar{x})=\sigma_{i}(h(x))=\sigma_{i}(\varphi_{i}(\omega)h(x))$ with $\varphi_{i}(\omega)h(x)\in B_{0.75R_{i}}(p_{i})$, we have
\begin{align}
\tau_{i}(g)\tau_{i}(h)(\bar{x}) = \sigma_{i}(g\varphi_{i}(\omega)h(x)) = \sigma_{i}(\varphi_{i}(\gamma)gh(x)) = \tau_{i}(\varphi_{i}(\gamma)gh)(\bar{x}).
\end{align}
Since $\bar{x}\in \pi_{i}^{-1}(B_{0.2R_{i}}(\underline{p}_{i}))$ is arbitrary, (\ref{9.9}) holds.

(\ref{9.9}) means that the composition operation is closed in $K_{i}$.
In addition, by (\ref{9.9}), for any $g\in G_{i}(0.25R_{i})$, there exists $\gamma'\in\Gamma(0.6R_{i})$ such that, $\tau_{i}(g)\tau_{i}(g^{-1}) = \tau_{i}(\varphi_{i}(\gamma')) =\id$,
which verifies $\tau_{i}(g)$ is reversible in $K_{i}$.
Hence $K_{i}$ is a group.
Moreover, for any $g, h\in G_{i}(0.25R_{i})$, by (\ref{9.9}),
\begin{align}
\tau_{i}(g)\tau_{i}(h) = \tau_{i}(\varphi_{i}(\gamma)gh)= \tau_{i}(\varphi_{i}(\gamma)hg)=\tau_{i}(h)\tau_{i}(g),
\end{align}
hence $K_{i}$ is Abelian.
In addition, it is easy to see that $K_{i}$ is closed and discrete.
In conclusion, $K_{i}$ is a closed discrete Abelian subgroup of $\mathrm{Isom}(\pi_{i}^{-1}(B_{0.2R_{i}}(\underline{p}_{i})))$.

In the following, we will verify that $\pi_{i}^{-1}(B_{0.2R_{i}}(\underline{p}_{i}))/K_{i}$ is isometric to $B_{0.2R_{i}}(\underline{p}_{i})$.
We just need to verify that for any $\bar{x},\bar{y}\in \pi_{i}^{-1}(B_{0.2R_{i}}(\underline{p}_{i}))$ with $\pi_{i}(\bar{x}) =\pi_{i}(\bar{y})$, there exists $g\in G_{i}(0.25R_{i})$ such that $\bar{x}=\tau_{i}(g)(\bar{y})$.
By (\ref{9.6}), take $x, y\in B_{0.25R_{i}}(p_{i})$ such that $\sigma_{i}(x) = \bar{x}$ and $\sigma_{i}(y) = \bar{y}$.
So $P_{i}(x) = P_{i}(y)$.
This implies that there exists $h\in G$ such that $x = h(y)$.
It is easy to see $h\in G_{i}(0.5R_{i})$.
By fact 1, there exists $\gamma\in \Gamma(R_{i})$, such that $g := \varphi_{i}(\gamma)h \in G_{i}(D + 4\delta_{i})$.
So $x =\varphi_{i}(\gamma^{-1})g(y)$.
Note that $\varphi_{i}(\gamma^{-1})(y)= g^{-1}(x) \in B_{0.75R_{i}}(p_{i})$.
Thus,
\begin{align}
\bar{x} =\sigma_{i}(x)  =\sigma_{i}(\varphi_{i}(\gamma^{-1})g(y))= \sigma_{i}(g\varphi_{i}(\gamma^{-1})(y))= \tau_{i}(g)(\bar{y}).
\end{align}

By the above discussion, up to a subsequence, there exists some $K< \mathrm{Isom}(X/\Gamma)$ such that (\ref{diag9.3}) holds.
To finish the proof, we need to verify that $K = \mathrm{im}(G/\Gamma \rightarrow \mathrm{Isom}(X/\Gamma))$.

For any $\bar{g}\in K$, up to a subsequence, we choose $g_{i}\in G_{i}(0.25R_{i})$ such that $K_{i}\ni \tau_{i}(g_{i})\xrightarrow{\mathrm{GH}}\bar{g}$.
By fact 1, for each large $i$, there exists $\gamma_{i}\in \Gamma(0.3R_{i})$ such that $h_{i}:=\varphi_{i}(\gamma_{i})g_{i}\in G_{i}(D + 4\delta_{i})$.
It is easy to see that $\tau_{i}(h_{i}) =\tau_{i}(g_{i})$.
Passing to a subsequence, we may assume $h_{i}\xrightarrow{\mathrm{GH}} h\in G(D)$.
Let $\bar{h}$ be the image of $h\Gamma$ through $G/\Gamma \rightarrow \mathrm{Isom}(X/\Gamma)$.
For any $\bar{x}\in X/\Gamma$, take $\bar{x}_{i}\in \pi_{i}^{-1}(B_{0.2R_{i}}(\underline{p}_{i}))$ such that $\bar{x}_{i}\xrightarrow{\mathrm{GH}} \bar{x}$, and then according to (\ref{9.6}), we can choose $x_{i}\in B_{0.25R_{i}}(p_{i})$ such that $\sigma_{i}(x_{i}) = \bar{x}_{i}$.
Passing to a subsequence, we may assume $x_{i}\xrightarrow{\mathrm{GH}}x\in X$, then $\sigma(x) =\bar{x}$.
So
\begin{align}
\bar{g}(\bar{x}) = \lim \tau_{i}(h_{i})(\bar{x}_{i}) = \lim \sigma_{i}(h_{i}(x_{i})) = \lim \sigma(h(x)) = \bar{h}(\bar{x}),
\end{align}
which concludes that $\bar{g}= \bar{h}$.
Thus $K\subset \mathrm{im} (G/\Gamma\rightarrow\mathrm{Isom}(X/\Gamma))$.

For the other inclusion, let $g \in G$ with $G_{i}\ni g_{i}\xrightarrow{\mathrm{GH}}g$, and $\bar{g}$ be the image of $g\Gamma$ in $\mathrm{im} (G/\Gamma\rightarrow\mathrm{Isom}(X/\Gamma))$.
For any $\bar{x}\in X/\Gamma$, take $\bar{x}_{i}\in \pi_{i}^{-1}(B_{0.2R_{i}}(\underline{p}_{i}))$ such that $\bar{x}_{i}\xrightarrow{\mathrm{GH}} \bar{x}$, and then according to (\ref{9.6}), we can choose $x_{i}\in B_{0.25R_{i}}(p_{i})$ such that $\sigma_{i}(x_{i}) = \bar{x}_{i}$.
Passing to a subsequence, we may assume $x_{i}\xrightarrow{\mathrm{GH}}x\in X$ with $\sigma(x) =\bar{x}$.
So
\begin{align}
\lim \tau_{i}(g_{i})(\bar{x}_{i}) = \lim \sigma_{i}(g_{i}(x_{i})) = \sigma(g(x)) = \bar{g}(\bar{x}),
\end{align}
which implies $K_{i}\ni \tau_{i}(g_{i})\xrightarrow{\mathrm{GH}}\bar{g}$, and $\bar{g} \in K$.
By the arbitrariness of $\bar{g}$, we conclude that $\mathrm{im} (G/\Gamma\rightarrow\mathrm{Isom}(X/\Gamma)) \subset K$.

The proof is completed.
\end{proof}

\section{Proof of the pinching Theorems \ref{mainthm-mix-curv-pinching} and \ref{mainthm_B}}\label{sec-13}


\begin{proof}[Proof of Theorem \ref{mainthm-mix-curv-pinching}]
We argue by contradiction.
Suppose on the contrary, there exists a sequence of $n$-dimensional manifolds $M_{i}$ satisfies $\Ric_{M_{i}} \geq -\epsilon_{i}$, $\sec_{M_{i}} \geq -1$, and $\diam(M_{i})\leq 1$, $b_1(M_{i})=b$, where $\epsilon_{i}\downarrow0$, but each $M_{i}$ is not a fiber bundle over a $b$-torus.
Note that by Gromov (\cite{Gr81-II}), we have $b\leq n$, and the $b=0$ case is trivial, and the $b=n$ case is due to Colding (\cite{Col97}), so we may assume $b\in[1,n-1]$.

Denote by $\Gamma_{i}=\pi_{1}(M_{i})$.
Let $\tilde{M}_{i}$ be the Riemannian universal cover of $M_{i}$.
We choose (and fix) some $\tilde{p}_{i}\in \tilde{M}_{i}$ which projects to $p_{i}\in M_{i}$ by the covering map.
Let $\hat{M}_{i}:=(\tilde{M}_{i}/[\Gamma_{i},\Gamma_{i}])/T_{i}$, where $T_{i}$ is the torsion subgroup of $\Gamma_{i}/[\Gamma_{i},\Gamma_{i}]$, then $H_{i}=(\Gamma_{i}/[\Gamma_{i},\Gamma_{i}])/T_{i}$ is a finitely generated free Abelian group with $\mathrm{rank} (H_{i})=b$, and $\hat{M}_{i}/H_{i}=M_{i}$.
We denote by $\hat{p}_{i}\in \hat{M}_{i}$ the point which is the image of $\tilde{p}_{i}$ under the quotient map.

Up to passing to a subsequence, we consider the following commutative diagram:

\begin{align}\label{diag6.1}
\xymatrix@C=2.5cm{
  (\hat{M}_{i}, \hat{p}_{i}, H_{i}) \ar[d]_{P_{i}} \ar[r]^{\mathrm{GH}} & (\mathbb{R}^{s}\times\hat{Y}, (0^{s},\hat{y}_{\infty}), H) \ar[d]^{P_{\infty}} \\
  ({M}_{i}, p_{i}) \ar[r]^{\mathrm{GH}} & (X,p_{\infty}),  }
\end{align}
where $\hat{Y}$ cannot split off an $\mathbb{R}$-factor.
Obviously, $X$ is compact, and $H$ is Abelian. By the Cheeger-Gromoll's trick, one can prove that $\hat{Y}$ is also compact.

Since  $\diam(M_{i})\leq 1$,  by a standard argument, $H_i$ is generated by
$H_i(3)$. By Lemma \ref{lem:rank non-dec}, we have $\mathrm{rank}(H)\geq b$. (In fact, according to Theorem \ref{thm:rank-stability} and Remark \ref{rem:rank-stability}, we have $\mathrm{rank}(H)= b$; but the equality is not used in the proof.) Thus there is a $\mathbb{Z}^b$-subgroup of $H$ acting freely on $\R^s$, which implies $s\geq b$.

Note that $H$ acts separately on the two factors of $\mathbb{R}^{s}\times\hat{Y}$, that is, each $h\in H$ has the form $(h_{1}, h_{2})$, where $h_{1}\in \Isom(\mathbb{R}^{s})$, $h_{2}\in \Isom(\hat{Y})$.
Let $\Pi_{1}:\Isom(\mathbb{R}^{s}\times\hat{Y})\rightarrow \Isom(\mathbb{R}^{s})$, $\Pi_{1}:\Isom(\mathbb{R}^{s}\times\hat{Y})\rightarrow \Isom(\hat{Y})$ be the projection map defined by $\Pi_{1}((h_{1}, h_{2}))=h_{1}$, $\Pi_{2}((h_{1}, h_{2}))=h_{2}$.
Denote  $G=\mathrm{Im}(\Pi_{1})$.
Obviously, $\diam(\mathbb{R}^{s}/G)\leq 1$.
We recall a classical theorem which says that every element of an Abelian, cocompact subgroup of $\mathrm{Isom}(\R^{n})$ acts as a translation on $\R^{n}$.
Then each element of $G$ acts as a translation on $\R^{s}$.
We take a discrete subgroup $G_{0}$ of $G$ such that the flat torus $T^{s}:=\R^{s}/G_{0}$ satisfies $1\leq \diam(T^{s})\leq 4n$, and $\mathrm{inj}_{T^{s}}\geq 1$.
Note that $G/G_{0}$ is a compact group.

Denote by $N_{0}=\mathrm{ker}(\Pi_{1})$.
Note that elements in $N_{0}$ leave the $\R^{s}$ factor fixed, and hence we can view $N_{0}$ as a subgroup of $\Isom(\hat{Y})$, which is compact.


Let $\{\eta_{1},\eta_{2},\ldots,\eta_{s}\}$ be a set of linear independent generators of $G_{0}$.
For each $i\in\{1,\ldots,s\}$, fix $\gamma_{i}\in H$ such that $\Pi_{1}(\gamma_{i})=\eta_{i}$.
Let $H_{0}:=\mathrm{span}\{\eta_{1},\eta_{2},\ldots,\eta_{s}\}$, which is a discrete subgroup of $H$.
It is easy to check that $H/H_{0}$ is isometric to $G/G_{0}\rtimes N_{0}$.
Denote by $\bar{X}:=(\mathbb{R}^{s}\times\hat{Y})/H_{0}$.
It is easy to check $\diam(\bar{X})\leq \sqrt{D^{2}+(4n)^{2}} < 4nD$,
where $D:=\diam(\hat{Y})$.

There is a natural surjective map $F:\bar{X}\rightarrow T^{s}$ defined as follows: for every $\bar{x}\in \bar{X}$, take a pre-image $\hat{x}\in\mathbb{R}^{s}\times\hat{Y}$, then $F(\bar{x})$ is the projection of $\pi_{1}(\hat{x})$ to $T^{s}= \mathbb{R}^{s}/H_{0}$.
Moreover, it is easy to see that, $F$ is a submetry at scale $\frac{1}{4}$.


By (\ref{diag6.1}) and Lemma \ref{LocalAction}, up to choosing a subsequence of $i$, there exist $R_{i}\rightarrow \infty$ and pseudo-groups $\hat{K}_{i}\subset H_{i}(2R_{i}+\lambda_{i}')$
(where $\lambda_{i}'\downarrow 0$) such that
the following commutative diagram holds:
  \begin{align}\label{diag6.3}
\xymatrix@C=2.5cm{
  (B_{R_{i}}(\hat{p}_{i}), \hat{p}_{i}, \hat{K}_{i})) \ar[d]_{\sigma_{i}} \ar[r]^{\mathrm{GH}} & (\mathbb{R}^{s}\times\hat{Y}, (0^{s},\hat{y}_{\infty}), H_{0}) \ar[d]^{\sigma} \\
  (B_{R_{i}}(\hat{p}_{i})/\hat{K}_{i}, \bar{p}_{i}) \ar[d]_{\tau_{i}} \ar[r]^{\mathrm{GH}} & (\bar{X},\bar{p}) \ar[d]^{\tau},  \\
    (B_{R_{i}}(p_{i}), {p}_{i}) \ar[r]^{\mathrm{GH}} & (X, p_{\infty}),
    }
\end{align}
with $P_{i}|_{B_{R_{i}}(\hat{p}_{i})}=\tau_{i}\circ\sigma_{i}$.

Since $H/H_{0}\cong G/G_{0}\rtimes N_{0}$ is compact, by (\ref{diag6.1}), (\ref{diag6.3}) and (7) in Lemma \ref{LocalAction}, up to a subsequence, there exists a discrete closed subgroup $K_{i}<  \mathrm{Isom}(\tau_{i}^{-1}(B_{\frac{1}{5}R_{i}}(p_{i})))$, such that $\tau_{i}^{-1}(B_{\frac{1}{5}R_{i}}(p_{i}))/K_{i}$ is isometric to $B_{\frac{1}{5}R_{i}}(p_{i})$, and the following commutative diagram holds:
  \begin{align}\label{diag6.5}
\xymatrix@C=2.5cm{
  (\tau_{i}^{-1}(B_{\frac{1}{5}R_{i}}(p_{i})), \bar{p}_{i}, K_{i}) \ar[d]_{\tau_{i}} \ar[r]^{\mathrm{GH}} & (\bar{X},\bar{p}, K) \ar[d]^{\tau} \\
  (B_{\frac{1}{5}R_{i}}(p_{i}), p_{i}) \ar[r]^{\mathrm{GH}} & (X, p_{\infty}).  }
\end{align}
where $K< \mathrm{Isom}(\bar{X})$ coincides with the image of $H/H_{0}\rightarrow \mathrm{Isom}(\bar{X})$.
Since $\diam(M_{i})\leq 1$ and $\diam(\bar{X})< 4nD$, without loss of generality, we may assume $B_{\frac{1}{5}R_{i}}(p_{i})=M_{i}$ and $\diam(\tau_{i}^{-1}(M_{i}))< 5nD$ for every $i$.
$\tau_{i}^{-1}(M_{i})$ are compact manifolds without boundary.
Since $H$ acts separately on the two factors of $\mathbb{R}^{s}\times\hat{Y}$, it is easy to see that there is a natural homomorphism $\hat{\Psi}: K\rightarrow \Isom(T^{s})$ such that for every $g\in K$,
\begin{align}
\hat{\Psi}(g)(F(\bar{x}))=F(g(\bar{x})),
\end{align}
and each $\hat{\Psi}(g)$ can be lifted as a translation of $\R^s$.
In particular, for any discrete subgroup $K'< \hat{\Psi}(K)$, $T^s/K'$ is isometric to some $s$-dimensional torus.

Applying Theorem \ref{thm:eq-submetry} to the equivariant Gromov-Hausdorff approximations given in (\ref{diag6.5}) and the map $F:\bar{X}\rightarrow T^{s}$,  
there exist homomorphisms $\psi_{i}:K_{i}\rightarrow \hat{\Psi}(K)$ and
smooth maps $F_{i}:\tau_{i}^{-1}(M_{i})\rightarrow T^{s}$, such that $F_{i}$ converges uniformly to $F:\bar{X}\rightarrow T^{s}$,
\begin{align}\label{6.93}
F_{i}\circ g= \psi_{i}(g)\circ F \text{ for every }g\in K_{i}.
\end{align}
Since $\sec_{M_{i}} \geq -1$, each $M_{i}$ satisfies the $\Phi_{0}$-generalized Reifenberg condition for some $\Phi_{0}$ depending only on $n$.
Then by Corollary \ref{cor7.2-smooth}, $d F_{i}$ is non-degenerate at each point and $F_{i}$ is surjective onto $T^s$.
In particular, each $F_{i}$ is a fibration over $T^{s}$.


Denote $K_{i}'=\psi_{i}(K_{i})$, which is a discrete subgroup of $\hat{\Psi}(K)$.
By (\ref{6.93}), there is a well-defined smooth surjective map $\hat{F}_{i}:\tau_{i}^{-1}(M_{i})/K_{i}\rightarrow T^{s}_{i}:=T^{s}/K_{i}'$ such that $d \hat{F}_{i}$ is non-degenerate at every point.
Thus $\hat{F}_{i}$ gives a fibration from $M_{i}$ over $T^{s}_{i}$.

Note that the fibers of $M_{i}$ over $T^{s}_{i}$ comes from the fibers of $F_{i}:\tau_{i}^{-1}(M_{i})\rightarrow T^{s}$ by quotient out a suitable group.
Thus, the connectedness of the fiber of $M_{i}$ over $T^{s}_{i}$ comes from that of the fibers of $\tau_{i}^{-1}(M_{i})$ over $T^{s}$.
We assume the fiber of $\tau_{i}^{-1}(M_{i})$ over $T^{s}$ is disconnected, which is the disjoint unit of $\{S_{1},\ldots,S_{l}\}$.
Since $\mathrm{inj}_{T^{s}}\geq 1$, for every $q\in T^{s}$, $F_{i}^{-1}(B_{\frac{1}{5}}(q))$ is diffeomorphic to $B_{\frac{1}{5}}(q)\times \sqcup_{j=1}^{l} S_{j}$.
Since $F_{i}$ converges uniformly to $F:\bar{X}\rightarrow T^{s}$ and $F$ is a submetry at scale $\frac{1}{4}$, $F_{i}$ is a $\epsilon_{i}$-submetry at scale $\frac{1}{4}$ ($\epsilon_{i}\downarrow0$).
For every $\bar{q}_{1}, \bar{q}_{2}\in \tau_{i}^{-1}(M_{i})$ such that $d(\bar{q}_{1},\bar{q}_{2})\leq \frac{1}{10}$, since $F_{i}$ is an almost submetry, $d(F_{i}(\bar{q}_{1}),F_{i}(\bar{q}_{2}))\leq\frac{1}{9}$, we claim that $\bar{q}_{1}, \bar{q}_{2}$ fall in the same component of fibers.
Suppose otherwise, denote by $q=F_{i}(\bar{q}_{1})$, since $B_{\frac{1}{8}}(\bar q_{1})\subset  F_i^{-1}(B_{\frac{1}{5}}(q))\overset{\mathrm{diff}}{\sim} B_{\frac{1}{5}}(q)\times \sqcup_{j=1}^{l} S_{j}$, $B_{\frac{1}{8}}(\bar q_1)$ must be contained in $B_{\frac{1}{5}}(q)\times S_{1}$, where $ \bar{q}_{1}\in S_{1}$, which implies $\bar{q}_{2}\in S_{1}$, a contradiction.

We take $x_1, x_{2}\in \tau_{i}^{-1}(p_{i})$, let $y_1, y_{2}\in \bar{X}$ be the image of the Gromov-Hausdorff approximation of $x_1, x_{2}$, connect $y_1, y_{2}$ by a geodesic $\eta$, and take $\{\eta_1=y_1, \eta_2, \ldots , \eta_t=y_2\}\subset \eta$ such that $d(\eta_i,\eta_{i+1})\leq \frac{1}{20}$, and take $\{\gamma_1=x_1, \gamma_2, \ldots , \gamma_t=x_2\}\subset \tau_{i}^{-1}(p_{i})$ by the pre-image of the Gromov-Hausdorff approximation.
By the conclusion in the proceeding paragraph, we can conclude that $x_1, x_{2}$ is contained in the same component of the fiber.
This shows the connectedness of the fiber.

In conclusion, we get a contradiction, and the proof of Theorem \ref{mainthm-mix-curv-pinching} is complete.
\end{proof}

\begin{proof}[Proof of Theorem \ref{mainthm_B}]
The proof of Theorem~\ref{mainthm_B} is a minor modification of that of Theorem~\ref{mainthm-mix-curv-pinching}.
In the setting of Theorem~\ref{mainthm-mix-curv-pinching}, we apply Theorem \ref{thm:eq-submetry} to construct a smooth map $F$, and obtain that $dF$ is non-degenerate due to the lower bound on sectional curvature. Consequently, $F$ is a smooth fiber bundle map.

In contrast, in the setting of Theorem~\ref{mainthm_B}, the map $F$ constructed by Theorem \ref{thm:eq-submetry} may fail to be smooth. Nevertheless, by Theorem~\ref{thm:bundlemap} below, we can still conclude that $F$ is a topological fiber bundle map.
We remark that in order to apply Theorem~\ref{thm:bundlemap} , we need a fact that if a $\mathrm{RCD}(-\delta,n)$ space $(X,d,\mathcal{H}^n)$ satisfies the $(r,\delta)$-local rewinding Reifenberg condition, then any of its covering space $(\bar{X},\bar{d},\mathcal{H}^n)$ satisfies the $(r,\Psi(\delta|n))$-local rewinding Reifenberg condition.
In fact, the natural projection $\pi: \bar{X} \rightarrow X$ can be lifted to a local homeomorphism $\Pi: \widetilde{B_r(q)}\rightarrow \widetilde{B_r(p)}$ (where $\pi(q)=p$) which maps ${B_t(\tilde{q})}$ onto ${B_t(\tilde{p})}$ for any $t\leq r$.
\end{proof}

We can use the almost splitting map to construct a topological fiber bundle structure on a $\mathrm{RCD}$ space. See also \cite{Wang2024}.

\begin{theorem}\label{thm:bundlemap}
Suppose $(X,d,\mathcal{H}^n)$ is a compact $\mathrm{RCD}(-1,n)$ space satisfying the $(r,\delta')$-local rewinding Reifenberg condition for some $r>0$ and $\delta'\in (0,\delta'_n)$, where $\delta'_n$ is a sufficiently small positive constant depending only on $n$.
Let $F \colon X \to N$ be the map constructed in Theorem \ref{thm:eq-submetry}.
If $\epsilon$ in Theorem \ref{thm:eq-submetry} is sufficiently small, then $F \colon X \to N$ is a topological fiber bundle and the fiber is a topological $(n-k)$-manifold.

Moreover, if $G$ and $H$ are free actions, and $G$ is discrete, then the induced map
\begin{align*}
\bar{F} \colon X/G \to N/\psi(G)
\end{align*}
is also a topological fiber bundle.
\end{theorem}
\begin{proof}
Since (c1)-(c3) of Theorem \ref{thm:eq-submetry} hold, we rescale $X$, $N$ by a suitable large constant and still use $X$, $N$ and $F$ to denote the rescaled spaces and maps respectively, and rechoose suitable $r$ and $\epsilon$, we may assume that $(X,d,\mathcal{H}^n)$ is a $\mathrm{RCD}(-\epsilon,n)$ space satisfying the $(r,\delta')$-local rewinding Reifenberg condition for $r>100$, and for every $x \in X$, after compositing the normal coordinate chart $\Theta_y$, where $y=F(x)$, the resulting map $\Theta_y\circ F:B_{10}(x)\rightarrow \R^k$ is $(k,\epsilon)$-splitting.

For any $x \in X$, let $\tilde{x}$ be a lift of $x$ to the universal cover of $B_r(x)$. Then $\tilde{x}$ is a $(r/3,\delta)$-Reifenberg point. Since $X$ is compact, we may choose $r_0 < 1$ sufficiently small such that for every $x \in X$, the ball $B_{r_0}(x)$ is isometric to $B_{r_0}(\tilde{x})$. By the Reifenberg Theorem, $X$ is a topological manifold.

Fix $x \in X$ and set $y = F(x) \in N$. Using the inverse of the exponential map at $y$, we identify $B_{10}(y)$ with a ball in $T_yN \cong \mathbb{R}^k$. Then by the covering lemma (\cite{KW}*{Lemma 1.6}), it is standard that the restriction of $F$ to $B_{10}(x)$ lifts to a map
$\tilde{F} \colon B_{10}(\tilde{x}) \to \mathbb{R}^k$,
which is a $(k,\Psi(\epsilon|n))$-splitting map. Moreover, we may assume that $\tilde{F}(\tilde{x}) = 0^k \in \mathbb{R}^k$.
Since $\tilde{x}$ is an $(10,\delta')$-Reifenberg point, by the
Cheeger-Colding theory, if $\epsilon$ is sufficiently small, there exists a $(N-k,\Psi(\delta'|n))$-splitting harmonic map
$v \colon B_{10}(\tilde{x}) \to \mathbb{R}^{N-k}$
such that $(\tilde{F}, v)$ is a $(N,\Psi(\delta'|n))$-splitting map on $B_{10}(\tilde{x})$.

By the Transformation Theorem~\ref{thmA.1}, there exists a lower triangular matrix $T$ such that
$T \circ (\tilde{F}, v)$
is a $(N,\Psi(\delta'|n))$-splitting map on $B_{r_0}(\tilde{x})$. Since $T$ is lower triangular, it preserves the subspace $\{0^k\} \times \mathbb{R}^{N-k}$, i.e.,
$T\bigl(\{0^k\} \times \mathbb{R}^{N-k}\bigr) = \{0^k\} \times \mathbb{R}^{N-k}$.

Since $B_{r_0}(x)$ is isometric to $B_{r_0}(\tilde{x})$, the map
$T \circ (\tilde{F}, v)$
descends to a  $(N,\Psi(\delta'|n))$-splitting map $(\bar{F}, \bar{v})$ on $B_{r_0}(x)$.

By the canonical Reifenberg Theorem~\ref{Rei}, the map $(\bar{F}, \bar{v})$ is a homeomorphism from $B_{r_0}(x)$ onto its image. In particular, the set $\bar{F}^{-1}(0^k)$ is locally homeomorphic to $\mathbb{R}^{N-k}$ via the coordinate function $\bar{v}$. Hence, $\bar{F}^{-1}(0^k) \subset B_{r_0}(x)$ is a neighborhood of $x$ in $F^{-1}(y)$. It follows that $F^{-1}(y)$ is a topological $(N-k)$-manifold.

Next, we describe a local representation of $F$. The pair $(B_{r_0}(x), (\bar{F}, \bar{v}))$ provides a (topological) chart for $X$, while a neighborhood of $y$ can be identified with an open subset of $T_yN \cong \mathbb{R}^k$. In the $(B_{r_0}(x), (\bar{F}, \bar{v}))$ coordinate chart, the map
$F \colon X \to N$ is locally represented by
$\pi \circ T^{-1} \colon \mathbb{R}^N \to \mathbb{R}^k$,
where $\pi \colon \mathbb{R}^k \times \mathbb{R}^{N-k} \to \mathbb{R}^k$ is the projection map.

The map $\pi \circ T^{-1}$ is a submersion, since $T$ is non-degenerate. Therefore, $F$ is a topological submersion, that is, there exists an open neighborhood $U$ of $y=F(x)$ and an open neighborhood of $x$ homeomorphic to $U \times B_{r_0/2}(0^{n-k})$ such that $F:X \to N$ locally induces the projection map $\pi: U \times B_{r_0/2}(0^{n-k}) \to U$.

Since $F$ is a topological submersion and $X$ is compact, it follows from \cite{KirbySiebenmann}*{Essay II, Chapter 1} that
$F \colon X \to N$
is a topological fiber bundle with fibers homeomorphic to a $(n-k)$-manifold.

Moreover, for any $\gamma \in G$, we have
\[
F \circ \gamma = \psi(\gamma) \circ F.
\]
Thus, $F$ descends to a map between the quotient spaces
$\bar{F} : X/G \to N/\psi(G)$.
Since both $G$ and $H$ act freely and $G$ is discrete, it follows that $N/\psi(G)$ is a $k$-manifold.

By the argument above, $\bar{F}$ is also a topological submersion, and hence a topological fiber bundle.
\end{proof}

\section{Proof of the smooth fibration Theorem \ref{mainthm-fiber-sec}} \label{sec-4}

The fibration map $f:M\rightarrow N$ in Theorem \ref{mainthm-fiber-sec} is given by the following theorem.

\begin{theorem}\label{Fibrations}
    Given $n\in \mathbb{Z}^+$, $r_0>0$.
	Suppose $(M,g)$, $(N,h)$ are compact Riemannian manifolds of dimension $n$, $k$ respectively (where $k\le n$) such that
$\Ric_{M}\geq -1$,
	\begin{align}\label{5.2}
    \Ric_N \geq -1, \quad and \quad   d_{\mathrm{GH}}(B_{r_{0}}(q),B_{r_{0}}(0^{k})) < \delta \quad \text{ for any } q\in N,
    \end{align}
	and $d_{\mathrm{GH}}(M,N)\le\delta$. Then there exists a smooth map $f:M\rightarrow N$ which is a $\Psi(\delta)$-Gromov-Hausdorff approximation,
and for any $x\in M$, there exists a coordinate chart $\Theta_{y}:B_{2\sqrt{\Psi(\delta)}}(y)\rightarrow V \subset \mathbb{R}^{k}$ for $y=f(x)$,
	such that $\Theta_{y}$ is a $\Psi(\delta)$-Gromov-Hausdorff approximation, and $\bar{f}=\Theta_{y}\circ f|_{B_{\sqrt{\Psi(\delta)}}(x)}$ satisfies
	\begin{align}\label{1.1}
		\Lip \bar{f}\leq C_{0},
	\end{align}
	\begin{align}\label{1.2}
		|\Delta \bar{f}|_{L^{\infty}(B_{\sqrt{\Psi(\delta)}}(x))}\leq \frac{C_{0}}{\sqrt{\Psi(\delta)}},
	\end{align}
	\begin{align}\label{1.3}
		\bbint_{B_{\sqrt{\Psi(\delta)}}(x)}|\langle\nabla \bar{f}^{i}, \nabla \bar{f}^{j} \rangle- \delta_{ij}|\leq \Psi(\delta),
	\end{align}
where $C_{0}=C_{0}(n)$, $\Psi(\delta)=\Psi(\delta|n,r_{0})$.
If in addition, $M$ satisfies the $(\Phi_{0}, 1; k,\delta)$-generalized Reifenberg condition for a function $\Phi_{0}:\mathbb{R}^{+}\rightarrow \mathbb{R}^{+}$ with $\lim_{s\rightarrow 0^{+}}\Phi_{0}(s)=0$,
then if  $\delta$ is sufficiently small (depending only on $n$, $r_{0}$, $\Phi_{0}$), $f:M\rightarrow N$ is a surjective fibration.
\end{theorem}

Theorem \ref{Fibrations} generalizes \cite{HH24}*{Theorem 1.6} in two senses.
Compared with \cite{HH24}*{Theorem 1.6}, the assumptions on the lower dimensional manifold are slightly weaken to (\ref{5.2}), and we get more estimates (\ref{1.2}), (\ref{1.3}).

The proof of Theorem \ref{Fibrations} is tedious, and we delay it to Section \ref{sec7.1.1}.

In the following we will prove Theorem \ref{mainthm-fiber-sec}.
Since $\sec_{M}\geq -1$, according to \cite{HH24}*{Proposition 5.6}, $M$ satisfies the $(\Phi_{0}, 1; k,\delta)$-generalized Reifenberg condition for  $\Phi_{0}$ depending only on $n$.
We will verify that for sufficiently small $\delta$, the fibration $f:M\rightarrow N$ given by Theorem \ref{Fibrations} satisfies the topological properties.

Without loss of generality, we assume that $\Psi(\delta)$ in Theorem \ref{Fibrations} is greater than $\delta$.

In the following, we prove the topological properties of the fibrations as in Theorem \ref{mainthm-fiber-sec}.
For simplicity of notation, we assume $r_{0}=1$ in the assumptions of Theorem \ref{mainthm-fiber-sec}.
But the constants in the proof will implicitly depend on $r_{0}$.

In the proof, we will use an argument by contradiction. Let's fix some notations first.

Assume that we have a sequence of $\delta_{i}\downarrow 0$, sequences of manifolds $M_{i}$, $N_{i}$, sequence of maps $f_{i}:M_{i}\rightarrow N_{i}$ which is a fibration and a $\Psi(\delta_{i})$-Gromov-Hausdorff approximation as in Theorem \ref{Fibrations}.
We denote $f_{i}^{-1}(q_i)$ by $F_i$, and fix some $p_{i}\in F_{i}$.

By Theorem \ref{Fibrations}, $\Theta_{q_{i}}\circ f_{i}|_{B_{2\sqrt{\Psi(\delta_{i})}}(p_{i})}$ is a $\Psi(\delta_{i})$-Gromov-Hausdorff approximation.
Since we are interested in the properties of the fibers, for simplicity of notation, we replace $\Theta_{q_{i}}\circ f_{i}$ by $f_{i}$, and assume $N_{i}=B_{\sqrt{\Psi(\delta_{i})}}(0^{k})\subset \R^{k}$ with $q_{i}=0^{k}$, and $M_{i}=f_{i}^{-1}(B_{\sqrt{\Psi(\delta_{i})}}(0^{k}))$, $p_{i}\in f_{i}^{-1}(0^{k})$.
Note that in this case, $f_{i}$ is a trivial bundle, i.e., $M_{i}$ is diffeomorphic to $N_i\times F_{i}$.

Denote $\Gamma_{i}:=\pi_{1}(F_{i})$. Since $F_{i}$ is compact, $\Gamma_{i}$ is finitely generated, and $b_1^{(i)}:=\mathrm{dim} H_1(F_{i},\mathbb{R})<\infty$.
Let $\tilde{F}_{i}$ be the universal cover of $F_{i}$, and equip $\tilde{M}_{i}\overset{\mathrm{diff}}{\sim} N_{i} \times \tilde{F}_{i}$ with the natural pull-back metric from $M_{i}$.
We fix a $\tilde{p}_{i}\in \tilde{M}_{i}$ which projects to $p_{i}$ by the covering map.
Let $\hat{F}_{i}:=(\tilde{F}_{i}/[\Gamma_{i},\Gamma_{i}])/T_{i}$, where $T_{i}$ is the torsion subgroup of $\Gamma_{i}/[\Gamma_{i},\Gamma_{i}]$, then $H_{i}=(\Gamma_{i}/[\Gamma_{i},\Gamma_{i}])/T_{i}$ is a finitely generated free Abelian group with $\mathrm{rank} (H_{i})=b_{1}^{(i)}$, and $\hat{F}_{i}/H_{i}=F_{i}$.
We equip $\hat{M}_{i}\overset{\mathrm{diff}}{\sim} N_{i} \times \hat{F}_{i}$ with the quotient metric from $\tilde{M}_{i}$.
Let $\hat{p}_{i}\in \hat{M}_{i}$ be the point that is the image of $\tilde{p}_{i}$ under the quotient map.

\begin{claim}\label{claim5.2}
There exists a subsequence of $f_{i}:M_{i}\rightarrow N_i$ (here we use the same notation to denote the subsequence) as above, a sequence $\mu_{i}\downarrow 0$ and a positive integer $k_3\geq k$, so that the following holds:
if we rescale the metrics on $M_{i}$ and $N_{i}$ by $\mu_{i}^{-2}$, and still use $M_{i}$, $\tilde{M}_{i}$, $\hat{M}_{i}$ and $N_{i}$ to denote the above Riemannian manifolds with rescaled metrics, use $f_{i}$ to denote $\mu_{i}^{-1}f_{i}$, and use the same notation $F_{i}$, $\tilde{F}_{i}$, $\hat{F}_{i}$, $\Gamma_{i}$, $H_{i}$ as above, then we have
\begin{description}
  \item[(P1)] the fibration map $f_{i}:M_{i}\rightarrow N_{i}$ is a $\Psi_{1}(\delta_{i})$-Gromov-Hausdorff approximation;
  \item[(P2)] $\sec_{M_{i}}\geq -\Psi_{1}(\delta_{i})$;
  \item[(P3)] \begin{align}\label{3.2}
  \Lip f_{i}\leq C_{0},
  \end{align}
\begin{align}\label{3.4}
|\Delta {f}_{i}|_{L^{\infty}}\leq \Psi_{1}(\delta_{i}),
\end{align}
and for any $L\geq 1$, $1\leq s,t \leq k$,
\begin{align}\label{3.3}
\bbint_{B_{L}(p_{i})}|\langle\nabla {f}_{i}^{s}, \nabla {f}_{i}^{t} \rangle- \delta_{st}|\leq \Psi_{1}(\delta_{i});
\end{align}
  \item[(P4)] we have an equivariant Gromov-Hausdorff convergence in the following commutative diagram:
\begin{align}\label{diag2.1}
\xymatrix@C=2.5cm{
  (\hat{M}_{i}, \hat{p}_{i}, H_{i}) \ar[d]_{\pi_{i}} \ar[r]^{\mathrm{GH}} & (\mathbb{R}^{k_{3}}, 0^{k_{3}}, G) \ar[d]^{\pi_{\infty}} \\
  ({M}_{i}, p_{i}) \ar[r]^{\mathrm{GH}} & (\R^{k},0^{k}).  }
\end{align}

\end{description}
\end{claim}

\begin{proof}[Proof of Claim \ref{claim5.2}]

We take $\rho_{i}=(\Psi(\delta_{i}))^{\frac{1+n}{2n}}$.
Firstly, we rescale the metrics on $M_{i}$ and $N_{i}$ by $\rho_{i}^{-2}$.
Up to a subsequence, we may assume the following commutative diagram holds:

\begin{align}\label{diag3.1}
\xymatrix@C=2.5cm{
  (\rho_{i}^{-1}\hat{M}_{i}, \hat{p}_{i}, H_{i}) \ar[d]_{\pi_{i}} \ar[r]^{\mathrm{GH}} & (\mathbb{R}^{k_{1}}\times\hat{Y}, (0^{k_{1}},\hat{y}_{\infty}), H) \ar[d]^{\pi_{\infty}} \\
  (\rho_{i}^{-1}{M}_{i}, p_{i}) \ar[r]^{\mathrm{GH}} & (\R^{k},0^{k}),  }
\end{align}
where $\hat{Y}$ cannot split off an $\mathbb{R}$-factor.

By Cheeger-Colding's theory, we have $k_{1}\leq n$.
On the other hand, it is standard that $\pi_{\infty}$ is a submetry (hence $k_{1}\geq k$), and $H$ acts on $\mathbb{R}^{k_{1}-k}\times\hat{Y}$ transitively (where we view $\R^{k_{1}}$ as $\R^{k}\times\R^{k_{1}-k}$).

If $\hat{Y}$ contains more than one point, then $\hat{y}_{\infty}$ is a $k_{2}$-regular point of $\hat{Y}$ for some $k_{2}\in\mathbb{Z}^{+}$
(for otherwise the transitivity of $H$ on $\mathbb{R}^{k_{1}-k}\times\hat{Y}$ will imply that there exists no regular point on $\mathbb{R}^{k_{1}}\times\hat{Y}$, contradicting to Cheeger-Colding's theory).
Thus we can take a sequence of $\epsilon_{j}\downarrow0$ such that $(\mathbb{R}^{k_{1}}\times \epsilon_{j}^{-1}\hat{Y},(0^{k_{1}},\hat{y}_{\infty}))\xrightarrow{\mathrm{GH}}(\mathbb{R}^{k_{1}+k_{2}},0^{k_{1}+k_{2}})$.

For each $j$, take sufficiently large $i_{j}$ such that $d_{\mathrm{GH}}(B_{1}^{\rho_{i_{j}}^{-1}\hat{M}_{i_{j}}}(\hat{p}_{i_{j}}),B_{1}((0^{k_{1}},\hat{y}_{\infty})))< \epsilon_{j}^{2}$, $\max_{1\leq s, t\leq k}\bbint_{B_{1}(p_{i})}|\langle\nabla {f}_{i_j}^{s}, \nabla {f}_{i_j}^{t} \rangle- \delta_{st}|< \epsilon_{j}^{2+n}$, $|\Delta{f}_{i_j}|_{L^{\infty}}< \epsilon_{j}^{2}$, and $f_{i_{j}}: \rho_{i_{j}}^{-1}M_{i_{j}} \rightarrow B_{\rho_{i_{j}}^{-1}}(0^{k})$ is an $\epsilon_{j}^{2}$-Gromov-Hausdorff approximation (here we use $f_{i_{j}}$ to denote $\rho_{i_{j}}^{-1}f_{i_{j}}$ for short).
The existence of such a subsequence $\{i_{j}\}$ is ensured by (\ref{1.2}), (\ref{1.3}) and the choice of $\rho_{i}$.

Denote by $\mu_{j}:=\epsilon_{j}\rho_{i_{j}}\rightarrow0$, and consider the rescale manifolds $\mu_{j}^{-1}M_{i_{j}}$.
After a relabeling of index, we use $M_{i}$ to denote this rescaled sequence, and use $N_{i}$, $\tilde{M}_{i}$, $\hat{M}_{i}$ to denote the corresponding rescaled sequences, and use $f_{i}$ to denote $\mu_{i}^{-1}f_{i}$.
Then it is direct to check that (P1)-(P4) hold for suitable $\Psi_{1}(\delta)$.

If $\hat{Y}$ in (\ref{diag3.1}) consists of only one point, then we just take $\mu_{i}=\rho_{i}$.
This completes the proof of Claim \ref{claim5.2}.
\end{proof}

\textbf{In the remainder of Section \ref{sec-3}, we will use the notations as in Claim \ref{claim5.2}.}

Let $D_{i}\overset{\triangle}{=}\mathrm{diam}F_{i}$. Obviously, $D_{i}\rightarrow 0$.

\begin{rem}\label{rem-com-diam}
In (\ref{diag2.1}), since $H_i$ are Abelian groups, $G$ is also Abelian.
It is standard to check that $G$ acts on $\R^k$-factor trivially, and acts on $\mathbb{R}^{k_{3}-k}$ transitively (where we view $\R^{k_{3}}$ as $\R^{k}\times\R^{k_{3}-k}$).
Thus $G$ is isomorphic to $\R^{k_{3}-k}$, and each element of $G$ acts as a translation on $\mathbb{R}^{k_{3}-k}$.

In (\ref{diag2.1}), the fibrations $f_{i}$ can be lifted to fibrations $\hat{f}_{i}: \hat{M}_{i}\rightarrow N_{i}$.
It is obvious that, $\Lip \hat{f}_{i}=\Lip {f}_{i}\leq C_{0}$, and each $\hat{f}_{i}$ is $H_{i}$-invariant.
Up to a subsequence, we assume $\hat{f}_{i}$ converges in locally uniformly sense to a Lipschitz function $\hat{f}_{\infty}:\R^{k_{3}}\rightarrow \R^{k}$, then $\hat{f}_{\infty}$ is $G$-invariant.
It is not hard to check that there exists some $\varpi\in \mathrm{Isom}(\R^{k})$ such that $\hat{f}_{\infty}=\varpi\circ\pi_{\infty}$.

\textbf{Without loss of generality, in the following we assume}
\begin{description}
  \item[(P5)] $\varpi=\id$.
\end{description}
\end{rem}

\begin{proof}[Proof of (1) in Theorem \ref{mainthm-fiber-sec}]
Suppose (1) in Theorem \ref{mainthm-fiber-sec} does not hold, then we have a contradicting sequence, and slowly blow up it to obtain $f_{i}:M_{i}\rightarrow N_i$, $F_{i}$, $\hat{F}_{i}$, $H_{i}$ as in Claim \ref{claim5.2}.
By the contradicting assumption, we have $\mathrm{rank} (H_{i})\geq n-k+1$.
Because $H_{i}$ acts isometrically on $\hat{F}_{i}$ with $\hat{F}_{i}/H_{i}=F_{i}$,
by a standard argument, $H_{i}$ is generated by $H_i(3D_i)$.
Then according to Lemma \ref{lem:rank non-dec}, $\mathrm{rank}(G)\geq n-k+1$.
However, by Remark \ref{rem-com-diam}, $\mathrm{rank}(G)\leq n-k$, which is a contradiction. This completes the proof of  (1) in Theorem \ref{mainthm-fiber-sec}.
\end{proof}

\begin{rem}
The regularity properties of $f$ (i.e. those in (P3)) are not used in the proof of (1) in Theorem \ref{mainthm-fiber-sec}, but they will be used in the proof of (2) of Theorem \ref{mainthm-fiber-sec}.
\end{rem}

\begin{proof}[Proof of (2) in Theorem \ref{mainthm-fiber-sec}]
Suppose (2) in Theorem \ref{mainthm-fiber-sec} does not hold, then by slowly blowing up a contradicting sequence as in Claim \ref{claim5.2}, we have $f_{i}:M_{i}\rightarrow N_i$, $F_{i}$, $\hat{F}_{i}$, $H_{i}$ satisfying (P1)-(P5), and $\mathrm{rank} (H_i)= b$ (we assume $1 \leq b\leq n-k$, since the case $b=0$ is trivial),
but $F_{i}$ does not fiber over $T^{b}$.

Denote $r=k_{3}-k$, where $k_{3}$ is from $(\ref{diag2.1})$.
Recall that $G$ is isomorphic to $\R^{r}$, and each element of $G$ acts as a translation on $\mathbb{R}^{r}$.
Lemma \ref{lem:rank non-dec} implies that $r\geq b$.
We take a lattice $G_{0}\cong\mathbb{Z}^{r}$ of $G$ such that the flat torus $T^{r}:=\R^{r}/G_{0}$ satisfies $1\leq \diam(T^{r})\leq 4n$, and $\mathrm{inj}_{T^{r}}\geq 1$.

By (\ref{diag2.1}) and Lemma \ref{LocalAction}, up to choosing a subsequence of $i$, there exist $R_{i}\rightarrow \infty$ and pseudo-groups $\hat{K}_{i}\subset H_{i}(2R_{i}+\lambda_{i}')$
(where $\lambda_{i}'\downarrow 0$) such that
the following commutative diagram holds:
  \begin{align}\label{diag5.3}
\xymatrix@C=2.5cm{
  (B_{R_{i}}(\hat{p}_{i}), \hat{p}_{i}, \hat{K}_{i})) \ar[d]_{\sigma_{i}} \ar[r]^{\mathrm{GH}} & (\R^{k+r}, 0^{k+r}, G_{0}) \ar[d]^{\sigma} \\
  (B_{R_{i}}(\hat{p}_{i})/\hat{K}_{i}, \bar{p}_{i}) \ar[d]_{\tau_{i}} \ar[r]^{\mathrm{GH}} & (\R^{k}\times T^{r},\bar{p}) \ar[d]^{\tau},  \\
    (B_{R_{i}}(p_{i}), {p}_{i}) \ar[r]^{\mathrm{GH}} & (\R^{k}, 0^{k}),
    }
\end{align}
with $\pi_{i}|_{B_{R_{i}}(\hat{p}_{i})}=\tau_{i}\circ\sigma_{i}$.
In (\ref{diag5.3}), $\tau:\R^{k}\times T^{r}\rightarrow \R^{k}$ is the projection map.

Since $G/G_{0}$ is compact, by (\ref{diag2.1}), (\ref{diag5.3}) and (7) in Lemma \ref{LocalAction}, up to a subsequence, there exists a discrete Abelian closed subgroup $K_{i}<  \mathrm{Isom}(\tau_{i}^{-1}(B_{\frac{1}{5}R_{i}}(p_{i})))$, such that $\tau_{i}^{-1}(B_{\frac{1}{5}R_{i}}(p_{i}))/K_{i}$ is isometric to $B_{\frac{1}{5}R_{i}}(p_{i})$, and the following commutative diagram holds:
  \begin{align}\label{4.5}
\xymatrix@C=2.5cm{
  (\tau_{i}^{-1}(B_{\frac{1}{5}R_{i}}(p_{i})), \bar{p}_{i}, K_{i}) \ar[d]_{\tau_{i}} \ar[r]^{\mathrm{GH}} & (\R^{k}\times T^{r},\bar{p}, K) \ar[d]^{\tau} \\
  (B_{\frac{1}{5}R_{i}}(p_{i}), p_{i}) \ar[r]^{\mathrm{GH}} & (\R^{k}, 0^{k}),  }
\end{align}
where $K< \mathrm{Isom}(\R^{k}\times T^{r})$ coincides with the image of $G/G_{0}\rightarrow \mathrm{Isom}(\R^{k}\times T^{r})$.
Thus $K$ is isomorphic to a $r$-torus, acting trivially on the $\R^{k}$-factor, and acting transitively on the $T^{r}$-factor, and $T^{r}/K'$ is also a flat $r$-torus for any discrete subgroup $K'< K$.

The non-degenerate map $f_{i}|_{B_{\frac{1}{5}R_{i}}(p_{i})}$ can be lifted to
$\bar{f}_{i}:\tau_{i}^{-1}(B_{\frac{1}{5}R_{i}}(p_{i}))\rightarrow \R^k$.
Then
\begin{align}\label{4.314}
\Lip \bar{f}_{i}\leq C_{0},
\end{align}
\begin{align}\label{4.312}
|\Delta \bar{f}_{i}|_{L^{\infty}}\leq \Psi(\delta_{i}).
\end{align}
In addition, by \cite{HH24}*{Lemma 5.3} and (\ref{3.3}), for any $L\geq 1$, $1\leq \alpha, \beta\leq k$,
\begin{align}\label{4.313}
\bbint_{B_{L}(\bar{p}_{i})}|\langle\nabla \bar{f}_{i}^{\alpha}, \nabla \bar{f}_{i}^{\beta} \rangle- \delta_{\alpha\beta}|\leq \Psi(\delta_{i}).
\end{align}
Up to a subsequence, it is not hard to check that, $\bar{f}_{i}$ converges in locally uniformly sense to $\tau$ (recall that (P5) holds).

Let $\Lambda_{i}:\tau_{i}^{-1}(B_{\frac{1}{5}R_{i}}(p_{i}))\rightarrow \R^{k}\times T^{r}$ be a $\delta_{i}'$-Gromov-Hausdorff approximation given in (\ref{4.5}).
Since each $f_{i}$ is a $\Psi_{1}(\delta_{i})$-Gromov-Hausdorff approximation, and each $\hat{f}_{i}$ is $H_{i}$-invariant, it is direct to verify that, the map $\Lambda_{i}'=(\bar{f}_{i}, P_{2}\circ \Lambda_{i})$ is a $\delta_{i}''$-Gromov-Hausdorff approximation (where $\delta_{i}''\downarrow0$) compatible with the isometric actions in (\ref{4.5}), where $P_{2}:\R^{k}\times T^{r}\rightarrow T^{r}$ is the natural projecting map.

\begin{claim}\label{claim5.5} For any $x\in \R^{k}$ and for sufficiently large $i$, the extrinsic diameter of $\bar{f}_{i}^{-1}(x)$ is at most $4n+1$.
\end{claim}

\begin{proof}[Proof of Claim \ref{claim5.5}]
Take any $x\in \R^{k}$. For sufficiently large $i$ and any $\eta>0$, choose $y_{1},y_{2}\in \bar{f}_{i}^{-1}(x)$ so that $d(y_{1},y_{2})>\diam(\bar{f}_{i}^{-1}(x))-\eta$.
By the Gromov--Hausdoff distance estimate $|d_{i}(y_{1},y_{2})-d_{T^{r}}(P_{2}(\Lambda_{i}'(y_{1})),P_{2}(\Lambda_{i}'(y_{2})))|\leq \delta_{i}''$, and the arbitrariness of $\eta$, we know that $\diam(\bar{f}_{i}^{-1}(x))\leq 4n+1$.
\end{proof}
Now we fix some large $L_{0}>100n$.

\begin{claim}\label{claim5.6}
For any $x\in B_{L_{0}}(0^{k})$ and large $i$, $\bar{f}_i^{-1}(x)$ is a compact  connected manifold without boundary, and $\bar{f}_{i}^{-1}(B_{L_{0}}(0^{k}))$ is diffeomorphic to $B_{L_{0}}(0^{k})\times \bar{f}_{i}^{-1}(0^{k})$.
\end{claim}

\begin{proof}[Proof of Claim \ref{claim5.6}]
For any $x\in B_{L_{0}}(0^{k})$, by the non-degeneracy of $\bar{f}_{i}$, $\bar{f}_{i}^{-1}(x)\cap (\tau_i^{-1}(B_{L}({p}_{i})))$ is an $(n-k)$-dimensional manifold.
On the other hand, since $\diam(\bar{f}_{i}^{-1}(x))\leq 4n+1$, $\bar{f}_{i}^{-1}(x)$ is contained in $\overline{B_{2L_{0}}(\hat{p}_{i})}/\hat{K}_i$, which is compact in $B_{R_{i}}(\hat{p}_{i})/\hat{K}_i$.
Thus $\bar{f}_{i}^{-1}(x)$ is a compact manifold without boundary.
Combining with the non-degeneracy of $\bar{f}_{i}$, as well as the simply connectedness of $B_{L_{0}}(0^{k})$, it is standard that $\bar{f}_{i}^{-1}(B_{L_{0}}(0^{k}))$ is diffeomorphic to $B_{L_{0}}(0^{k})\times \bar{f}_{i}^{-1}(0^{k})$.

Note that $\bar{f}_i ^{-1}(0^k) \subset B_{4n+2}(\bar p_i)\subset \bar{f}_i^{-1}(B_{12n}(0^k))\overset{\mathrm{diff}}{\sim} B_{12n}(0^{k})\times \bar{f}_{i}^{-1}(0^{k})$.
If $\bar{f}_i ^{-1}(0^k)$ is disconnected, then $\bar{f}_i^{-1}(B_{12n}(0^k))$ is also disconnected, and hence $B_{4n+2}(\bar p_i)$ is contained in a connected component of $\bar{f}_i^{-1}(B_{12n}(0^k))$, contradicting the fact that $\bar{f}_i ^{-1}(0^k) \subset B_{4n+2}(\bar p_i)$.
\end{proof}

Since $\tau_{i}^{-1}(B_{L_{0}}(p_{i}))$ has a sectional curvature lower bound, it satisfies the $(\Phi_{0}', 1; k,\delta)$-generalized Reifenberg condition for some $\Phi_{0}'$. Then we apply Proposition \ref{prop-eq-fibration} and Corollary \ref{cor7.2-smooth} to the equivariant Gromov-Hausdorff approximations given by $\Lambda_{i}'$.
We conclude that there exist open sets $U_{i}'\subset \tau_{i}^{-1}(B_{L_{0}}(p_{i}))$, a sequence of smooth surjective maps $\bar{\Xi}_{i}:U_{i}'\rightarrow V'$ (where $V':=B_{\frac{1}{10}L_{0}}(0^{k})\times T^{r}$) which are both $\epsilon_{i}'$-equivariant Gromov-Hausdorff approximations (where $\epsilon_{i}'\downarrow 0$) and fibrations, and there exist  homomorphisms $\psi_{i}:K_{i}\rightarrow K$ so that
\begin{align}\label{4.9}
\bar{\Xi}_{i}\circ g= \psi_{i}(g)\circ \bar{\Xi}_{i} \text{ for every }g\in K_{i},
\end{align}
and
\begin{align}\label{4.91}
\tau\circ\bar{\Xi}_{i}=\bar{f}_{i}.
\end{align}
Denote $K_{i}'=\psi_{i}(K_{i})$, which is a discrete subgroup of $K$.
Then for every large $i$, $V'/K_{i}'$ is isometric to $B_{\frac{1}{10}L_{0}}(0^{k})\times T_{i}^{r}$, where $T_{i}^{r}:=T^{r}/K_{i}''$ is a flat $r$-torus.
By (\ref{4.9}), there is a well-defined map $\hat{\Xi}_{i}:U_{i}'/K_{i}\rightarrow B_{\frac{1}{10}L_{0}}(0^{k})\times T_{i}^{r}$.
By (\ref{4.91}), we have $\hat{\tau}_{i}\circ\hat{\Xi}_{i}={f}_{i}$, where $\hat{\tau}_{i}:B_{\frac{1}{10}L_{0}}(0^{k})\times T_{i}^{r}\rightarrow B_{\frac{1}{10}L_{0}}(0^{k})$ is the projection map.
Obviously $\hat{\Xi}_{i}$ is non-degenerated, which gives a fibration, and $\hat{\Xi}_{i}|_{f_{i}^{-1}(0^{k})}$ gives a fibration from $F_{i}$ onto $T_{i}^{r}$.
Finally, since the fibration $\bar{\Xi}_{i}:U_{i}'\rightarrow V'$ is an $\epsilon_{i}'$-equivariant Gromov-Hausdorff approximation, similar to the last part of the proof of Claim \ref{claim5.6}, the fiber of $\bar{\Xi}_{i}$ is connected.
The fiber of $\hat{\Xi}_{i}$ is also connected since it is the quotient of that of $\bar{\Xi}_{i}$ by a suitable group.
We get a contradiction, and this completes the proof of (2) in Theorem \ref{mainthm-fiber-sec}.
\end{proof}

\begin{rem}\label{rem5.4}
Theorem \ref{mainthm-fiber-sec} can be generalized to the following theorem.

\begin{theorem}\label{thm-Betti-number-rigidity}
Given $n\ge 2$, and $r_{0}>1$, and a function $\Phi_{0}:\mathbb{R}^{+}\rightarrow \mathbb{R}^{+}$ with $\lim_{s\rightarrow 0^{+}}\Phi_{0}(s)=0$, there exists a $\delta_0$ depending only on $n$, $r_{0}$, $\Phi_{0}$ such that the following holds for every $\delta\in (0,\delta_0)$.
Suppose $M$, $N$ are compact Riemannian manifolds with dimension $n$, $k$ respectively (where $k\le n$), such that
$\Ric_{M}\geq -1$,  $d_{\mathrm{GH}}(M,N)\le\delta$, and (\ref{5.2-001})  holds.
Suppose in addition that for any $p\in M$, $\widetilde{B_{p}(1)}$ satisfies the $(\Phi_{0}, 1; k,\delta)$-generalized Reifenberg condition.
Then there exists a smooth fibration map $f:M\rightarrow N$ which is a $\Psi_{1}(\delta|n,r_0)$-Gromov-Hausdorff approximation and satisfies the followings:
	\begin{enumerate}
		\item for any $q\in N$,
		\begin{align}
			b_1\overset{\triangle}{=}\mathrm{dim} H_1(f^{-1}(q),\mathbb{R})\leq n-k;
		\end{align}
		\item  $f^{-1}(q)$ fibers over $T^{b_1}$ with connected fibers.
	\end{enumerate}
	
\end{theorem}

In fact, since $\widetilde{B_{p}(1)}$ satisfies the $(\Phi_{0}, 1; k,\delta)$-generalized Reifenberg condition for any $p\in M$, according to \cite{HH24}*{Proposition 5.2},  $M$ satisfies the $(\Phi_{0}', 1; k,\delta)$-generalized Reifenberg condition for some $\Phi_{0}'$. Then we have a smooth fibration $f:M \rightarrow N$ according to Theorem \ref{Fibrations}.
Moreover, by  (\ref{5.2-001}) , $B_{\frac{1}{2}}(q)\subset N$ is simply-connected.
Hence, the universal cover of $f^{-1}(B_{\frac{1}{2}}(q))$, which is a subset of  $\widetilde{B_{p}(1)}$,  is differmorphic to $B_{\frac{1}{2}}(0^k)\times\widetilde{f^{-1}(q)} $.

Then we use a contradiction argument as in  the proof of Theorem \ref{mainthm-fiber-sec}.
Note that in the proof of Theorem \ref{mainthm-fiber-sec}, except for the use of Corollary \ref{cor7.2-smooth} at the last step, every other step uses only properties of Ricci curvature and the conclusion of Theorem \ref{Fibrations}.
In the last step,  to apply Corollary \ref{cor7.2-smooth}, we require that $\tau_{i}^{-1}(B_{L_{0}}(p_{i}))$  satisfies the generalized Reifenberg condition, which is implied by the assumption that $\widetilde{B_{p}(1)}$ satisfies the $(\Phi_{0}, 1; k,\delta)$-generalized Reifenberg condition for any $p\in M$.

We \textbf{conjecture} that the conclusion of Theorem \ref{thm-Betti-number-rigidity} holds even if $M$ satisfies the generalized Reifenberg condition.
\end{rem}

\section{Stability of rank of Abelian group action}\label{sec-5}

\subsection{Upper semi-continuity of the rank}

For an Abelian Lie group $G=\R^{l_1} \times \Z^{l_2} \times C$ where $C$ is compact, define $\mathrm{rank}(G)=l_1+l_2$.
\begin{lem}\label{lem:rank non-dec}
Let $(X_i,p_i,d_i,\meas_i,G_i)$ be a sequence of $\RCD(-1,N)$ spaces, $G_i$ be Abelian groups acting freely and properly discontinuously on $X_i$ such that $\mathrm{rank}(G_i)=k$ and $G_i$ is generated by $G_i(R)$ for some $R>0$. If
$$(X_i,p_i,d_i,\meas_i,G_i) \overset{\rm{GH}}{\longrightarrow} (X,p,d,\meas,G).$$ Then $\mathrm{rank}(G) \ge k$.

\end{lem}
\begin{proof}
The case $k=0$ is trivial. Assume the lemma holds for $k-1$ and $\mathrm{rank}(G_i)=k$. Let $H_i \cong \mathbb{Z}$ be a subgroup of $G_i$. Let $g_i$ be a generator of $H_i$ such that $d(g_ip_i,p_i)\le R$. Since the action of $G_i$ is properly discontinuous, there exists a minimal $l_i\in \N$ such that $d(g_i^{l_i}p_i,p_i)\ge R/2$, then it is easy to see $d(g_i^{l_i}p_i,p_i)\le 3R/2$.  Passing to a subsequence if necessary, $H_i$ converges to a subgroup $H$ of $G$ and $g_i^{l_i}$ converges to a nontrivial element $g\in H$. Since $g\in H$ has infinite order and its powers have unbounded displacement, $g$ generates a noncompact subgroup of $H$, then $\mathrm{rank}(H) \ge 1$.

Since the action of $G_i$ is free and properly discontinuous, it is easy to check that the natural induced isometric action of $G_{i}/H_{i}$ on $X_{i}/H_{i}$ is also free and properly discontinuous. In addition, $G_i/H_i$ is generated by $G_i/H_i(R)$, and $\mathrm{rank}(G_i/H_i)=k-1$.
Consider the subconvergent sequence
$$(X_i/H_i,\bar{p}_i,\bar{d}_i,\bar{\meas}_i,G_i/H_i) \overset{\rm{GH}}{\longrightarrow} (X/H,\bar{p},\bar{d},\bar{\meas},G/H).$$
By the induction hypothesis, $\mathrm{rank}(G/H) \ge k-1$. Thus $\mathrm{rank}(G)= \mathrm{rank}(H)+\mathrm{rank}(G/H) \ge k$.
\end{proof}

\begin{rem}
    Note that the strict inequality can occur in Lemma \ref{lem:rank non-dec}. Consider the flat torus $T^2=S^1\times S^1$ where each of $S^1$ factor has length $k\in \N$. We can define a free $G_k\defeq \Z/(k^2+1)\mb Z$ action on $T^2$ generated by the rotation $(e^{\frac{2\pi i}{k^2+1}}, e^{\frac{2k\pi i}{k^2+1}})$. Let $k\to\infty$, the $T^2$ converges to $\R^2$ and $G_k$ action converges to the $G\defeq\mb Z^2$ action by standard translations. So $\rank(G_k)=0$ and $\rank(G)=2$. Here the sequence has no uniform diameter bound. Compare our stability of rank Theorem \ref{thm:rank-stability}.

    Meanwhile, also note that bounded displacement assumption on the generating set of $G_i$ is necessary. Take $\R$ with $G_i=\Z$ action generated by translating $i$ units: $x\to x+i$. Let $i\to \infty$, the $G_i=\Z$ action converges to the trivial action on $\R$ and the rank drops from $1$ to $0$.
\end{rem}

\subsection{Proof of Theorem \ref{thm:rank-stability}}

\begin{lem}\label{lem-surj-ghomo}
Let $({X}_{i},{d}_{i},{m}_{i}, p_{i})$ be a sequence of compact $\RCD(-1,N)$ spaces, and let $(\hat{X}_{i},\hat{d}_{i},\hat{m}_{i}, \hat{p}_{i})$ be the covering space of $({X}_{i},{d}_{i},{m}_{i},p_i)$ with deck transformation group ${G}_{i}$.
Suppose $(\hat{X}_{i},\hat{d}_{i},{G}_{i}, \hat{p}_{i})\xrightarrow{\mathrm{GH}} (\hat{Y}_{\infty},\hat{d}_{\infty},{G}_{\infty}, \hat{p}_{\infty})$ for some compact $\RCD$-space $\hat{Y}_{\infty}$, where the convergence is given by a sequence of $\epsilon_{i}$-equivariant Gromov-Hausdorff approximations $f_{i}:\hat{X}_{i}\rightarrow \hat{Y}_{\infty}$, $\psi_{i}: G_{i}\rightarrow G_{\infty}$, where each $\psi_{i}$ is a homomorphism, and $f_i(\hat{p}_{i})=\hat{p}_{\infty}$.
Then for sufficiently large $i$, there is a surjective homomorphism $\Psi_{i}:\pi_{1}(X_{i})\rightarrow \pi_{1}(\hat{Y}_{\infty}/\psi_{i}(G_{i}))$.
In particular, $\mathrm{rank}(\pi_{1}(X_{i}))\geq \mathrm{rank}(\pi_{1}(\hat{Y}_{\infty}/\psi_{i}(G_{i})))$.
\end{lem}

\begin{proof}
The proof is a modification of \cite{SormaniWei2001} and \cite{Tuschman1995}.
According to \cite{Wang2024}, there exists some $\delta>0$ such that for every $y\in \hat{Y}_{\infty}$, every loop contained in $B_{\delta}(y)$ is null-homotopic in $\hat{Y}$.
Now we fix a sufficiently large $i$ such that $\epsilon_{i}< \frac{\delta}{100}$.
Given any $\alpha\in \pi_{1}(X_{i},p_{i})$ represented by a loop $\gamma:[0,1]\rightarrow X_{i}$ with $\gamma(0)=\gamma(1)=p_{i}$, we lift $\gamma$ to a curve $\hat{\gamma}:[0,1]\rightarrow \hat{X}_{i}$ with $\hat{\gamma}(0)=\hat{p}_{i}$, $\hat{\gamma}(1)= g_{i}(\hat{p}_{i})$ for some $g_{i} \in \hat{G}_{i}$.
Then we take finitely many points $\{0 =t_{0}< t_{1} < \ldots< t_{m}=1\} \subset [0,1]$ such that for every $1\leq j\leq m$, $\gamma([t_{j-1}, t_{j}])$ is contained in a ball of radius $\frac{\delta}{50}$.
Note that $f_{i}(\hat{\gamma}(1))$ and $\psi_{i}(g_{i})(\hat{p}_{\infty})$ can be  connected by a geodesic contained in a ball of radius $\frac{\delta}{40}$.
Let $\varphi(\hat{\gamma})$ be a curve in $\hat{Y}_{\infty}$ given by sequentially concatenating  minimizing geodesics from $f_{i}(\hat{\gamma}(t_{j-1}))$ to $f_{i}(\hat{\gamma}(t_{j}))$ and  from $f_{i}(\hat{\gamma}(1))$ to $\psi_{i}(g_{i})(\hat{p}_{\infty})$.
Then $\sigma_{i}(\varphi(\hat{\gamma}))$ is a loop in $\hat{Y}_{\infty}/\psi_{i}(G_{i})$ passing through $\hat{y}_i:=\sigma_{i}(\hat{p}_{\infty})$, where $\sigma_{i}: \hat{Y}_{\infty}\rightarrow \hat{Y}_{\infty}/\psi_{i}(G_{i})$ is the natural projection.
Define $\Psi_{i}(\alpha)$ as the element in $\pi_{1}(\hat{Y}_{\infty}/\psi_{i}(G_{i}),\hat{y}_i)$ represented by $\sigma_{i}(\varphi(\hat{\gamma}))$.

Given any loop $\gamma$ as above, two different choices of $\{t_{0}< t_{1} < \ldots< t_{m}\}$ will give curves uniformly $\frac{\delta}{20}$-close to each other in $\hat{Y}_{\infty}$, and by the semi-locally simply connectedness, the resulting $\varphi(\gamma)$ are easy to see to homotopic rel. endpoints in $\hat{Y}_{\infty}$, and after projecting to $\hat{Y}_{\infty}/\psi_{i}(G_{i})$, they give the same element in $\pi_{1}(\hat{Y}_{\infty}/\psi_{i}(G_{i}))$.
Similarly, if two loops $\gamma_{1},\gamma_{2}:[0,1]\rightarrow X_{i}$  with $\gamma_1(0)=\gamma_1(1)=\gamma_2(0)=\gamma_2(1)=p_{i}$ are homotopic rel. endpoints, then their lifts $\hat{\gamma}_{1},\hat{\gamma}_{2}$ are homotopic rel. endpoints, then by the semi-locally simply connectedness of $\hat{Y}$, it is easy to construct a homotopy rel. endpoints from $\varphi(\hat{\gamma}_{1})$ to $\varphi(\hat{\gamma}_{2})$, which gives the same element in $\pi_{1}(\hat{Y}_{\infty}/\psi_{i}(G_{i}),\hat{y}_i)$.
Thus $\Psi_{i}([\alpha])$ is well-defined, and $\Psi_{i}$ is a group homomorphism.

For any element $\beta\in \pi_{1}(\hat{Y}_{\infty}/\psi_{i}(G_{i}),\hat{y}_i)$, we take a loop $\hat{\eta}_i:[0,1]\rightarrow \hat{Y}_{\infty}/\psi_{i}(G_{i})$ representing it.
Because $\sigma_i:\hat{Y}_{\infty}\rightarrow\hat{Y}_{\infty}/\psi_{i}(G_{i})$ is a submetry,  we can lift $\eta$ to a curve $\hat{\eta}:[0,1]\rightarrow \hat{X}_{i}$ with $\hat{\eta}(0)=\hat{p}_{\infty}$, $\hat{\eta}(1)= \psi_{i}(g_{i})(\hat{p}_{\infty})$ for some $g_{i} \in G_{i}$.
Then we take finitely many points $\{0 =t_{0}< t_{1} < \ldots< t_{m}=1\} \subset [0,1]$ such that for every $1\leq j\leq m$, $\hat{\eta}([t_{j-1}, t_{j}])$ is contained in a ball of radius $\frac{\delta}{40}$.
Take $\hat{x}_1=\hat{p}_i$, $\hat{x}_j\in f_{i}^{-1}(\hat{\eta}(t_{j}))$ for $j>1$, and let $\hat{\gamma}$ be a curve in $\hat{X}_{i}$ given by sequentially concatenating  minimizing geodesics from $x_{i-1}$ to $x_{i}$ and  from $x_{m}$ to $g_{i}(\hat{p}_{i})$.  Suppose  $\alpha\in \pi_{1}(X_{i},p_{i})$ is represented by the loop obtained by projecting $\hat{\gamma}$  onto $X_i$, then by the semi-locally simply connectedness of $\hat{Y}_{\infty}$ and the above construction, it is easy to see $\Psi_{i}(\alpha)=\beta$. In conclusion, $\Psi_{i}$ is surjective.
This completes the proof of the lemma.
\end{proof}

\begin{proof}[Proof of Theorem \ref{thm:rank-stability}]
Without loss of generality, we assume $D=1$, and $(X_{i},p_{i},d_{i})\xrightarrow{\mathrm{GH}}(Z,p_{\infty},d_{\infty})$. then $\diam (Z)\leq 1$, and the following commutative diagram holds:

\begin{align}\label{diag6.1111}
\xymatrix@C=2.5cm{
  (\hat{X}_{i}, \hat{p}_{i}, H_{i}) \ar[d]_{P_{i}} \ar[r]^{\mathrm{GH}} & (Y, \hat{p}, H) \ar[d]^{P_{\infty}} \\
  ({X}_{i}, p_{i}) \ar[r]^{\mathrm{GH}} & (Z,p_{\infty}),  }
\end{align}

Since each $H_i$ is generated by $H_i(3)$, according to Lemma \ref{lem:rank non-dec}, we have $ k_1+k_2\geq k$.  In the following we will prove $k_1+k_2 \leq k$.

We fix a discrete subgroup $\Z^{k_2}$ of $\R^{k_2}$, and denote by $H_0:=\Z^{k_1}\times \Z^{k_2}\times \{e\} < H$, and $\bar{Y}:=Y/H_{0}$.

By (\ref{diag6.1111}) and Lemma \ref{LocalAction}, up to choosing a subsequence of $i$, there exist $R_{i}\rightarrow \infty$ and pseudo-groups $\hat{K}_{i}\subset H_{i}(2R_{i}+\lambda_{i}')$
(where $\lambda_{i}'\downarrow 0$) such that the following commutative diagram holds:
  \begin{align}\label{diag6.31111}
\xymatrix@C=2.5cm{
  (B_{R_{i}}(\hat{p}_{i}), \hat{p}_{i}, \hat{K}_{i})) \ar[d]_{\sigma_{i}} \ar[r]^{\mathrm{GH}} & (Y, \hat{p}, H_{0}) \ar[d]^{\sigma} \\
  (B_{R_{i}}(\hat{p}_{i})/\hat{K}_{i}, \bar{p}_{i}) \ar[d]_{\tau_{i}} \ar[r]^{\mathrm{GH}} & (\bar{Y},\bar{p}) \ar[d]^{\tau},  \\
    (B_{R_{i}}(p_{i}), {p}_{i}) \ar[r]^{\mathrm{GH}} & (Z, p_{\infty}),
    }
\end{align}
with $P_{i}|_{B_{R_{i}}(\hat{p}_{i})}=\tau_{i}\circ\sigma_{i}$.

Since $H/H_{0}$ is compact, by (\ref{diag6.1111}), (\ref{diag6.31111}) and (7) in Lemma \ref{LocalAction}, up to a subsequence, there exists a discrete closed Abelian subgroup $K_{i}<  \mathrm{Isom}(\tau_{i}^{-1}(B_{\frac{1}{5}R_{i}}(p_{i})))$, such that $\tau_{i}^{-1}(B_{\frac{1}{5}R_{i}}(p_{i}))/K_{i}$ is isometric to $B_{\frac{1}{5}R_{i}}(p_{i})$, and the following commutative diagram holds:
  \begin{align}\label{diag6.51111}
\xymatrix@C=2.5cm{
  (\tau_{i}^{-1}(B_{\frac{1}{5}R_{i}}(p_{i})), \bar{p}_{i}, K_{i}) \ar[d]_{\tau_{i}} \ar[r]^{\mathrm{GH}} & (\bar{Y},\bar{p}, K) \ar[d]^{\tau} \\
  (B_{\frac{1}{5}R_{i}}(p_{i}), p_{i}) \ar[r]^{\mathrm{GH}} & (Z, p_{\infty}).  }
\end{align}
where $K< \mathrm{Isom}(\bar{Y})$ coincides with the image of $H/H_{0}\rightarrow \mathrm{Isom}(\bar{Y})$.
Since $\diam(X_{i})\leq 1$ and $\diam(\bar{Y})< \infty$, without loss of generality, we assume that for every $i$, $B_{\frac{1}{5}R_{i}}(p_{i})=X_{i}$, and $\bar{X}_i:=\tau_{i}^{-1}(X_{i})$ are compact $\RCD$ spaces with uniformly bounded diameter.

According to \cite{MRW08}*{Lemma 3.2}, we assume the equivariant Gromov-Hausdorff approximation in (\ref{diag6.51111}) is given by a group homomorphism $\psi_{i}:K_{i}\rightarrow K$.
Denote $K_{i}'=\psi_{i}(K_{i})$, which is a discrete subgroup of $K$.
Choose suitable discrete groups $\hat{K}'_{i}$ such that $H_{0}<\hat{K}'_{i}<H$ and the image of $\hat{K}'_{i}$ under the homomorphism $H\rightarrow H/H_0\rightarrow \mathrm{Isom}(\bar{Y})$ is $K_{i}'$.
Thus $\bar{Y}/K'_{i}$ is isometric to $Y/\hat{K}'_{i}$.

Let $\pi_1:H \rightarrow \R^{k_1}\times \Z^{k_2}$ be the projection map.
Denote $C_i:=\hat{K}'_{i}\cap \pi_1^{-1}(0^{k_1+k_2}) $.

\begin{claim} \label{claim6.4}
$\hat{K}'_{i}/C_i$ acts effectively and freely on $Y/C_i$.
\end{claim}

\begin{proof}[Proof of the Claim \ref{claim6.4}:]
Recall that the action of $\hat{K}'_{i}/C_i$ on $Y/C_i$ is given by:  for any $k\in \hat{K}'_{i}$, $y\in Y$, $[k]([y]):=[k(y)]\in Y/C_i$.
If $[k]([y])=[y]$ for some $y\in Y$, then there exists $h_y\in  C_{i}$ such that $h_y^{-1}k(y)=y$.
Thus $h_y^{-1}k$ belongs to a compact subgroup of $H$, which implies $\pi_1(h_y^{-1}k)=0^{k_1+k_2}$.
Thus $k\in C_i$, i.e., $[k]$ is the identity element.
\end{proof}
    It is easy to see that $\hat{K}'_{i}/C_i$ is isomorphic to $\Z^{k_1}\times \Z^{k_2}$.
Note that $(Y/C_i)/(\hat{K}'_{i}/C_i)$ is isometric to $\bar{Y}/K'_{i}$.
By Claim \ref{claim6.4}, it is standard that   $b_{1}(\bar{Y}/K'_{i})\geq \rank(\hat{K}'_{i}/C_i)=k_1+k_2$.

Finally, according to Lemma \ref{lem-surj-ghomo}, we have $k=b_{1}(X_{i})\geq b_{1}(\bar{Y}/K'_{i})\geq k_1+k_2$.

The proof of Theorem \ref{thm:rank-stability} is completed.
\end{proof}

\begin{rem}
If the condition $H_i = \pi_1(X_i,p_i) \big/ [\pi_1(X_i,p_i), \pi_1(X_i,p_i)]$ in Theorem  \ref{thm:rank-stability} is replaced by $H_i = (\pi_1(X_i,p_i) \big/ [\pi_1(X_i,p_i), \pi_1(X_i,p_i)])/T_i$, where $T_i$ is the torsion subgroup of $\pi_1(X_i,p_i) \big/ [\pi_1(X_i,p_i), \pi_1(X_i,p_i)]$, then we still have $b_{1}(X_{i})=\rank(H_i)=k$, and the same proof of Theorem  \ref{thm:rank-stability} still holds.
\end{rem}

\subsection{Alternative proof of Theorem \ref{thm:rank-stability}: via Abelian progressions}
We present an alternative proof of Theorem~\ref{thm:rank-stability}. The material in this subsection will not be used in other sections.

\subsubsection{Groupification of local groups}
In this subsection, we explain how local groups determine a covering space in the $\RCD$ setting.
Let $(X,p,d,\mathfrak{m})$ be an $\RCD(-1,N)$ space. By \cite{MondinoWei2019}, any covering space $(\hat{X},\hat{p}, \hat{d},\hat{\mathfrak{m}})$ of $X$, with the induced metric and measure, is again an $\RCD(-1,N)$ space.

Assume $\diam (X) \le D$ and there exists an Abelian transformation group $\Gamma$ on $\hat{X}$ such that $X=\hat{X}/\Gamma$. Define the local group $$\Gamma(20D)=\{ g \in \Gamma \mid d(\hat{p},g\hat{p}) \le 20D\}.$$

$\Gamma(20D)$ is a pseudo-group; that is, for some $g_1,g_2 \in \Gamma(20D)$, the product $g_1g_2$ may not lie in $\Gamma(20D)$. To address this issue, we define the groupification $\hat{\Gamma}$ of $\Gamma(20D)$ as follows.

Let $F$ be the free group generated by elements $e_g$ for each $g \in \Gamma(20D)$. Take the quotient of $F$ by the normal subgroup generated by all elements of the form $e_{g_1}e_{g_2}e_{g_1g_2}^{-1}$, where $g_1,g_2 \in \Gamma(20D)$ and $g_1g_2 \in \Gamma(20D)$. The resulting quotient group is denoted by $\hat{\Gamma}$. Let $\tilde{\Gamma}$ be the Abelianization of $\hat{\Gamma}$.

There is a natural (pseudo-group) homomorphism
\begin{align*}
\iota : \Gamma(20D) \to \tilde{\Gamma}, \qquad \iota(g)=[e_g],
\end{align*}
where $[e_g]$ denotes the equivalence class of $e_g$ under the quotient map.

Define
\begin{align*}
\pi:\tilde{\Gamma} \to \Gamma, \quad
\pi([e_{g_1}e_{g_2}\cdots e_{g_k}])=g_1g_2\cdots g_k .
\end{align*}
Then $\pi \circ \iota$ is the identity on $\Gamma(20D)$; in particular, $\iota$ is injective. Since $\Gamma(20D)$ generates $\Gamma$, the map $\pi$ is surjective.

We say that $\Gamma(20D)$ determines $\Gamma$ if $\pi$ is an isomorphism. Roughly speaking, this means that all relations in $\Gamma$ can already be detected inside $\Gamma(20D)$.

We now consider the case where
$\Gamma = H = \pi_1(X,p)/[\pi_1(X,p),\pi_1(X,p)]$,
the maximal Abelian transformation group. We will show that $H(20D)$ determines $H$.

The proof follows the construction in \cites{FukayaYamaguchi1992,SantosZamoraPi1,Wang2023}. For the reader's convenience, we include a brief sketch of the argument.

\begin{lem}\label{lem:det}
Let $(\hat{X},\hat{p},H)$ be the covering space of $X$ with deck transformation group $H=\pi_1(X,p)/[\pi_1(X,p),\pi_1(X,p)]$. Then $H(20D)$ determines $H$.
\end{lem}
\begin{proof}
We construct a space by gluing via an equivalence relation:
\begin{align*}
\tilde{H} \times_{H(20D)} \bar{B}_{10D}(\hat{p})
= (\tilde{H} \times \bar{B}_{10D}(\hat{p})) / \sim ,
\end{align*}
where the equivalence relation is defined by
\begin{align*}
(\tilde{g}\iota(g),x) \sim (\tilde{g},gx)
\end{align*}
for all $g \in H(20D)$, $\tilde{g} \in \tilde{H}$, and $x \in \bar{B}_{10D}(\hat{p})$ such that $gx \in \bar{B}_{10D}(\hat{p})$.
We endow $\tilde{H} \times_{H(20D)} \bar{B}_{10D}(\hat{p})$ with the induced length metric.

Then $\tilde{H} \times_{H(20D)} \bar{B}_{10D}(\hat{p})$ is a covering space of $\hat{X}$ with deck transformation group $\ker(\pi)$, where $\pi:\tilde{H}\to H$.

Moreover, $\tilde{H} \times_{H(20D)} \bar{B}_{10D}(\hat{p})$ is also a covering space of $X$ with Abelian deck transformation group $\tilde{H}$. By the universal property of Abelianization for $H$, $\ker \pi$ must be trivial, then the map $\pi$ must be injective and hence an isomorphism.
\end{proof}

We have the following lemma for quotient groups, whose proof is ignored here.

\begin{lem}\label{lem6.7}
Assume $(\hat{X},\hat{p},\Gamma)$ is a covering space of $X$ and $\Gamma(20D)$ determines the Abelian group $\Gamma$. Let $\Gamma'$ be a subgroup of $\Gamma$ such that $\Gamma'(20D)$ generates $\Gamma'$. Then $(\Gamma/\Gamma')(20D)$ determines $\Gamma/\Gamma'$.
\end{lem}

If local groups determine global groups, we have the following theorem from \cite{Wang2023}; see also \cite{FukayaYamaguchi1992}*{Theorem 3.10}.

\begin{theorem}\label{thm:G_0}
Fix $N,D,R>0$. Let $(X_i,p_i,d_i,\meas_i)$ be a sequence of $\RCD(-1,N)$ spaces with $\diam(X_i) \le D$. Suppose that $(\hat{X}_i,\hat{p}_i,H_i)$ is a normal covering space of $X_i$ with deck transformation group $H_i$. Assume that each $H_i$ is Abelian and that $H_i(20D)$ determines $H_i$.
Suppose that
$(\hat{X}_i,\hat{p}_i,H_i) \xrightarrow{\mathrm{GH}} (Y,\hat{p},H)$,
and that there exists a closed normal subgroup $H' \vartriangleleft H$ such that $H/H'$ is discrete and $H'$ is generated by $H'(R)$.

Then there exists a sequence of normal subgroups $H_i' \vartriangleleft H_i$ such that $H_i'$ converges to $H'$, $H_i/H_i' \cong H/H'$ for all sufficiently large $i$, and $H_i'$ is generated by $H_i'(R+\epsilon_i)$, where $\epsilon_i \to 0$. Moreover, $H/H'$ is finitely presented.

\end{theorem}

Notice that the original statements in \cites{Wang2023,FukayaYamaguchi1992} require that $\hat{X}_i$ be simply connected; this assumption is used only to ensure that $H_i(20D)$ determines $H_i$.

\subsubsection{Abelian progressions and relations in local groups}

In this subsection, we use the structure of Abelian progressions to study local groups. The main references for this subsection are \cites{BGT,Zamora2020,NSZ2025}.

Assume that a discrete Abelian group $\Gamma_i$ acts freely and isometrically on a metric space $(\hat{X}_i,\hat{p}_i)$. Suppose that we have the equivariant Gromov--Hausdorff convergence
\begin{align*}
(\hat{X}_i,\hat{p}_i,\Gamma_i) \xrightarrow{\mathrm{GH}} (\hat{X},\hat{p},\Gamma = \mathbb{R}^{l} \times C),
\end{align*}
where $C$ is a compact Lie group. Let $\phi_i$ be an equivariant Gromov--Hausdorff approximation from $\Gamma_i$ to $\Gamma$. In the following, we will focus on the groups and omit the base spaces on which they act.

Let $\pi: \Gamma \to \mathbb{R}^{l}$ denote the projection map. Then $\pi \circ \phi_i$ is a sequence of good approximations from $\Gamma_i$ to $\mathbb{R}^{l}$; see \cite{NSZ2025}*{Section 1.2} for the definition of good approximations.

Since any non-trivial subgroup of $\mathbb{R}^{l}$ is non-compact, for all sufficiently large $i$ we can find a maximal subgroup $S_i \vartriangleleft \Gamma_i$ such that $\pi \circ \phi_i(S_i)$ converges to the identity element. By quotienting out $S_i$ if necessary, we may assume that there is no non-trivial subgroups $S_i$ of $\Gamma_i$ such that $\pi \circ \phi_i(S_i)$ converging to $e \in \mathbb{R}^l$.

Next, we can apply the structure theory of approximate groups to
$\pi \circ \phi_i : \Gamma_i \to \mathbb{R}^l$,
see \cite{BGT} and \cite{Zamora2020}*{Theorem 7.1}.
For any $\delta > 0$ and all sufficiently large $i$, the set $\Gamma_i(\delta)$ is isomorphic to the grid part of an Abelian progression of rank $l$.

Roughly speaking, we can find $l$ elements
\begin{align*}
\{g_{i,1}, g_{i,2}, \dots, g_{i,l}\} \subset \Gamma_i(\varepsilon_i), \quad \varepsilon_i \to 0,
\end{align*}
such that they generate $\Gamma_i(\delta)$ and are linearly independent with respect to the relations in $\Gamma_i(\delta)$.
More precisely, if $m_s \in \mathbb{Z}$ for $1 \le s \le l$ and at least one $m_s \neq 0$, then for any product
$\prod_{t=1}^T (g_{t,1} g_{t,2} g_{t,3}^{-1})$,
where $g_{t,1}, g_{t,2}, g_{t,3} \in \Gamma_i(\delta)$ satisfy $g_{t,1} g_{t,2} = g_{t,3}$, it is impossible to reduce
$\prod_{t=1}^T (g_{t,1} g_{t,2} g_{t,3}^{-1})$
to the form $\prod_{s=1}^l g_{i,s}^{m_s}$ using only commutations of adjacent elements and cancellation of the form $g g^{-1}$ with $g \in \Gamma_i(\delta)$.

Notice that in \cites{BGT,Zamora2020}, $\delta$ is required to be sufficiently small so that the $\delta$-ball of the limit group contains no non-trivial subgroup. However, in our setting, since we use approximations to $\mathbb{R}^l$, which has no non-trivial compact subgroups, $\delta$ can be taken arbitrarily.

\subsubsection{Proof of Theorem \ref{thm:rank-stability}}
We prove Theorem \ref{thm:rank-stability} in this subsection.

Since $\mathrm{rank}(H_i) = k$, we can find elements
\begin{align*}
\{g_{i,s} \mid 1 \le s \le k\} \subset H_i(3D)
\end{align*}
that are linearly independent. Let $K_i$ denote the subgroup generated by these elements.

Passing to a subsequence if necessary, we may assume that $K_i$ converges to a subgroup $K \vartriangleleft H$.

\begin{lem}\label{lemma:com quotient}
The quotient group $H/K$ is compact.
\end{lem}
\begin{proof}
Let $\bar{X}_i = \hat{X}_i / K_i$ and $\bar{H}_i = H_i / K_i$.
Since $H_i$ and $K_i$ have the same rank, the group $\bar{H}_i$ consists only of torsion elements.

By \cite{Wang2023}*{Theorem 3.1}, passing to a subsequence if necessary, we have
\begin{align*}
(\bar{X}_i, \bar{p}_i, \bar{H}_i) \xrightarrow{\mathrm{GH}} (\bar{Y}, \bar{p}, \bar{H} = H/K).
\end{align*}

Since $H_i(20D)$ determines $H_i$ and $K_i$ is generated by $K_i(3D)$,
it follows from Lemma \ref{lem6.7} that $\bar{H}_i(20D)$ determines $\bar{H}_i$.
By Theorem \ref{thm:G_0}, we can find a subgroup $\bar{H}_i' \le \bar{H}_i$ converging to $\bar{H}_0$, the identity component of $\bar{H}$, such that
$\bar{H}_i / \bar{H}_i' \cong \bar{H} / \bar{H}_0$.
Since $\bar{H}_i$ is finite, it follows that $\bar{H} / \bar{H}_0$ is compact.

Assume that $\bar{H} = H / K$ is non-compact. Since $\bar{H}$ has finitely many components,
$\bar{H} = \mathbb{R}^{l} \times C$,
where $l > 0$ and $C$ is compact. Let $\phi_i: \bar{H}_i \to \bar{H}$ be an eGHA, and let
$\pi: \bar{H} \to \mathbb{R}^{l}$
denote the projection map. Then the composition
$\pi \circ \phi_i: \bar{H}_i \to \mathbb{R}^{l}$
forms a sequence of good approximate groups.

After quotienting out subgroups if necessary, we may assume that $\pi \circ \phi_i$ has no subgroups $S_i \vartriangleleft \bar{H}_i$ such that $\pi \circ \phi_i (S_i)$ converges to $e \in \mathbb{R}^l$. Using the structure theory of approximate groups, for all sufficiently large $i$, the local group $\bar{H}_i(20D)$ coincides with the grid part of an Abelian progression of rank $l$. Since $\bar{H}_i(20D)$ determines $\bar{H}_i$, it follows that $\bar{H}_i \cong \mathbb{Z}^{l}$ for sufficiently large $i$. This contradicts the fact that $\bar{H}_i$ consists only of torsion elements.
\end{proof}

Since $H / K$ is compact and $Y / H$ is compact, it follows that $Y / K$ is also compact. In particular, $\hat{X}_i / K_i$ is a normal covering space of $X_i$ with bounded diameter. By choosing $D$ large enough if necessary, we may assume
\begin{align*}
\operatorname{diam}(\hat{X}_i / K_i) \le D.
\end{align*}

Let $G_i$ denote the maximal torsion-free subgroup of $H_i$ containing $K_i$. Then
$H_i = G_i \times T_i$,
where $T_i$ is the torsion subgroup of $H_i$. Since
\begin{align*}
\operatorname{diam}(\hat{X}_i / G_i) \le \operatorname{diam}(\hat{X}_i / K_i) \le D,
\end{align*}
it follows that $G_i(20D)$ determines $G_i$.

Assume that
\begin{align*}
(\hat{X}_i, \hat{p}_i, G_i) \xrightarrow{\mathrm{GH}} (Y, \hat{p}, G).
\end{align*}

Since $H / G$ is compact, we have
\begin{align*}
G \cong \mathbb{R}^{k_1} \times \mathbb{Z}^{k_2} \times C',
\end{align*}
where $C'$ is compact. Let
$\phi_i: G_i \to G$
denote an eGHA.

Since $G_i(20D)$ determines $G_i$, it follows from Theorem \ref{thm:G_0} that we can find a subgroup $A_i \vartriangleleft G_i$ such that
\begin{align*}
A_i \xrightarrow{\mathrm{GH}} A \cong \mathbb{R}^{k_1} \times C'  \quad \text{and} \quad G_i / A_i \cong G / (\mathbb{R}^{k_1} \times C') \cong \mathbb{Z}^{k_2}.
\end{align*}

Since $G_i$ is torsion-free, the short exact sequence
\begin{align*}
1 \longrightarrow A_i \longrightarrow G_i \longrightarrow G_i / A_i \longrightarrow 1
\end{align*}
splits. Hence we can find a subgroup $B_i \cong \mathbb{Z}^{k_2}$ of $G_i$ such that
\begin{align*}
G_i = A_i \times B_i.
\end{align*}

The generators of $B_i$ are eGH close to the generators of $G / (\mathbb{R}^{k_1} \times C')$, so we may assume that they lie in $B_i(20D)$. The generators of $A$ lie in a bounded set; by taking $D$ larger if necessary, we may assume that $A_i(20D)$ generates $A_i$.

By Lemma \ref{lem6.7}, $A_i(20D)$ determines $A_i \cong G_i / B_i$, and $B_i(20D)$ determines $B_i \cong G_i / A_i$. After passing to a subsequence, we may assume that
\begin{align*}
B_i \xrightarrow{\mathrm{GH}} B \vartriangleleft G.
\end{align*}

\begin{lem}\label{discrete}
We have
$A \cap B = \{\mathrm{Id}\}$.
In particular, $B \cong \mathbb{Z}^{k_2}$ and
$G = A \times B.$
\end{lem}

\begin{proof}
Suppose, for contradiction, that there exists $g \neq \mathrm{Id}$ in $A \cap B$.
Then there exist sequences $a_i \in A_i$ and $b_i \in B_i$ converging to $g$. In particular, $b_i \neq \mathrm{Id}$ for all sufficiently large $i$.

Consider the elements
$b_i a_i^{-1}$,
which converge to $\mathrm{Id}$. By the definition of $A_i$, $b_i a_i^{-1} \in A_i$, and hence $b_i \in A_i$. This contradicts the fact that $B_i \cap A_i = \{\mathrm{Id}\}$.

Therefore, $A \cap B = \{\mathrm{Id}\}$. Since $A$ contains $G_0$ and $A \cap B = \{\mathrm{Id}\}$, the subgroup $B$ is discrete and satisfies
\begin{align*}
B \cong G / A \cong \mathbb{Z}^{k_2}.
\end{align*}
It follows that
$G = A \times B$.
\end{proof}

\begin{proof}[Proof of Theorem \ref{thm:rank-stability}]
We have
\begin{align*}
k = \mathrm{rank}(H_i) = \mathrm{rank}(G_i) = \mathrm{rank}(A_i) + \mathrm{rank}(B_i).
\end{align*}
Since $B_i \cong B \cong \mathbb{Z}^{k_2}$, it remains to show that
$\mathrm{rank}(A_i) = k_1$.

Consider $ A_i \xrightarrow{\mathrm{GH}} A=\mathbb{R}^{k_1} \times C'$. We proved before that $A_i(20D)$ determines $A_i$. By the same proof of Lemma \ref{lemma:com quotient} (last two paragraphs),
$\mathrm{rank}(A_i) = k_1$.
\end{proof}

\begin{cor}\label{cor: compact torsion}
Let $T_i$ be the torsion subgroup of $H_i$, and assume that
\begin{align*}
T_i \xrightarrow{\mathrm{GH}} T.
\end{align*}
Then $T$ is compact.
\end{cor}

\begin{proof}
Suppose, for contradiction, that $T$ is non-compact. Consider the quotient groups
\begin{align*}
H_i / T_i \xrightarrow{\mathrm{GH}} H / T.
\end{align*}
Then
\begin{align*}
\mathrm{rank}(H_i / T_i) = k > \mathrm{rank}(H / T) = k - \mathrm{rank}(T),
\end{align*}
which contradicts Lemma \ref{lem:rank non-dec}. Therefore, $T$ must be compact.
\end{proof}

\subsection{An application}

In what follows, the rectifiable dimension of an $\mathrm{RCD}$ space $(X,d,\mathfrak{m})$ will be denoted by $\operatorname{dim}(X)$. 
As an application of Theorem \ref{thm:rank-stability}, we give a new proof of Theorem \ref{thm:betti-stability} below, which was established by Zamora \cite{Zamora2022} and Zamora-Santos-Rodriguez \cite{SantosZamoraPi1}.

\begin{theorem}\label{thm:betti-stability}
	Let $D>0$ and let $(X_i,p_i,d_i,\mathfrak{m}_i)$ be a sequence of $\mathrm{RCD}\bigl(-1,N\bigr)$ spaces with a uniform diameter bound $\operatorname{diam}(X_i)\le D$. If $(X_i,p_i,d_i,\mathfrak{m}_i)$ converges in the pointed measured Gromov-Hausdorff  sense to $(X,p,d,\mathfrak{m})$, then for all sufficiently large $i$,
	\begin{align*}
	b_1(X_i)-b_1(X)\;\le\;\operatorname{dim}(X_i)-\operatorname{dim}(X).
	\end{align*}
\end{theorem}

\begin{proof}
Passing to a subsequence if necessary, we may assume that $b_1(X_i)=k$ for some fixed integer $k$.

Let
$H_i = \pi_1(X_i,p_i)\big/[\pi_1(X_i,p_i),\pi_1(X_i,p_i)]$
be the Abelianization of $\pi_1(X_i,p_i)$, and let $(\hat{X}_i,\hat{p}_i,\hat{d}_i,H_i)$ denote the normal covering space of $(X_i,p_i,d_i)$ with deck transformation group $H_i$. Passing to a subsequence if necessary, Theorem~\ref{thm:rank-stability} yields the following diagram:
\begin{align*}
\xymatrix@C=2.5cm{
  (\hat{X}_i, \hat{p}_i, \hat{d}_i, H_i) \ar[d]_{\pi} \ar[r]^{\mathrm{GH}} & (Y, \hat{p}, d, H = \mathbb{R}^{k_1} \times \mathbb{Z}^{k_2} \times C) \ar[d]^{\pi} \\
  (X_i, p_i) \ar[r]^{\mathrm{GH}} & (X,p)  }
\end{align*}
where $C$ is compact and $k=k_1+k_2$.

Note that $\mathbb{R}^{k_1}\times \{0\} \times C \cong \mathbb{R}^{k_1}\times C$ is a subgroup of $H$. We next show that
$H / (\mathbb{R}^{k_1}\times C) \cong \mathbb{Z}^{k_2}$
acts freely on
$\overline{Y} \defeq Y / (\mathbb{R}^{k_1}\times C)$.
It follows that $\overline{Y}$ is a normal covering of $X = Y/G$ with deck transformation group $\mathbb{Z}^{k_2}$.

Suppose that a nontrivial element $\bar{g} \in \mathbb{Z}^{k_2}$ fixes some $\bar{y} \in \overline{Y}$. Then there exist lifts $g \in H$ and $y \in Y$ such that $\pi(g)=\bar{g}$, $\pi(y)=\bar{y}$, and $gy = hy$
for some $h \in \mathbb{R}^{k_1}\times C \subset H$. In particular, $gh^{-1}$ fixes $y$. Since $g$ projects to a nontrivial element in $\mathbb{Z}^{k_2}$, it follows that $gh^{-1}$ has infinite order, so the cyclic group $\langle gh^{-1} \rangle$ is noncompact. This contradicts the fact that the isotropy group at $y$ is compact. Therefore, the $\mathbb{Z}^{k_2}$-action is free.

By the universal property of Abelianization, $b_1(X) \ge k_2$. Let $\dim X = l$. Taking a sequence $s_i \to \infty$ and a point $\hat{q} \in Y$ such that $q = \pi(\hat{q}) \in X$ is regular, we obtain the following diagram:
\begin{align*}
\xymatrix@C=2.5cm{
  (s_i Y, \hat{q}, H) \ar[d]_{\pi} \ar[r]^{\mathrm{GH}} & (\R^l\times Z, (0,z) ,G) \ar[d]^{\pi} \\
  (s_i X, q) \ar[r]^{\mathrm{GH}} & (\R^l,0)  }
\end{align*}
Then $G = \mathbb{R}^{k'} \times C'$, where $C'$ is compact, and $G$ acts transitively on $Z$.
By \cite{PanYe2024}*{Corollary 3.3}, $G$ contains a closed subgroup isomorphic to $\mathbb{R}^{k_1}$. In particular, $k' \ge k_1$.

Since $Z$ is $\mathrm{RCD}(0,n-l)$ and $G$ acts transitively (hence cocompactly) on $Z$, the splitting theorem implies that $Z$ splits off an $\mathbb{R}^{k'}$-factor. Consequently,
\begin{align*}
\dim Z \ge k' \ge k_1,
\end{align*}
and hence
\begin{align*}
\dim Y \ge k_1 + l = k_1 + \dim X.
\end{align*}

For sufficiently large $i$, by the lower semicontinuity of the rectifiable dimension (cf.~\cite{Kitabeppu2019}), we have
\begin{align*}
\dim(X_i) = \dim(\hat{X}_i) \ge \dim Y.
\end{align*}
Therefore, a direct computation yields
\begin{align*}
    b_1(X) \ge k_2
    = b_1(X_i) - k_1
    \ge b_1(X_i) + \dim(X) - \dim(Y)
    \ge b_1(X_i) + \dim(X) - \dim(X_i),
\end{align*}
as desired.
\end{proof}

\section{Regular maps var center of mass}\label{sec-6}

The following lemma will be frequently used in this section.

\begin{lem}\label{lem-6.1}
   Suppose $(X,d,m)$ is an $\RCD(-\delta, N)$ space, $B_p(10)\subset X$. Suppose $f_1, f_2:\Omega\rightarrow\R$ are two Lipschitz functions belonging to the domain of Laplacian, and satisfy
   \begin{align}
       |\Delta f_i|_{L^{\infty}(B_{10}(p))}\leq L, \quad |\nabla f_i|_{L^{\infty}(B_{10}(p))}\leq L, \nonumber
   \end{align}
   \begin{align}
|f_1(x)-f_2(x)|\leq \delta \nonumber
\end{align}
for every $x\in \overline{B_{10}(p)}$.
Then
\begin{align}
\bbint_{B_{10}(p)}|\nabla(f_1-f_2)|^{2}\leq\Psi(\delta|L, N).\nonumber
\end{align}

\end{lem}
\begin{proof}
By the Gauss-Green formula on the $\RCD$-spaces (see \cite{BPS23}*{Theorem 5.2}),
\begin{align}
&\bbint_{B_{10}(p)}|\nabla(f_1-f_2)|^{2} dm\nonumber\\
=&-\bbint_{B_{10}(p)}(f_1-f_2) \Delta(f_1-f_2)dm- \frac{1}{\mathrm{m}(B_{10}(p))}\int_{\partial B_{10}(p)}(f_1-f_2) (\nabla(f_1-f_2)\cdot \nu_{E})_{\mathrm{int}}d\mathrm{Per}\nonumber\\
\leq &L\max_{B_{10}(p)}|f_1-f_2|+2L\frac{\mathrm{Per}(\partial B_{10}(p))}{m(B_{10}(p))}\max_{\partial B_{10}(p)}|f_1-f_2|\nonumber\\
\leq&\Psi(\delta),\nonumber
\end{align}
where we use $|(\nabla(f_1-f_2)\cdot \nu_{E})_{\mathrm{int}}|_{L^{\infty}(\partial B_{10}(p),\mathrm{Per}) }\leq|\nabla(f_1-f_2)|_{L^{\infty}(B_{10}(p),m)}\leq 2L$  in the first inequality and use the volume comparison in the last inequality.
\end{proof}

\subsection{Existence of regular maps and fibrations}

\begin{prop}\label{fibration-blowup}
Given $n\in \Z^+ $. Suppose $(\tilde{X}, \tilde{d}_{\tilde{X}}, \tilde{m})$ is a $\RCD(-\delta,n)$ space and $(\tilde{N}, \tilde{h})$ is a $k$-dimensional Riemannian manifolds $(k\leq n)$,
and

\begin{align}\label{5.8}
|\Rm_{\tilde{N}}|\leq \delta,
\end{align}
\begin{align}\label{5.9}
|\nabla^{m} \Rm_{\tilde{N}}|\leq 1, \quad m=1,2,3,4,
\end{align}
\begin{align}\label{5.10}
\inj_{\tilde{N}}\geq \delta^{-1}.
\end{align}
Denote by $\tilde{Z}=\tilde{N}\times \R^s$  (where $s\in\{0\}\cup \mathbb{Z}^{+}$, $s\leq n-k$), and let $\pi_{1}:\tilde{Z}\rightarrow \tilde{N}$, $\pi_{2}:\tilde{Z}\rightarrow \R^s$ be the projection maps.
Let $\tilde{\Omega}$ be an open set in $\tilde{X}$.

Suppose that there exists a map $\tilde{\Phi}:T_{200}(\tilde{\Omega})\rightarrow \tilde{Z}$, such that
\begin{align}\label{5.11}
&\text{ for every }p\in T_{100}(\tilde{\Omega}), \text{ there exists a $\delta$-splitting harmonic map $u_p:B_{100}(p)\rightarrow \R^k$}\\
&\text{such that $u_p$ is $\delta$-close to $\exp_q^{-1}\circ \pi_1\circ \tilde{\Phi}$, where $q=\pi_1\circ \tilde{\Phi}(p)$.}\nonumber
\end{align}

Suppose in addition that $v:T_{200}(\tilde{\Omega})\rightarrow \R^{s}$ is a continuous map belonging to the domain of Laplacian with
\begin{align}\label{5.123323}
|\Delta v^{\alpha}|_{L^{\infty}}\leq \delta,\quad \forall \alpha\in\{1,\ldots ,s\},
\end{align}
\begin{align}\label{5.12322311}
|v(x)-\pi_{2}\circ \tilde{\Phi}(x)|\leq \delta,
\end{align}
and for any $x\in T_{100}(\tilde{\Omega})$, we have
\begin{align}\label{5.123223}
\bbint_{B_{100}(x)}|\langle\nabla v^{\alpha}, \nabla v^{\beta} \rangle- \delta_{\alpha\beta}|\leq \delta.
\end{align}

Then for every sufficiently small $\delta$, there exists a Lipschitz map $\Xi: T_{20}(\tilde{\Omega})\rightarrow \tilde{Z}$ such that $\pi_{2}\circ \Xi=v$, and $\Xi$ is $\Psi_{0}(\delta|n)$-close to $\tilde{\Phi}$.
In addition, for any $p\in \tilde{\Omega}$,  $\bar{\Xi}_{p}:=(\exp_{q}^{-1},\mathrm{id})\circ \Xi|_{B_{10}(p)}$ satisfies
\begin{align}\label{5.1231}
\Lip \bar{\Xi}_{p}\leq C(n),
\end{align}
\begin{align}\label{5.1233}
|\Delta \bar{\Xi}_{p}^{\alpha}|_{L^{\infty}(B_{10}(p))}\leq C(n),
\end{align}
\begin{align}\label{5.1232}
\bbint_{B_{10}(p)}|\langle\nabla \bar{\Xi}_{p}^{\alpha}, \nabla \bar{\Xi}_{p}^{\beta} \rangle- \delta_{\alpha\beta}|\leq \Psi_{0}(\delta|n)
\end{align}
for every $\alpha,\beta\in\{1,\ldots,k+s\}$, where $q=\pi_1\circ \Xi(p)$.
\end{prop}

\begin{prop}\label{fibration-smooth}
Suppose $\tilde{X}$ in Proposition \ref{fibration-blowup} is an $n$-dimensional manifold with $\Ric \geq -\delta$, then $\pi_1\circ \Xi$ is smooth.
Suppose in addition, $\tilde{X}$ satisfies the $(\Phi_{0}, 1; k+s,\delta)$-generalized Reifenberg condition (where $s\in\{0\}\cup \mathbb{Z}^{+}$, $s\leq n-k$), then if $\delta$ is sufficiently small (depending on $n$ and $\Phi_{0}$), $d\Xi$ is non-degenerated.

Furthermore, under the above assumptions, the following two conclusions hold respectively:
\begin{description}
  \item[(A)] if in addition $\tilde{\Omega}=B_{p_0}(L)$ for some $L$,  $\tilde{\Phi}:T_{200}(\tilde{\Omega})\rightarrow B_{(q_0,0^{s})}(L+200)\subset\tilde{Z}$ is a $\delta$-Gromov-Hausdorff approximation and $v(p_0)=0^s$, then there exists an open set $U$ such that $ U \subset \tilde{\Omega}\subset T_2(U)$ so that $\Xi(U)=B_{(q_0,0^{s})}(L-1)$, and $\Xi:U\rightarrow B_{(q_0,0^{s})}(L-1)$ is a fiberation and a $\Psi_{0}(\delta|n)$-Gromov-Haudorff approximation, and for any $p_{1},p_{2}\in U$ with $d_{\tilde{X}}(p_1,p_2)\leq1$, it holds
      \begin{equation}\label{BiHolder}
	(1-\Psi(\delta))d_{\tilde{X}}(p_{1},\Xi^{-1}(\Xi(p_{2})))^{1+\Psi(\delta)}\le d_{\tilde{Z}}(\Xi(p_{1}),\Xi(p_{2}))\le (1+\Psi(\delta))d_{\tilde{X}}(p_{1},\Xi^{-1}(\Xi(p_{2}))),
\end{equation}
where $\Psi(\delta)=\Psi(\delta|n,\Phi_{0})$.
  \item[(B)] if $s=0$,  $\tilde{\Omega}=\tilde{X}$, $\tilde{Z}=\tilde{N}$ are both compact manifolds, and $\tilde{\Phi}:\tilde{X}\rightarrow \tilde{N}$ satisfies
  \begin{align}\label{5.111}
\text{ for every }q\in \tilde{N}, \text{ there exists some }p\in\tilde{X} \text{ such that }d_{\tilde{N}}(q,\tilde{\Phi}(p))<\delta,
\end{align}
then $\Xi:\tilde{X}\rightarrow \tilde{N}$ is a surjective fibration and $\Psi_{0}(\delta|n)$-close to $\tilde{\Phi}$; and for any $p_{1},p_{2}\in \tilde{X}$ with $d_{\tilde{X}}(p_1,p_2)\leq1$, (\ref{BiHolder}) holds.
\end{description}
\end{prop}

\begin{proof}[Proof of Proposition \ref{fibration-blowup}]
Firstly, we construct $\Xi$ satisfying (\ref{5.1231})-(\ref{5.1232}) by the technique of center of mass.
Note that in this step, we do not require that $\tilde{M}$ satisfies the generalized Reifenberg condition.

In the proof, we use $\tilde{\Phi}_{1}=\pi_{1}\circ\tilde{\Phi}$, $\tilde{\Phi}_{2}=\pi_{2}\circ\tilde{\Phi}$ for short.
We use $C$, $C_{1}$, $C_{2}, \ldots$ (respectively, $\Psi$, $\Psi_{1}$, $\Psi_{2}, \ldots$) to denote positive constants (respectively, positive functions with $\lim_{s\downarrow0}\Psi(s)=0$ etc.) depending only on $n$.
The values of $C$ or $\Psi$ may change from line to line.

Since $(\tilde{X}, \tilde{d}, \tilde{m})$ is a $\RCD(-\delta,n)$ space and (\ref{5.123323}) and (\ref{5.123223}) hold, by \cite{HondaPeng2023}*{Corollary 3.3}, it is easy to see that if $\delta$ is sufficiently small (depending only on $n$), then
\begin{align}\label{5.123123}
\Lip v\leq C_{1}
\end{align}
holds on $T_{50}(\tilde{\Omega})$.

We find a maximal set $\mathcal{A}=\{p_{\lambda}|\lambda=1,\ldots,\Lambda\}\subset T_{50}(\tilde{\Omega})$ such that $d_{\tilde{X}}(p_{\lambda},p_{\lambda'})\geq 2$ for any $\lambda\neq\lambda'$.
Then $\{B_{2}(p_{\lambda})\}_{p_{\lambda}\in \mathcal{A}}$ covers $T_{50}(\tilde{\Omega})$.

Suppose $\delta$ is sufficiently small (depending only on $n$).
By volume comparison theorem, there exists $C_{2}>0$ so that, for any $p\in \tilde{\Omega}$, there exist at most $C_{2}$ points $\{p_{i_{1}},p_{i_{2}},\ldots,p_{i_{r}}\}\subset \mathcal{A}$, such that $p\in B_{4}(p_{l})$ for $l\in \{i_{1},\ldots,i_{r}\}$, and $p\notin B_{4}(p_{s})$ for $p_{s}\in \mathcal{A}\setminus \{p_{i_{1}},\ldots,p_{i_{r}}\}$.
In addition, for each $\lambda\in \{1,\ldots,\Lambda\}$, we fix a good cut-off function $\Psi_{\lambda}:\tilde{X}\rightarrow \R^{+}$ such that $\supp \Psi_{\lambda} \subset B_{4}(p_{\lambda})$, $\Psi_{\lambda}\equiv 1$ on $B_{2}(p_{\lambda})$, $|\nabla \Psi_{\lambda}|_{L^{\infty}}+|\Delta\Psi_{\lambda}|_{L^{\infty}}\leq C_{3}$.
Let $\varphi_{\lambda}=(\sum_{s=1}^{\Lambda}\Psi_{s})^{-1}\Psi_{\lambda}$. It is easy to see that
\begin{align}
\sum_{\lambda=1}^{\Lambda}\varphi_{\lambda}\equiv 1 \quad \text{on} \quad T_{50}(\tilde{\Omega}),
\end{align}
\begin{align}\label{5.12}
|\nabla \varphi_{\lambda}|_{L^{\infty}}+|\Delta\varphi_{\lambda}|_{L^{\infty}}\leq C_{4}.
\end{align}

Since (\ref{5.8})-(\ref{5.10}) hold, by Appendix \ref{sec-8}, for each $q\in \tilde{N}$,
the exponential map $E_{q}\overset{\Delta}{=}\exp_{q}:(B_{\delta^{-\frac{1}{4}}}(0^k),0^k)\to (B_{\delta^{-\frac{1}{4}}}(q),q)$,
and its inverse map, which is denoted by $F_{q}:(B_{\delta^{-\frac{1}{4}}}(q),q)\to (B_{\delta^{-\frac{1}{4}}}(0^k),0^k)$, are both $\Psi(\delta)$-Gromov-Hausdorff approximations and $(1+\Psi(\delta))$-Lipschitz maps, and under the normal coordinate $F_{q}$,

\begin{align}\label{5.160}
\left|\frac{\partial^{t} }{\partial y_{i_{1}}^{j_{1}}\ldots\partial y_{i_{t}}^{j_{t}}}d^{2}_{\tilde{N}}(y_{1},y_{2})\right|\leq {C}_{5}, \quad t=1, 2, 3, \quad i_{r}=1\text{ or }2,
\end{align}
\begin{align}\label{5.162}
\left|\frac{\partial^{2} }{\partial y_{1}^{i}\partial y_{1}^{j}}d^{2}_{\tilde{N}}(y_{1},y_{2})-2\delta_{ij}\right|\leq \Psi_{2}(\delta),
\end{align}
\begin{align}\label{5.161}
\left|\frac{\partial^{2} }{\partial y_{1}^{i}\partial y_{2}^{j}}d^{2}_{\tilde{N}}(y_{1},y_{2})+2\delta_{ij}\right|\leq \Psi_{2}(\delta),
\end{align}
for every $y_{1},y_{2}\in B_{500}(q)$.

Given any $p_{\lambda} \in \mathcal{A}$, we use $E_{\lambda}$ to denote $E_{\tilde{\Phi}_{1}(p_{\lambda})}$ and use $F_{\lambda}$ to denote the inverse of $E_{\lambda}$ for simplicity.

Because (\ref{5.11}) (\ref{5.123323}) (\ref{5.12322311}) and (\ref{5.123223}) hold, for each $\lambda$, there exists a map $u_\lambda: (B_{50}(p_\lambda),p_\lambda)\to (\R^k,0^k)$, so that $(u_{\lambda},v)|_{B_{50}(p_\lambda)}$ is $\Psi_{3}(\delta)$-close to $(F_\lambda,\mathrm{id}_{\mathbb{R}^s})\circ \tilde{\Phi}|_{B_{50}(p_\lambda)}$, and
\begin{align}\label{5.131}
\Delta u_\lambda^{i} =0,
\end{align}
\begin{align}\label{5.132}
\mathrm{Lip} u^{i}_\lambda \leq C_6,
\end{align}
\begin{align}\label{5.133}
\bbint_{B_{50}(p_{\lambda})}|\langle\nabla u_\lambda^{i}, \nabla u_\lambda^{j} \rangle- \delta_{ij}|\leq \Psi_{3}(\delta)
\end{align}
for $i,j\in\{1,\ldots,k\}$, and
\begin{align}\label{5.133567}
\bbint_{B_{50}(p_{\lambda})}|\langle\nabla u_\lambda^{i}, \nabla v^{j} \rangle|\leq \Psi_{3}(\delta)
\end{align}
for $i\in\{1,\ldots,k\}$, $j\in\{1,\ldots,s\}$,
where (\ref{5.133567}) can be obtained using a standard contradiction argument based on the theory of $W^{1,2}$ convergence.

Take $f_\lambda:=E_\lambda\circ u_\lambda$.
Note that $d_{\tilde{N}}(f_{\lambda}(x),\tilde{\Phi}_{1}(x))\leq \Psi(\delta)$ for any $x\in B_{50}(p_{\lambda})$.
Define an energy function $H:T_{20}(\tilde{\Omega})\times \tilde{N}\to\R$ by
\begin{align}
	H(x,y)=\frac{1}{2}\sum_{\lambda} \varphi_\lambda(x) d_{\tilde{N}}^{2}(f_\lambda(x),y).
\end{align}

By (\ref{5.8}), (\ref{5.10}) and the convex radius estimate, for any fixed $x\in T_{20}(\tilde{\Omega})$, $H(x,\cdot)$ is strictly convex on $B_5(\tilde{\Phi}_{1}(x))$.
Furthermore, because
$$H(x,\tilde{\Phi}_{1}(x))=\frac{1}{2}\sum_{\lambda} \varphi_\lambda(x) d_{\tilde{N}}^{2}(f_\lambda(x),\tilde{\Phi}_{1}(x))\leq \Psi(\delta),$$
and $H(x,y)>8$ for $y\notin B_5(\tilde{\Phi}_{1}(x))$, we can see that $H(x,\cdot)$ has a unique minimum point, which is denoted by $f(x)$. Then we define $\Xi(x)=(f(x),v(x))$.
It is easy to check that $f$ is $\Psi(\delta)$-close to $\tilde{\Phi}_{1}$ on $T_{20}(\tilde{\Omega})$.
Combing this fact with (\ref{5.12322311}), it is obvious that $\Xi$ is $\Psi(\delta)$-close to $\tilde{\Phi}$ on $T_{20}(\tilde{\Omega})$.

Given any fixed $p\in \tilde{\Omega}$, let $q=f(p)$. We will abuse the notation that $y=(y^{i})$ stands for both a point around $q$ and its coordinate with respect to the normal coordinate $F_{q}$.

Suppose that $\lambda \in \mathcal{A}$ satisfies $\varphi_{\lambda}(x)\neq 0$ for some $x\in B_{10}(p)$, then by the triangle inequality, $B_{10}(p)\subset B_{24}(p_{\lambda})$.
Hence $B_{\frac{21}{2}}(q)\subset B_{25}(\tilde{\Phi}_{1}(p_{\lambda}))$.
For any such $\lambda$, we consider the coordinate transformation matrix $\eta_{\lambda}:F_{\lambda}(B_{\frac{21}{2}}(q))\rightarrow F_{q}(B_{\frac{21}{2}}(q))$.
i.e. $\eta_{\lambda}(\vec{y})=F_{q}(E_{\lambda}(\vec{y}))$.
By Appendix \ref{sec-8},
we have
\begin{align}\label{5.43}
\left|\frac{\partial \eta_{\lambda}^{\nu}}{\partial y^{\xi}_{\lambda}}\right|\leq C_{7},
\quad
\left|\frac{\partial^{2}\eta_{\lambda}^{\xi}}{\partial y_{\lambda}^{\omega} \partial y_{\lambda}^{\theta}}\right|\leq C_{7}.
\end{align}
Moreover, the matrix $(\frac{\partial \eta_{\lambda}^{\nu}}{\partial y^{\xi}_{\lambda}})_{\nu\xi}$ can be written in the form $(B^{(\lambda)}+A^{(\lambda)})_{\nu\xi}$, where the matrix $B^{(\lambda)}\in \mathrm{SO}(k)$, while the components of the matrix $A^{(\lambda)}$ satisfy $|A^{(\lambda)}_{\nu\xi}|\leq \Psi_4(\delta)$.

Let $\tilde{f}_{\lambda}:B_{10}(p)\rightarrow \mathbb{R}^{k}$ be given by
\begin{align}\label{5.37}
\tilde{f}_{\lambda}=\eta_{\lambda}\circ u_{\lambda}=F_{q}\circ f_{\lambda},
\end{align}
then by the chain rule,
it is easy to see that
\begin{align}\label{5.5332}
|\Lip \tilde{f}_{\lambda}|\leq C_{8},
\end{align}
\begin{align}\label{5.5333}
|\Delta \tilde{f}_{\lambda}^{\nu}|_{L^{\infty}(B_{10}(p))}\leq C_{8},
\end{align}
and
\begin{align}\label{5.45}
&\bbint_{B_{10}(p)}\bigl| \langle\nabla \tilde{f}^{\alpha}_{\lambda},\nabla \tilde{f}^{\beta}_{\lambda}\rangle-\delta_{\alpha\beta}\bigr|\\
=&\bbint_{B_{10}(p)}\bigl|(A^{(\lambda)}_{\alpha\xi}+B^{(\lambda)}_{\alpha\xi}) (A^{(\lambda)}_{\beta\theta}+B^{(\lambda)}_{\beta\theta}) \langle\nabla u^{\xi}_{\lambda},\nabla u^{\theta}_{\lambda}\rangle-\delta_{\alpha\beta}\bigr|\nonumber\\
\leq&\bbint_{B_{10}(p)}\bigl|B^{(\lambda)}_{\alpha\xi} B^{(\lambda)}_{\beta\theta} \langle\nabla u^{\xi}_{\lambda},\nabla u^{\theta}_{\lambda}\rangle-\delta_{\alpha\beta}\bigr|+\Psi(\delta)\nonumber\\
\leq &\bbint_{B_{10}(p)}|B^{(\lambda)}_{\alpha\xi} B^{(\lambda)}_{\beta\theta}|\bigl| \langle\nabla u^{\xi}_{\lambda},\nabla u^{\theta}_{\lambda}\rangle-\delta_{\xi\theta}\bigr|+\Psi(\delta)\nonumber\\
\leq &\Psi(\delta).\nonumber
\end{align}

Since $f_{\lambda}$ is $\Psi(\delta)$-Gromov-Hausdorff close to $\tilde{\Phi}_{1}$, it is easy to see $\tilde{f}_{\lambda}$ is $\Psi(\delta)$-Gromov-Hausdorff close to $F_{q}\circ \tilde{\Phi}_{1}$.
Thus, for those $\lambda, \lambda'\in \mathcal{A}$ such that $\supp \varphi_{\lambda}\cap \supp \varphi_{\lambda'} \cap B_{10}(p)\neq \emptyset$, we have
\begin{align}
|\tilde{f}_{\lambda}(x)-\tilde{f}_{\lambda'}(x)|\leq \Psi(\delta)
\end{align}
for every $x\in B_{10}(p)$, and then by Lemma  \ref{lem-6.1}, we have
\begin{align}\label{5.46}
\bbint_{B_{10}(p)}|\nabla(\tilde{f}^{\nu}_{\lambda}-\tilde{f}^{\nu}_{\lambda'})|^{2}
\leq\Psi(\delta).
\end{align}
Thus,
\begin{align}\label{5.67}
&\bbint_{B_{10}(p)}\bigl| \langle\nabla \tilde{f}^{\alpha}_{\lambda},\nabla \tilde{f}^{\beta}_{\lambda'}\rangle-\delta_{\alpha\beta}\bigr|\\
\leq &\bbint_{B_{10}(p)}\bigl| \langle\nabla \tilde{f}^{\alpha}_{\lambda},\nabla \tilde{f}^{\beta}_{\lambda}\rangle-\delta_{\alpha\beta}\bigr|+\bbint_{B_{10}(p)}\bigl| \langle\nabla \tilde{f}^{\alpha}_{\lambda},\nabla (\tilde{f}^{\beta}_{\lambda'}-\tilde{f}^{\beta}_{\lambda})\rangle\bigr|\nonumber\\
\leq&\Psi(\delta).\nonumber
\end{align}

Denote $\bar{f}_{p}\overset{\Delta}{=}(f^{1},\ldots,f^{k})=F_{q}\circ f:B_{10}(p)\rightarrow \R^{k}$.
By (\ref{5.11}) and Cheeger-Colding theory, it is easy to see that $\bar{f}_{p}(B_{10}(p)) \subseteq B_{11}(\bar{f}_{p}(p))$.
Note that the map $y=\bar{f}_{p}=(f^{1},\ldots,f^{k})$ satisfies the equations
\begin{align}\label{5.14}
\frac{\partial}{\partial y^\alpha }H(x,y)=0,\qquad \alpha=1,\ldots,k.
\end{align}
It is obvious that $\frac{\partial H}{\partial y^\alpha }(x,y)$ is $W^{1,2}$ with respect to the $x$-variable, smooth with respect to the $y$-variable for $y\in B_{20}(f(x))$, and Lipschitz on $(x,y)\in T_{20}(\tilde{\Omega})\times \tilde{N}$.

In the following, we will prove that there exist $C_{9}$ and $\Psi_{9}$ such that
\begin{align}\label{5.18-1}
\mathrm{Lip} \bar{f}_p \leq C_{9},
\end{align}
\begin{align}\label{5.18-3}
\bbint_{B_{10}(p)}|\langle\nabla {f}^{i}, \nabla {f}^{j} \rangle- \delta_{ij}|\leq \Psi_{9}(\delta),
\end{align}
\begin{align}\label{5.18-2}
|\Delta {f}^{i}|_{L^{\infty}(B_{10}(p))}\leq C_{9}
\end{align}
for every $i,j\in\{1,\ldots,k\}$.

We first prove (\ref{5.18-1}), which is the base for (\ref{5.18-3}) and (\ref{5.18-3}).






By (\ref{5.162}), we have
\begin{align}\label{5.125}
\left| \frac{\partial^{2}H}{\partial y^{\alpha}\partial y^{\beta}}(x,y)-\delta_{\alpha\beta} \right|
=\left| \sum_{\lambda}\varphi_{\lambda}\left[\frac{1}{2}\frac{\partial^{2}d^{2}_{f_{\lambda}(x)}}{\partial y^{\alpha}\partial y^{\beta}}(y)-\delta_{\alpha\beta} \right]\right|
\leq \Psi(\delta)
\end{align}
for any $(x,y)\in B_{10}(p)\times B_{25}(q)$.
Here and in the following, we use $d(y_{1}, y_{2})$ to denote $d_{\tilde{N}}(y_{1}, y_{2})$ for short, and use $d_{y_{1}}(y_{2})$ to denote $d_{\tilde{N}}(y_{1}, y_{2})$ as a function of $y_{2}$.

By (\ref{5.12}), (\ref{5.160}) and (\ref{5.132}), each $\frac{\partial}{\partial y^\alpha } H$ is $C(n)$-Lipschitz on $B_{10}(p)\times B_{25}(q)$.
For any $x_1,x_2\in B_{10}(p)$, let $y_1=\bar{f}_{p}(x_1)$, $y_2=\bar{f}_{p}(x_2)$. Then we have
\begin{align}\label{7.28}
\left|\frac{\partial}{\partial y^\alpha } H(x_1,y)-\frac{\partial}{\partial y^\alpha } H(x_2,y)\right|\leq C(n)\tilde{d}_{\tilde{X}}(x_1,x_2)
\end{align}
for any $y\in B_{25}(q)$ and $\alpha=1,\ldots,k$.

Suppose $y_1-y_2\neq 0$. Denote $\vec{P}=(\omega^1,\ldots,\omega^k):=\frac{1}{|y_2-y_1|}(y_2-y_1)\in \R^k$, $\bar{H}:=\Sigma_{\alpha=1}^{k}\omega^\alpha \frac{\partial}{\partial y^\alpha }H$.
By (\ref{5.14}), we have
\begin{align}\label{7.29}
\bar{H} (x_1,y_1)-\bar{H} (x_1,y_2)=-[\bar{H} (x_1,y_2)-\bar{H} (x_2,y_2)].
\end{align}
According to Lagrange's mean value theorem, there exists some $y'$ lying in the segment between $y_1$ and $y_2$, such that $|\bar{H} (x_1,y_1)-\bar{H} (x_1,y_2)|=|\nabla_{\vec{P}}\bar{H}(x_1,y')|\cdot|y_1-y_2|$, where $\nabla_{\vec{P}}\bar{H}$ is the directional derivative along $\vec{P}$.
Note that by (\ref{5.125}), $|\nabla_{\vec{P}}\bar{H}(x_1,y')|\geq \frac{1}{2}$.
Combing it with (\ref{7.28}) and (\ref{7.29}), we have
\begin{align}
|y_{1}-y_2|\leq 2C(n)\tilde{d}_{\tilde{X}}(x_1,x_2),
\end{align}
which implies (\ref{5.18-1}).

Since $\tilde{X}$ satisfies doubling property and Poincare inequality, according to the generalized Rademacher theorem (\cite{C99}), $f^{1},\ldots,f^{k}$ are differentiable almost everywhere.

In the following, given a Lipschitz function $f':\tilde{X}\rightarrow \R$, we use $D f'$ to denote the differential of $f'$, i.e. the $L^{\infty}$-section of $T^{*}\tilde{X}$  determined by  $f'$ (see \cite{C99}).

For any $\lambda\in \mathcal{A}$, under the intermediate invariant $y_{1}=\tilde{f}_{\lambda}(x)$ introduced in (\ref{5.37}), if $\tilde{f}_{\lambda}(x)$ is differentiable at $x \in B_{10}(p)$, and $y$ satisfies $d_{\tilde{N}}(f_{\lambda}(x),y)<100$, then by the chain rule, at $x \in B_{10}(p)$,
\begin{align}\label{7.42}
D \frac{\partial d^{2}(f_{\lambda}(x),y)}{\partial y^{s}}= \frac{\partial^{2}d^{2}(y_{1},y_{2})}{\partial y_{1}^{\nu}\partial y_{2}^{s}} D \tilde{f}_{\lambda}^{\nu}(x),
\end{align}
where we use the smoothness of $d_{\tilde{N}}(y_1,y_2)$ if $d_{\tilde{N}}(y_1,y_2)<100$.

Thus for any $x \in B_{10}(p)$ at where all of $\varphi_{\lambda}$, $f_{\lambda}$, ($\lambda\in \mathcal{A}$) are differentiable, by the Leibnitz rule of differential and (\ref{7.42}),
\begin{align}\label{5.23}
D(\frac{\partial H}{\partial y^{s}})=\frac{1}{2}\sum_{\lambda}\left[\frac{\partial d^{2}(f_{\lambda}(x),y)}{\partial y^{s}} D \varphi_{\lambda}+\varphi_{\lambda} \frac{\partial^{2}d^{2}(y_{1},y_{2})}{\partial y_{1}^{\nu}\partial y_{2}^{s}}\biggl|_{\substack{y_{1}=\tilde{f}_{\lambda}(x)\\y_{2}=\bar{f}_{p}(x)}} D \tilde{f}_{\lambda}^{\nu}\right]
\end{align}
holds at $(x,\bar{f}_{p}(x))$.

Let $x \in B_{10}(p)$ be any point at where all of $\varphi_{\lambda}$, $f_{\lambda}$, ($\lambda\in \mathcal{A}$), $f^{1},\ldots,f^{k}$ are differentiable.
Suppose in addition that there exist a Borel set $U\ni x$ and a Lipschitz map $\phi: X \rightarrow \R^{n'}$ forming a chart around $x$, whose existence is ensured by \cite{C99}*{Theorem 4.38}.
We consider $U\ni x'\rightarrow x$. By (\ref{5.14}), for each $\alpha=1,\ldots,k$, we have
\begin{align}\label{7.29-11}
0=\left[\frac{\partial H}{\partial y^\alpha }(x',\bar{f}_{p}(x'))-\frac{\partial H}{\partial y^\alpha }(x',\bar{f}_{p}(x))\right]+\left[\frac{\partial H}{\partial y^\alpha }(x',\bar{f}_{p}(x))-\frac{\partial H}{\partial y^\alpha }(x,\bar{f}_{p}(x))\right].
\end{align}
According to Lagrange's mean value theorem, there exists some $\xi$ lying in the segment between $\bar{f}_{p}(x)$ and $\bar{f}_{p}(x')$, such that
\begin{align}\label{7.29-11qt}
&\frac{\partial H}{\partial y^\alpha }(x',\bar{f}_{p}(x'))-\frac{\partial H}{\partial y^\alpha }(x',\bar{f}_{p}(x))\\
=&\frac{\partial^{2}H}{\partial y^{\beta} \partial y^\alpha }(x',\xi)(f^{\beta}(x')-f^{\beta}(x))\nonumber\\
=&\left[\frac{\partial^{2}H}{\partial y^{\beta} \partial y^\alpha }(x,\bar{f}_{p}(x))+o(1)\right]\cdot \left[Df^{\beta}(x)(\phi(x')-\phi(x))+o(|x'-x|)\right] \nonumber\\
=&\frac{\partial^{2}H}{\partial y^{\beta} \partial y^\alpha }(x,\bar{f}_{p}(x))Df^{\beta}(x)(\phi(x')-\phi(x))+o(|x'-x|).\nonumber
\end{align}
Combing (\ref{7.29-11qt}) with (\ref{7.29-11}), at such $x$, we have
\begin{align}\label{5.15}
D\frac{\partial H}{\partial y^\alpha }(x,y)\biggl|_{y=\bar{f}_{p}(x)}=-\frac{\partial^{2}H}{\partial y^{\beta} \partial y^\alpha }(x,\bar{f}_{p}(x))Df^{\beta}(x).
\end{align}
Let $(K^{\alpha\beta})$ be the inverse matrix of $(\frac{\partial^{2}H}{\partial y^{\alpha}\partial y^\beta}\bigl|_{(x,\bar{f}_{p}(x))})$.
Then
\begin{align}\label{5.19}
D f^{\alpha}=-K^{\alpha\beta}D\frac{\partial H}{\partial y^{\beta}}.
\end{align}

Note that by (\ref{5.125}), we have
\begin{align}\label{5.90}
|K^{\alpha\beta}-\delta_{\alpha\beta}|\leq \Psi(\delta).
\end{align}

By (\ref{5.23}), (\ref{5.19}), (\ref{5.90}) and (\ref{5.161}), for any $x \in B_{10}(p)$ at where all of $\varphi_{\lambda}$, $f_{\lambda}$, ($\lambda\in \mathcal{A}$), $f^{1},\ldots,f^{k}$ are differentiable and around where a chart exists, we have
\begin{align}\label{5.27}
&\left\langle D f^{\alpha}, D f^{\beta}\right\rangle
\\
=&K^{\alpha s}K^{\beta t}\left\langle D\frac{\partial H}{ \partial y^{s}}, D\frac{\partial H}{\partial y^{t}}\right\rangle\nonumber\\
=&\frac{1}{4}( \delta_{\alpha s}\pm \Psi(\delta))(\delta_{\beta t}\pm \Psi(\delta))\sum_{\lambda,\mu}\bigl\langle\varphi_{\lambda} (-2\delta_{\nu s}\pm \Psi(\delta))D \tilde{f}_{\lambda}^{\nu}\nonumber\\
&+\frac{\partial d^{2}(f_{\lambda}(x),y)}{\partial y^{s}}D \varphi_{\lambda},\varphi_{\mu} (-2\delta_{\tau t}\pm \Psi(\delta))D  \tilde{f}_{\mu}^{\tau} +\frac{\partial d^{2}(f_{\mu}(x),y)}{\partial y^{t}}D \varphi_{\mu}\bigr\rangle. \nonumber
\end{align}

At $(x, f(x))$, we have
\begin{align}\label{5.29}
\left|\frac{\partial d^{2}(f_{\mu}(x),y)}{\partial y^{t}}\right|=2\left|d(f_{\mu}(x),f(x))\langle \nabla d, \nabla y^{t}\rangle\right|\leq \Psi(\delta).
\end{align}
Combing it with (\ref{5.12}), (\ref{5.5332}), we have
\begin{align}\label{5.31}
\left|\sum_{\lambda,\mu}\varphi_{\lambda} \frac{\partial d^{2}(f_{\mu}(x),y)}{\partial y^{t}}\langle D \tilde{f}^{\nu}_{\lambda}, D \varphi_{\mu}\rangle\right|\leq \Psi(\delta),
\end{align}
\begin{align}\label{5.33}
&\left|\sum_{\lambda,\mu} \frac{\partial d^{2}(f_{\lambda}(x),y)}{\partial y^{s}} \frac{\partial d^{2}(f_{\mu}(x),y)}{\partial y^{t}}\langle D \varphi_{\lambda}, D \varphi_{\mu} \rangle\right|\leq \Psi(\delta).
\end{align}



Substituting (\ref{5.31}),  (\ref{5.33}) into (\ref{5.27}), and then integrating over $x\in B_{10}(p)$, we have

\begin{align}\label{5.35}
&\bbint_{B_{10}(p)}\bigl|\langle\nabla f^{\alpha}, \nabla f^{\beta}\rangle-\delta_{\alpha\beta}\bigr| \\
=&\bbint_{B_{10}(p)}\bigl|\langle D f^{\alpha}, D f^{\beta}\rangle-\delta_{\alpha\beta}\bigr| \nonumber\\
\leq &\bbint_{B_{10}(p)}\bigl|\sum_{\lambda,\mu}\varphi_{\lambda}\varphi_{\mu} \langle\nabla \tilde{f}_{\lambda}^{\alpha},\nabla \tilde{f}_{\mu}^{\beta}\rangle-\delta_{\alpha\beta}\bigr|+ \Psi(\delta) \nonumber\\
\leq &\sum_{\lambda,\mu}\bbint_{B_{10}(p)}\varphi_{\lambda}\varphi_{\mu}\bigl|\langle\nabla \tilde{f}_{\lambda}^{\alpha},\nabla \tilde{f}_{\mu}^{\beta}\rangle-\delta_{\alpha\beta} \bigr|+ \Psi(\delta). \nonumber
\end{align}
There are at most $C_{2}^{2}$ terms in the summation of (\ref{5.35}) which do not vanish, then  (\ref{5.18-3})  follows from  (\ref{5.67}) and (\ref{5.35}).

Now we turn to prove that for each $\alpha=1,\ldots,k$, $f^{\alpha}$ belongs to the domain of the local Laplacian on $B_{10}(p)$, and (\ref{5.18-2}) holds.
By (\ref{5.19}), (\ref{5.90}) and the Leibniz rule for the divergence, it suffice to prove that
\begin{align}\label{5.80-1}
\left|\left\langle D K^{\alpha\beta}, D\frac{\partial H}{\partial y^{\beta}}\right\rangle \right|\leq C,
\end{align}
almost everywhere, and $\frac{\partial H}{\partial y^{\beta}}$ belongs to the domain of the local Laplacian with
\begin{align}\label{5.80-3}
\left|\Delta \frac{\partial H}{\partial y^{\beta}} \right|_{L^{\infty}}\leq C.
\end{align}

Note that
\begin{align}\label{5.80-2}
&D K^{\alpha\beta}=-K^{\alpha s}K^{\beta\gamma}D\frac{\partial^{2} H}{\partial y^{s}\partial y^{\gamma}} \\
=&-\frac{1}{2}K^{\alpha s}K^{\beta\gamma}\sum_{\lambda}\biggl[\frac{\partial^{2} d^{2}(f_{\lambda}(x),y)}{\partial y^{s}\partial y^{\gamma}} D \varphi_{\lambda}+\varphi_{\lambda} \frac{\partial^{3}d^{2}(y_{1},y_{2})}{\partial y_{1}^{\nu}\partial y_{2}^{s}\partial y_{2}^{\gamma}}\biggl|_{\substack{y_{1}=\tilde{f}_{\lambda}(x)\\y_{2}=\bar{f}_{p}(x)}} D \tilde{f}_{\lambda}^{\nu}\nonumber\\
&+\varphi_{\lambda}\frac{\partial^{3} d^{2}(f_{\lambda}(x),y)}{\partial y^{s}\partial y^{\gamma}\partial y^{\omega} }D f^{\omega}\biggr]. \nonumber
\end{align}
where we use the fact that $\frac{\partial^{2} H}{\partial y^{s}\partial y^{\gamma}}(x, \bar{f}_{p}(x))$ is Lipschitz, and its differential can be computed similar to (\ref{7.29-11}) and (\ref{7.29-11qt}).

(\ref{5.80-1}) follows from (\ref{5.23}), (\ref{5.80-2}), (\ref{5.160}), (\ref{5.18-1}), (\ref{5.5332}) and (\ref{5.12}).

By (\ref{5.23}), (\ref{5.5333}), (\ref{5.12}), (\ref{5.29}), (\ref{5.160}) and the Leibniz rule for the divergence, to prove (\ref{5.80-3}), it suffice to prove that
\begin{align}\label{5.80-4}
\left|\left\langle D \left(\frac{\partial d^{2}(f_{\lambda}(x),y)}{\partial y^{\beta}}\biggl|_{y=f(x)}\right), D\varphi_{\lambda}\right\rangle \right|\leq C,
\end{align}
\begin{align}\label{5.80-5}
\left|\left\langle D \left(\varphi_{\lambda} \frac{\partial^{2}d^{2}(y_{1},y_{2})}{\partial y_{1}^{\nu}\partial y_{2}^{\beta}}\biggl|_{\substack{y_{1}=\tilde{f}_{\lambda}(x)\\y_{2}=\bar{f}_{p}(x)}}\right), D\tilde{f}_{\lambda}^{\nu} \right\rangle \right|\leq C
\end{align}
almost everywhere.

Note that
\begin{align}\label{5.80-6}
D \left(\frac{\partial d^{2}(f_{\lambda}(x),y)}{\partial y^{\beta}}\biggl|_{y=f(x)}\right)=\frac{\partial^{2}d^{2}(y_{1},y_{2})}{\partial y_{1}^{\alpha}\partial y_{2}^{\beta}}\biggl|_{\substack{y_{1}=\tilde{f}_{\lambda}(x)\\y_{2}=\bar{f}_{p}(x)}}D \tilde{f}_{\lambda}^{\alpha}(x)+\frac{\partial^{2}d^{2}(f_{\lambda}(x),y)}{\partial y^{\beta}\partial y^{\alpha}} D\bar{f}_{p}^{\alpha}(x),
\end{align}
(\ref{5.80-4}) follows from  (\ref{5.160}), (\ref{5.18-1}), (\ref{5.5332}) and (\ref{5.12}).
Similarly, (\ref{5.80-5}) can be proved.
Thus we finish the proof of (\ref{5.18-2}).

By the above notation, for any $p\in \tilde{\Omega}$, take $\bar{\Xi}_{p}(x):=(F_{q},\mathrm{id})\circ \Xi (x)= (\bar{f}_{p}(x),v(x))$ (where $q=f(p)$), then  (\ref{5.1231}) and (\ref{5.1233}) hold, and (\ref{5.1232}) holds for $\alpha,\beta\in\{1,\ldots,k\}$ or $\alpha,\beta\in\{k+1,\ldots,k+s\}$.

For $\alpha\in\{1,\ldots,k\}$ and $\beta\in\{k+1,\ldots,k+s\}$, we take some $p_{\lambda}\in\mathcal{A}$ such that $d_{\tilde{X}}(p,p_{\lambda})\leq 2$, and consider the map $\tilde{f}_{\lambda}:B_{10}(p)\rightarrow \mathbb{R}^{k}$ given by
(\ref{5.37}).
Note that
\begin{align}\label{5.99}
|\bar{f}_{p}(x)-\tilde{f}_{\lambda}(x)|\leq \Psi(\delta)
\end{align}
holds for every $x\in B_{10}(p)$.
Then by Lemma  \ref{lem-6.1},
\begin{align}\label{4.3121111}
\bbint_{B_{10}(p)}|\nabla ({f}^{\alpha}-\tilde{f}_{\lambda}^{\alpha})|^{2}
\leq\Psi(\delta).
\end{align}
On the other hand, by (\ref{5.133567}), (\ref{5.43}) and (\ref{5.37}),
\begin{align}\label{4.32345}
\bbint_{B_{10}(p)}|\langle\nabla \tilde{f}_{\lambda}^{\alpha}, \nabla v^{\beta} \rangle| \leq \Psi(\delta).
\end{align}
Hence
\begin{align}\label{4.31111}
&\bbint_{B_{10}(p)}|\langle\nabla \bar{\Xi}_{p}^{\alpha}, \nabla \bar{\Xi}_{p}^{\beta} \rangle|\\
=&\bbint_{B_{10}(p)}|\langle\nabla {f}^{\alpha}, \nabla v^{\beta} \rangle|\nonumber\\
\leq&\bbint_{B_{10}(p)}|\langle\nabla ({f}^{\alpha}-\tilde{f}_{\lambda}^{\alpha}), \nabla v^{\beta} \rangle|+ \bbint_{B_{10}(p)}|\langle\nabla \tilde{f}_{\lambda}^{\alpha}, \nabla v^{\beta} \rangle|\nonumber\\
\leq&\Psi(\delta).\nonumber
\end{align}
The proof of Proposition \ref{fibration-blowup} is completed.
\end{proof}

\begin{proof}[Proof of Proposition \ref{fibration-smooth}]

If $\tilde{X}$ is a smooth Riemannian manifold, then each $u_p$ in  (\ref{5.11}) is smooth.
In addition, the good cut-off functions $\varphi_{\lambda}$ $(\lambda\in \mathcal{A})$ can be chosen to be smooth. Then by the implicit function Theorem, the map $f$, which is constructed by center of mass, is smooth.

Suppose $\tilde{X}$ satisfies the $(\Phi_{0}, 1; k+s,\delta)$-generalized Reifenberg condition.
Since $\bar{\Xi}_p:B_{10}(p)\rightarrow \R^{k+s}$ satisfies (\ref{5.1231}), (\ref{5.1233}) and (\ref{5.1232}), by Theorem \ref{thm-trans-GRC},
if $\delta$ is sufficiently small (depending on $n$ and $\Phi_{0}$), then $d\bar{\Xi}_p$ is non-degenerated on $B_{5}(p)$, and
\begin{align}\label{6.21}
B_{4}((0^{k},v(p)))\subset \bar{\Xi}_p(B_{4+\Psi(\delta)}(p)) .
\end{align}
Because $\Xi|_{B_{5}(p)}=(E_{q},\mathrm{id})\circ \bar{\Xi}_{p}$, and by the arbitrariness of $p$, $d\Xi$ is non-degenerated on $\tilde{\Omega}$.

Now we go to prove (A).  Under the assumption in (A), $\tilde{\Phi}:B_{L+200}(p_{0})\rightarrow B_{L+200}((q_{0},0^s))$ is a $\delta$-Gromov-Hausdorff approximation,  for any $z=(q,t)\in B_{L-1}((q_{0},0^s))$, there exists some $p_{\lambda}\in \mathcal{A}$ such that $B_{10}(p_{\lambda})\subset B_{L+200}(p_{0})$ and $d_{\tilde{Z}}(z, \tilde{\Phi}(p_{\lambda}))< \frac{5}{2}$.
Hence
$$d_{\R^{k+s}}((F_{\lambda}(q),t),\bar{\Xi}_{p_{\lambda}}(p_{\lambda}))<3,$$
and by (\ref{6.21}), we conclude that $(F_{\lambda},\mathrm{id})(z)\in \mathrm{Im}(\bar{\Xi}_{p_{\lambda}})$,
and  $z\in\mathrm{Im}(\Xi)$.
By the arbitrariness of $z$, we have $B_{L-1}((q_{0},0^s))\subset \mathrm{Im}(\Xi)$.
We take the open set ${U}:=\Xi^{-1}(B_{L-1}((q_{0},0^s)))\cap\tilde{\Omega}$, then  $\Xi:U\rightarrow B_{L-1}((q_{0},0^s))$ is a surjective fiberation and Gromov-Hausdorff approximation.

Combining (2) of Theorem \ref{thm-trans-GRC} with the fact that $F_{q}:(B_{10}(q),q)\to (B_{10}(0^k),0^k)$ is $(1+\Psi(\delta|n))$-Lipschitz,
it is easy to see that $\Xi$ satisfies (\ref{BiHolder}) for  suitable $\Psi(\delta)=\Psi(\delta|n,\Phi_{0})$.
This finishes the proof of (A).

(B) can be proved similarly. By (\ref{5.111}),  for any $q\in \tilde{N}$, there exists some $p_{\lambda}\in \mathcal{A}$ such that $|F_{\lambda}(q)|=d_{\tilde{N}}(q, \tilde{\Phi}(p_{\lambda}))< 3$. Then we can prove $\tilde{N}= \mathrm{Im}(\Xi)$.

The proof of Proposition \ref{fibration-smooth} is completed.
\end{proof}

\subsubsection{Proof of Theorem \ref{Fibrations}}\label{sec7.1.1}

In order to prove Theorem \ref{Fibrations}, we will use the Ricci flow to deform the metrics satisfying (\ref{5.2}) to ones satisfying (\ref{5.8})-(\ref{5.10}).
We recall the following result, which originates from Perelman's famous pseudo-locality theorem (see \cite{Per02}*{Theorem 10.1, Corollary 10.2}).

\begin{thm}\label{thm-psudoloc}
For every small $0<\alpha\ll 1$, there exist constants $\delta = \delta(k, \alpha), \epsilon = \epsilon(k, \alpha)$ with the following properties.

Suppose $(X, g(t))_{0 \leq t \leq 1}$  is a Ricci flow solution on a $k$-dimensional compact manifold $X$. Suppose
\begin{align}
\mathrm{Ric}(x, 0) \geq -\delta,
\end{align}
\begin{align}
d_{\mathrm{GH}}(B_{\delta^{-1}}(x), B_{\delta^{-1}}(0^{k}))< \delta, \quad \forall \, x \in X.
\end{align}
Then for each $t\in (0,\epsilon^{2}]$, we have
\begin{align}\label{7.67-1}
|\mathrm{Rm}|(x,t) \leq \alpha t^{-1},
\end{align}
\begin{align}\label{7.67-2}
|\nabla^{m}\mathrm{Rm}|(x,t) \leq C_{m}\alpha t^{-1-\frac{m}{2}},
\end{align}
\begin{align}\label{7.67-3}
\mathrm{Vol}^{g(t)} (B_{\sqrt{t}}^{g(t)}(x)) \geq c t^{\frac{m}{2}},
\end{align}
\begin{align}\label{7.67-4}
\inj_{g(t)}\geq \rho\sqrt{t},
\end{align}
and for any $x, y \in X$ with $d_t(x,y) \leq \sqrt{t}$ or $d_0(x,y) \leq \sqrt{t}$, we have
\begin{equation} \label{7.67-5}
|d_t(x,y) - d_0(x,y)| \leq \alpha \sqrt{t}, \quad t \in [0, \epsilon^2],
\end{equation}
where $ c = c(k)$, $\rho = \rho(k)$, $C_{m}=C_{m}(k)$ are constants depending only on $k$.
\end{thm}

\begin{rem}
(\ref{7.67-1}) and (\ref{7.67-3}) follow from the pseudo-locality Theorem, see \cite{TW15}*{Proposition 3.1}, \cite{Per02}*{Theorem 10.1, Corollary 10.2}.
(\ref{7.67-4}) follows from (\ref{7.67-1}), (\ref{7.67-3}) and a standard result of Cheeger-Gromov-Taylor (\cite{CGT82}).
(\ref{7.67-2}) follows from (\ref{7.67-1}) and Shi's estimate (\cite{Shi89}).
The proof of (\ref{7.67-5}) can be found in \cite{CRX19}*{Lemma 2.10} (note that we have slightly modified the small $\alpha$ in the statement in \cite{CRX19}*{Lemma 2.10}).
\end{rem}

For simplicity of notation, we use $\Psi(\delta)$ to denote suitable positive functions $\Psi(\delta|n,r_{0})$, and the value of $\Psi(\delta)$ may change from lines to lines.
We assume $r_{0}=1$ in (\ref{5.2}) for simplicity.

Firstly, we blow up $(M,g)$, $(N,h)$ by $\delta$, and then, by a suitable rechoice of $\delta$, we may assume
	\begin{align}\label{5.2-0}
    \Ric_{M}\geq -\delta,
    \end{align}
	\begin{align}\label{5.2-1}
    \Ric_N \geq -\delta,
    \end{align}
	\begin{align}\label{5.3-1}
    d_{\mathrm{GH}}(B_{\delta^{-1}}(q),B_{\delta^{-1}}(0^{k})) < \delta \quad \text{ for any } q\in N,
    \end{align}
    \begin{align}\label{5.4-1}
    d_{\mathrm{GH}}(M,N)\le\delta.
    \end{align}
We assume $\Phi: M\rightarrow N$ is a $\delta$-Gromov-Hausdorff approximation that fulfills (\ref{5.4-1}).

Let $\frac{\partial h(t)}{\partial t}=-2\Ric(h(t)), h(0)=h$ be the Ricci flow solution on $N$.

Take $\alpha=\Psi_{1}(\delta|k)$ in Theorem \ref{thm-psudoloc} such that $\alpha>\delta^{\frac{1}{2}}$, $\delta< \epsilon(n,\alpha)$ and the conclusions of Theorem \ref{thm-psudoloc} hold.
Denote $\hat{h}=\alpha^{-1}\delta^{-1}h(\delta)$, $\hat{g}=\alpha^{-1}\delta^{-1}g$, $\check{h}=\alpha^{-1}\delta^{-1}h$.
For simplicity, we will use $\hat{M}$, $\hat{N}$, $\check{N}$ to denote $(M,\hat{g})$, $(N,\hat{h})$, and $(N, \check{h})$ respectively.
According to Theorem \ref{thm-psudoloc},
\begin{align}
|\mathrm{Rm}_{\hat{N}}|\leq \alpha,
\end{align}
\begin{align}
|\nabla^{m} \mathrm{Rm}_{\hat{N}}|\leq C_m\alpha^{1+\frac{m}{2}}, \quad m\geq 1,
\end{align}
\begin{align}
\inj_{\hat{N}}\geq \rho\alpha^{-\frac{1}{2}},
\end{align}
and for any $x, y \in N$ with $d_{\hat{N}}(x,y) \leq \alpha^{-\frac{1}{2}}$ or $d_{\check{N}}(x,y) \leq \alpha^{-\frac{1}{2}}$, we have
\begin{equation} \label{7.67-6}
|d_{\hat{N}}(x,y) - d_{\check{N}}(x,y)| \leq \alpha^{\frac{1}{2}}.
\end{equation}
Hence, given any geodesic ball $\Omega=B_{\alpha^{-\frac{1}{2}}}(q)\subset \check{N}$,
the identity map $\id: (\Omega, \check{h})\rightarrow (\Omega, \hat{h})$ is a $\alpha^{\frac{1}{2}}$-Gromov-Hausdorff approximation.

By (\ref{5.2-0}) and (\ref{5.4-1}), we have
\begin{align}
\Ric_{\hat{M}}\geq -\delta^{2},
\end{align}
\begin{align}\label{5.4-1-1}
d_{\mathrm{GH}}(\hat{M},\check{N})\le\delta^{\frac{1}{4}},
\end{align}
and $\Phi: \hat{M}\rightarrow \check{N}$ is a $\delta^{\frac{1}{4}}$-Gromov-Hausdorff approximation.

Let $\hat{\Phi}:\hat{M}\rightarrow \hat{N}$ be the map given by the composition of $\Phi: \hat{M}\rightarrow \check{N}$ and $\id: \check{N}\rightarrow \hat{N}$, then for any $p\in \hat{M}$, $\hat{\Phi}|_{B_{\alpha^{-\frac{1}{2}}}(p)}$ is a $\Psi(\delta)$-Gromov-Hausdorff approximation onto its image.
According to Cheeger-Colding's theory, for any $p\in \hat{M}$, there exists a $\Psi(\delta)$-Gromov-Hausdorff approximation, $u_p:B_{100}(p)\rightarrow \R^k$, so that $u_p$ is $\Psi(\delta)$-splitting harmonic map and is $\Psi(\delta)$-close to $F_q\circ \tilde{\Phi}$, where $q=\tilde{\Phi}(p)$, and $F_{q}$ is the normal coordinate.
In addition, for any $q \in \hat{N}=\check{N}$, we can find a $p\in\hat{M}$ such that $d_{\check{N}}(q,\Phi(p))<\delta^{\frac{1}{4}}$, and then $d_{\hat{N}}(q,\hat{\Phi}(p))<\delta^{\frac{1}{4}}+\alpha^{\frac{1}{2}}<\Psi(\delta)$.

By Proposition \ref{fibration-blowup}, there exists a smooth map $\hat{f}: \hat{M}\rightarrow \hat{N}$ which is $\Psi(\delta|n,r_{0})$-close to $\hat{\Phi}$,
and for the normal coordinates $F_{q}$, where $q=\hat{f}(p)$, $\bar{f}:=F_{q}\circ \hat{f}|_{B_{10}(p)}$
satisfies estimates similar to (\ref{5.1231}), (\ref{5.1233}) and (\ref{5.1232}) with $\Psi_{0}(\delta|n)$ replaced by suitable $\Psi(\delta|n,r_{0})$.

Now we consider the map $\check{f}:\hat{M}\rightarrow\check{N}$ obtained by composing $\hat{f}$ with the identity map $\id:\hat{N}\rightarrow\check{N}$.

Since $d_{\hat{N}}(\hat{f}(x),\hat{\Phi}(x))\leq\Psi(\delta)$ holds for every $x\in \hat{M}$, by (\ref{7.67-6}), we have $d_{\check{N}}(\check{f}(x),\Phi(x))\leq\Psi(\delta)$.
Since $\Phi:\hat{M}\rightarrow\check{N}$ is a $\delta^{\frac{1}{4}}$-Gromov-Hausdorff approximation, it is easy to check that  $\check{f}$ is a $\Psi(\delta)$-Gromov-Hausdorff approximation for  suitable $\Psi(\delta)$.

Given any $q$, let $\Omega_{q}$ be the geodesic ball $B_{20}(q)\subset \hat{N}$, and $\Theta_{q}:(\Omega_{q}, \check{h})\rightarrow B_{20}(0^{k})$ be the smooth map obtained by composing the identity map $\id: (\Omega_{q}, \check{h})\rightarrow (\Omega_{q}, \hat{h})$ with $F_{q}:(\Omega_{q}, \hat{h})\rightarrow B_{20}(0^{k})$.
$\Theta_{q}$ gives a chart around $q$, but in general $\Theta_{q}$ is not the normal coordinate.
It is easy to see that $\Theta_{q}:(\Omega_{q}, \check{h})\rightarrow B_{20}(0^{k})$ is a $\Psi(\delta)$-Gromov-Hausdorff approximation for suitable $\Psi(\delta)$.
For any $p\in\hat{M}$, $\bar{f}= F_{q}\circ \hat{f}|_{B_{10}(p)}= \Theta_{q}\circ \check{f}|_{B_{10}(p)}$ satisfies the estimates (\ref{5.1231}), (\ref{5.1233}) and (\ref{5.1232}), thus if we rescale back the metrics, the map $f:(M,g)\rightarrow (N,h)$ given by $f(x)=\check{f}(x)$, together with the suitably rescaled coordinate charts $\Theta_{q}$ satisfy (\ref{1.1})-(\ref{1.3}).

Finally, if $M$ satisfies the $(\Phi_{0}, 1; k,\delta)$-generalized Reifenberg condition, then $\hat{M}$ also satisfies the generalized Reifenberg condition.
According to case (B) of Proposition \ref{fibration-smooth}, for sufficiently small $\delta$ (depending on $n$, $r_{0}$ and $\Phi_{0}$), $\hat{f}: \hat{M}\rightarrow \hat{N}$ is a surjective fibration, and hence $\check{f}:\tilde{M}\rightarrow\check{N}$ is also a surjective fibration.

The proof of Theorem \ref{Fibrations} is completed.

\subsection{Existence of equivariant regular maps}\label{sec-6-1}

Proposition \ref{prop-eq-fibration} and its blow up version, are equivariant versions of Proposition \ref{fibration-blowup}.  Theorem \ref{thm:eq-submetry}
is the case of $s=0$ in Proposition \ref{prop-eq-fibration}.

\begin{prop}\label{prop-eq-fibration}
Given $n,k\in \mathbb{Z}^{+}$, $s\in \mathbb{Z}^{+}\cup \{0\}$, with $n\geq k$, $k+s\leq n$,  $L\ge 10$.
Suppose $(N,h)$ is a $k$-dimensional compact Riemannian manifold.
Suppose $(X,d, m)$ is a $\RCD(-1,n)$-space, $U_{0}\subset X$ is a bounded open set,
$v:T_2(U_0)\rightarrow \R^{s}$ is a continuous map belonging to the domain of Laplacian, such that for every $\omega_0\in B_{\frac{1}{2}L}(0^s)(\subset \R^s)$, $v^{-1}(\omega_0)$ is a non-empty compact subset of $U_0$, and
\begin{align}\label{4.09--1}
\bbint_{B_{1}(p)}|\langle\nabla v^{i}, \nabla v^{j} \rangle- \delta_{ij}|\leq \delta,\quad  \text{for every } p\in U_0,
\end{align}
\begin{align}\label{4.10--1}
|\Delta v|_{L^{\infty}}\leq 1.
\end{align}
Suppose that there exist a finite group $\Gamma\subset \mathrm{Isom}(U_{0})$, and a closed Lie group  $G\subset \mathrm{Isom}(N)$   such that
\begin{align}\label{6.7--1}
v\circ \gamma = v,
\end{align}
for every  $\gamma \in \Gamma$.
Consider $N\times B_{L}(0^{s})$  which is equipped with the product metric,  and suppose  $\eta\in G$ acts on  $N\times B_{L}(0^{s})$  with the $B_{L}(0^{s})$-factor fixed. Denote by $\pi_{1}:N \times \mathbb{R}^{s}\rightarrow N$, $\pi_{2}:N \times \mathbb{R}^{s}\rightarrow \R^{s}$  the projections.
Suppose that there exist maps $\Xi: T_{2}(U_{0})\rightarrow N\times B_{L}(0^{s})$ and $\phi:\Gamma\rightarrow G$ such that
	\begin{description}
		\item[(a)]  for every  $\gamma \in \Gamma$ and  $x \in U_0$,
		\begin{align} \label{1.343411234}
			d\big(\Xi(\gamma(x)),\, \phi(\gamma)(\Xi(x))\big) \leq \delta,
		\end{align}
        \item[(b)] $\Xi_{2}:=\pi_{2}\circ \Xi=v$,
\end{description}
and $\Xi$ satisfies one of the following conditions:
\begin{description}
    \item[(c1)]\label{c1} for every $p \in U_0$ and every $r \in(0, 1)$, $\Xi_{1}:=\pi_{1}\circ \Xi$ satisfies
\begin{align}\label{7.589}
     \Xi_{1}(B_{r-\delta}(p)) \subset B_{r}(\Xi_{1}(p)) \subset T_{\delta}(\Xi_{1}(B_{r+\delta}(p)\cap v^{-1}(v(p)))),
\end{align}
\item[(c2)]\label{c2}  $\Xi:T_{2}(U_{0})\rightarrow N\times B_{L}(0^{s})$ is a $\delta$-Gromov-Hausdorff approximation.

\end{description}
Then there exist a homomorphism $\psi:\Gamma\rightarrow G$ and a Lipschitz map $\Upsilon: U\rightarrow N\times B_{\frac{1}{10}L}(0^{s})$ (where $U\subset U_{0}$ is an open set) so that
\begin{align}\label{8.987-1}
\pi_{2}\circ \Upsilon=v,
\end{align}
\begin{align}\label{8.987-2}
    |\Upsilon(p)-\Xi(p)|<\Psi_{1}(\delta) \quad \text{for every} \quad p\in U,
\end{align}
\begin{align}\label{8.987-3}
    \Upsilon\circ \gamma=\psi(\gamma)\circ \Upsilon \quad \text{for every} \quad \gamma\in \Gamma.
\end{align}
In addition, for any $p\in U$ such that $B_{\sqrt{\Psi_{1}(\delta)}}(p)\subset U$,  $\bar{\Upsilon}_{p}:=(\exp_{q}^{-1},\mathrm{id})\circ \Upsilon|_{B_{\sqrt{\Psi_{1}(\delta)}}(p)}$ satisfies
\begin{align}\label{8.987-4}
\Lip \bar{\Upsilon}_{p}\leq C(n),
\end{align}
\begin{align}\label{8.987-5}
|\Delta \bar{\Upsilon}_{p}^{\alpha}|_{L^{\infty}(B_{\sqrt{\Psi_{1}(\delta)}}(p))}\leq \frac{C(n)}{\sqrt{\Psi_{1}(\delta)}},
\end{align}
\begin{align}\label{8.987-6}
\bbint_{B_{\sqrt{\Psi_{1}(\delta)}}(p)}|\langle\nabla \bar{\Upsilon}_{p}^{\alpha}, \nabla \bar{\Upsilon}_{p}^{\beta} \rangle- \delta_{\alpha\beta}|\leq  \Psi_{1}(\delta)
\end{align}
for every $\alpha,\beta\in\{1,\ldots,k+s\}$, where $q=\pi_1\circ \Upsilon(p)$.
Here $\Psi_{1}(\delta)$ depends on $n$, $L$, $G$, $N$ and the $G$-action on $N$.
\end{prop}

\begin{cor}\label{cor7.2-smooth}
In Proposition \ref{prop-eq-fibration}, if $(X,d, m)$ is induced by an $n$-dimensional Riemannian manifold  $(X,g)$ with $\mathrm{Ric}\geq -1$, and $\Gamma$ is a compact Lie group, then $\Upsilon: U\rightarrow N\times B_{\frac{1}{10}L}(0^{s})$ can be constructed so that $\pi_{1}\circ \Upsilon$ is smooth.
Suppose in addition that $U_0$ satisfies the $(\Phi_{0}, 1; k+s,\delta)$-generalized Reifenberg condition for some function $\Phi_{0}:\mathbb{R}^{+}\rightarrow \mathbb{R}^{+}$ with $\lim_{s\rightarrow 0^{+}}\Phi_{0}(s)=0$,  then if $\Psi_{1}(\delta)$ is sufficiently small (depending on $\Phi_{0}$),  $\Upsilon$ is  a non-degenerate surjective fibration.
\end{cor}

\begin{proof}[Proof of Proposition \ref{prop-eq-fibration}.] Firstly, we have the following fact.

\begin{lem}\label{lem7.200001}
    Suppose $\Xi$ satisfies (c1), then for every  $y_0\in N$, $\omega_0\in B_{\frac{1}{2}L}(0^s)$, there exists $x\in U_{0}$ such that $d_N(\Xi_1(x),y_0))<\delta$ and $v(x)=\omega_0$.
\end{lem}

\begin{proof}
    For every $\omega_0\in B_{\frac{1}{2}L}(0^s)$, let $A:=\mathrm{Im}(\Xi_1(v^{-1}(\omega_0)))$, which is not empty. It suffices to prove  $N=T_\delta(A)$.  Obviously, $T_\delta(A)$ is open. We will prove that $T_\delta(A)$ is closed.
    Suppose $y_i\in  T_\delta(A)$, $y_i\rightarrow y_0$, then take $z_i\in A$ with $d_N(y_i,z_i)<\delta$, and $z_i=\Xi_1(x_i)$, $v(x_i)=\omega_0$ for some $x_i\in U_0$.  Up to choosing a subsequence, suppose $z_i\rightarrow z$. By (\ref{7.589}),
    \begin{align}
B_{\frac{1}{2}}(z_i)=B_{\frac{1}{2}}(\Xi_{1}(x_i)) \subset T_{\delta}(\Xi_{1}(B_{\frac{1}{2}+\delta}(x_i)\cap v^{-1}(\omega_0))).
\end{align}
Then $y_0\in B_{\frac{1}{2}}(z_i)\subset T_\delta(A)$.
\end{proof}
Hence, no matter whether $\Xi$ satisfies (c1) or (c2) , we always have:
\begin{cor}\label{cor-almost_onto}
     $\Xi:v^{-1}(B_{\frac{1}{2}L}(0^s))\rightarrow N\times B_{\frac{1}{2}L}(0^s) $ is $\delta$-almost onto.
\end{cor}
Recall the following lemma from \cite{MRW08}.

\begin{lem}[\cite{MRW08}*{Lemma A2}]\label{lem6.1}
Given a compact Lie group $G$ with a bi-invariant metric $g$ (which induces a distance $d$), there is $\epsilon(G)> 0$ so that if $H$ is a closed Lie group and $\phi:H\rightarrow G$ is a map such that $d(\phi(kh),\phi(k)\phi(h)) <\epsilon<\epsilon(G)$ for any $h,k\in H$, then there is a homomorphism $\psi:H\rightarrow G$  such that $d(\phi(h),\psi(h))<2\epsilon$ for all $h\in H$.
\end{lem}

The following lemma is a fibration version of \cite{MRW08}*{Lemma 3.2}.
\begin{lem}\label{lem6.2}
Under the assumption of Proposition \ref{prop-eq-fibration}, there exists a homomorphism $\psi:\Gamma\rightarrow G$  such that  for every $\gamma \in \Gamma$ and  $x \in U_0$,
		\begin{align} \label{1.343411234123}
			d\big(\Xi(\gamma(x)),\, \psi(\gamma)(\Xi(x))\big) \leq C_1\delta.
		\end{align}
where $C_{1}$ depends on $G$, $N$ and the $G$-action on $N$.

\end{lem}

\begin{proof}

We fix a bi-invariant metric on $G$, whose induced distance is denoted by $d_{0}$.
Note that $G$ has a natural metric defined by $d_G(\eta,\eta'):=\max\{d_{N}(\eta(y),\eta'(y))|y\in N\}$.
Since $G$ is compact and $N$ is a compact manifold, there exist positive constants $c_{0}\leq C_{0}$ depending on $G$, $N$ and the $G$-action on $N$ such that
\begin{align}\label{6.1}
c_{0}d_{0}\leq d_G \leq C_{0}d_{0}.
\end{align}

Given any $y\in N$, by Corollary \ref{cor-almost_onto}, take $x\in U_{0}$ such that $d(\Xi(x),(y,0))<\delta$, then for every $\gamma,\gamma'\in \Gamma_{i}$,
\begin{align}\label{6.2}
&d_{N}(\phi(\gamma\gamma')(y),\phi(\gamma)(\phi(\gamma')(y)))\\
\leq &d(\Xi(x),(y,v(x)))+d(\phi(\gamma\gamma')(\Xi(x)),\Xi(\gamma\gamma'(x)))+ d(\Xi(\gamma\gamma'(x)),\phi(\gamma)(\Xi(\gamma'(x)))) \nonumber\\
&+ d(\phi(\gamma)(\Xi(\gamma'(x))),\phi(\gamma)(\phi(\gamma')(\Xi(x))))+d(\Xi(x),(y,v(x))) \nonumber\\
\leq &  5\delta.  \nonumber
\end{align}

By Lemma \ref{lem6.1}, (\ref{6.1}) and (\ref{6.2}), for every sufficiently small $\delta$ depending on  $c_{0}$, $C_{0}$ (in (\ref{6.1})) and $\epsilon(G)$ (in Lemma \ref{lem6.1}), there exists a homomorphism $\psi:\Gamma\rightarrow G$ such that
\begin{align}\label{6.6}
d_G(\phi(\gamma),\psi(\gamma))\leq C_{0}d_{0}(\phi(\gamma),\psi(\gamma))<10\frac{C_{0}}{c_{0}}\delta
\end{align}
for all $\gamma\in \Gamma$.
For any $\gamma\in \Gamma$ and $x\in U_{0}$,
\begin{align}
&d(\Xi(\gamma(x)),\psi(\gamma)(\Xi(x)))\\
\leq &d(\Xi(\gamma(x)),\phi(\gamma)(\Xi(x)))+ d(\psi(\gamma)(\Xi(x)),\phi(\gamma)(\Xi(x)))\nonumber\\
\leq &\delta + 10\frac{C_{0}}{c_{0}}\delta :=C_{1}\delta.\nonumber
\end{align}
This completes the proof of Lemma \ref{lem6.2}.
\end{proof}

In the following we will consider $\psi$ instead of $\phi$. Without loss of generality, we may assume $C_1=1$ in (\ref{1.343411234123}).

In the remaining of the proof, we use $C$, $C_{2}$, $C_{3}, \ldots$ (respectively, $\Psi$, $\Psi_{2}$, $\Psi_{3}, \ldots$) to denote positive constants (respectively, positive functions with $\lim_{s\downarrow0}\Psi(s)=0$ etc.) depending only on $n$, $L$, $C_{1}$ and the manifold $N$.
The values of $C$ or $\Psi$ may change from lines to lines.

Let $\tilde{d}:=\delta^{-\frac{1}{2n}}d$, $\tilde{m}=m$, $\tilde{h}:=\delta^{-\frac{1}{n}}h$.
We use $\tilde{X}$, $\tilde{U}_{0}$, $\tilde{N}$ to denote $(X,\tilde{d}, \tilde{m})$, $(U_{0},\tilde{d}, \tilde{m})$, $(N,\tilde{h})$ respectively,
and use $\tilde{B}_{0}$, $\tilde{B}_{1}$, $\tilde{B}_{2}$ to denote $B_{\delta^{-\frac{1}{2n}}\frac{1}{4}L}(0^{s})$, $B_{\delta^{-\frac{1}{2n}}\frac{1}{5}L}(0^{s})$ and $B_{\delta^{-\frac{1}{2n}}\frac{1}{6}L}(0^{s})$ respectively.
$\Gamma$ and $G$ act isometrically on $\tilde{U}_{0}$ and $\tilde{N}\times \tilde{B}_{0}$ in a natural way.
Denote  $\tilde{\Xi}=\delta^{-\frac{1}{2n}}\Xi$, $\tilde{v}=\delta^{-\frac{1}{2n}}v$ and $\tilde{\Xi}_{1}=\pi_{1}\circ\tilde{\Xi}$.
Then $(\tilde{\Xi},\psi)$ satisfies the following:
\begin{description}
\item[(a')]  for every  $\gamma \in \Gamma$ and  $x \in\tilde{U}_0$,
		\begin{align} \label{1.3434112344141}
			d\big(\tilde{\Xi}(\gamma(x)),\, \psi(\gamma)(\tilde{\Xi}(x))\big) \leq \delta^{\frac{2n-1}{2n}};
		\end{align}
        \item[(b')] $\tilde{\Xi}_{2}:=\pi_{2}\circ \tilde{\Xi}=\tilde{v}$;
        \item[(c')]  for every $p \in \tilde{U}_0$ and every $r \in(0, \delta^{-\frac{1}{2n}})$, $\tilde{\Xi}_{1}:=\pi_{1}\circ \tilde{\Xi}$ satisfies
\begin{align}\label{7.589'}
     \tilde{\Xi}_{1}(B_{r-\delta^{\frac{2n-1}{2n}}}(p)) \subset B_{r}(\tilde{\Xi}_{1}(p)) \subset T_{\delta^{\frac{2n-1}{2n}}}(\tilde{\Xi}_{1}(B_{r+\delta^{\frac{2n-1}{2n}}}(p)\cap \tilde{v}^{-1}(\tilde{v}(p)))),
\end{align}
or $\tilde{\Xi}:T_{2\delta^{-\frac{1}{2n}}}(\tilde{U}_{0})\rightarrow \tilde{N}\times B_{\delta^{-\frac{1}{2n}}L}(0^{s})$ is a $\delta^{\frac{2n-1}{2n}}$-Gromov-Hausdorff approximation;
\item[(d')] $\tilde{\Xi}:\tilde{v}^{-1}(B_{\frac{1}{2}L\delta^{-\frac{1}{2n}}}(0^s))\rightarrow \tilde{N}\times B_{\frac{1}{2}L\delta^{-\frac{1}{2n}}}(0^s) $ is $\delta^{\frac{2n-1}{2n}}$-almost onto.
\end{description}
In addition,
\begin{align}\label{4.10--100}
|\Delta \tilde{v}^{i}|_{L^{\infty}}\leq \delta^{\frac{1}{2n}},
\end{align}
and for any $p\in \tilde{U}_{0}$, we have
\begin{align}\label{4.09--100}
&\bbint_{B_{100}(p)}|\langle\nabla \tilde{v}^{i}, \nabla \tilde{v}^{j} \rangle- \delta_{ij}|\\
\leq & \frac{m(B_{\delta^{-\frac{1}{2n}}}(p))}{m(B_{100}(p))}
\bbint_{{B}_{\delta^{-\frac{1}{2n}}}(p)}| \langle\nabla \tilde{v}^{i}, \nabla \tilde{v}^{j} \rangle- \delta_{ij}|\leq C_{2}\delta^{\frac{1}{2}}. \nonumber
\end{align}
Then by \cite{HondaPeng2023}*{Corollary 3.3}, on $\tilde{U}_{0}$ we have
\begin{align}\label{4.11--100}
\Lip \tilde{v}^{i}\leq C_{3}.
\end{align}

Suppose $\delta$ is sufficiently small (depending only on the injectivity radius lower bound and curvature bounds of $N$) such that, for each $q\in \tilde{N}$,
the exponential map $E_{q}\overset{\Delta}{=}\exp_{q}:(B_{1000}(0^k),0^k)\to (B_{1000}(q),q)$,
and its inverse map, denoted by $F_{q}:(B_{1000}(q),q)\to (B_{1000}(0^k),0^k)$ are both $\Psi_2(\delta)$-Gromov-Hausdorff approximations and $(1+\Psi_2(\delta))$-Lipschitz maps, and under the normal coordinate $F_{q}$, (\ref{5.160})-(\ref{5.161}) holds.
Since (c') holds, by an argument by contradiction we can prove that, for each $p\in \tilde{U}_{0}$, there exists a map $u_p: (B_{50}(p),p)\to (\R^k,0^k)$, so that $(u_{p},\tilde{v})|_{B_{50}(p)}$ is $\Psi_{3}(\delta)$-close to $(F_\lambda,\mathrm{id})\circ \tilde{\Xi}|_{B_{50}(p)}$, and (\ref{5.131})-(\ref{5.133567}) hold.

Denote $\tilde{U}_{1}=T_{200}(\tilde{\Xi}^{-1}(\tilde{N}\times \tilde{B}_{1}))$, $\tilde{U}_{2}=T_{1}(\tilde{\Xi}^{-1}(\tilde{N}\times \tilde{B}_{2}))$.
By Proposition \ref{fibration-blowup}, there exists a map $f:\tilde{U}_{1} \to \tilde{N}$, such that  for any $p\in \tilde{U}_{2}$, the map  $\bar{f}_{p}\overset{\Delta}{=}(f^{1},\ldots,f^{k})=F_{q}\circ f:B_{10}(p)\rightarrow \R^{k}$ satisfies  (\ref{5.18-1})-(\ref{5.18-2}) and
\begin{align}\label{6.13}
\bbint_{B_{10}(p)}|\langle\nabla {f}^{i}, \nabla \tilde{v}^{\alpha} \rangle|\leq \Psi_{4}(\delta)
\end{align}
for $i\in\{1,\ldots,k\}$ and $\alpha\in\{1,\ldots,s\}$, where  $q=f(p)$.
Moreover, $f$ is $\Psi_4(\delta)$-close to $\tilde{\Xi}_1$ on $\tilde{U}_{1}$.

For each $\gamma\in \Gamma$, we consider the map $\pi_{1}\circ \psi(\gamma^{-1})\circ \Pi\circ \gamma: \tilde{U}_{1}\rightarrow \tilde{N}$, which is short-listed by $f_{\gamma}$.
Since (\ref{6.7--1}) and (\ref{1.3434112344141}) hold,  it is easy to check that for any $x\in \tilde{U}_{1}$,
\begin{align}\label{6.11}
    d_{\tilde{N}}(f_{\gamma}(x),f(x))<\Psi_{5}(\delta).
\end{align}
In addition, since (\ref{5.18-1})-(\ref{5.18-2}) , (\ref{6.13}) and (\ref{6.7--1}) hold,
it is easy to see that for every $\gamma\in \Gamma$, $\bar{f}_{\gamma,p}=(f_{\gamma}^{1},\ldots,f_{\gamma}^{k}):=F_{f_{\gamma}(p)}\circ f_{\gamma}$ satisfies
\begin{align}\label{5.18-111}
\mathrm{Lip} \bar{f}_{\gamma,p} \leq C_{4},
\end{align}
\begin{align}\label{5.18-211}
|\Delta f_{\gamma}^{i}|_{L^{\infty}}\leq C_{4},
\end{align}
\begin{align}\label{5.18-311}
\bbint_{B_{10}(p)}|\langle\nabla f_{\gamma}^{i}, \nabla f_{\gamma}^{j} \rangle- \delta_{ij}|\leq \Psi_{6}(\delta),
\end{align}
\begin{align}\label{6.13-11}
\bbint_{B_{10}(p)}|\langle\nabla f_{\gamma}^{i}, \nabla \tilde{v}^{\alpha} \rangle|\leq \Psi_{6}(\delta)
\end{align}
for every $i,j\in\{1,\ldots,k\}$, $\alpha\in\{1,\ldots,s\}$.

We fix a bi-invariant probability measure $m_{\Gamma}$ on $\Gamma$ ($m_{\Gamma}$ can be chosen to be $m_{\Gamma}=\frac{1}{|\Gamma|}(\sum_{\gamma\in\Gamma} \delta_{\gamma})$), and define $E : \tilde{U}_{2}\times \tilde{N}\rightarrow \R$ by
\begin{align}\label{6.12}
E(x,y)=\frac{1}{2}\int_{\Gamma}d_{\tilde{N}}^{2}(f_{\gamma}(x),y)d m_{\Gamma}(\gamma).
\end{align}

By the convex radius estimate, for any fixed $x\in \tilde{U}_{2}$, $E(x,\cdot)$ is strictly convex on $B_5(f(x))$.
Let $\vartheta(x)$ be the unique minimum point of $E(x,\cdot)$ on $\tilde{N}$.
Then we define $\Upsilon=(\vartheta,\tilde{v}):\tilde{U}_{2}\rightarrow \tilde{N}\times \tilde{B}_{2}$.
The following properties are easy to check:
\begin{description}
  \item[(1)] $d_{\tilde{N}}(\vartheta(x),f(x))<\Psi(\delta)$ for every $x\in \tilde{U}_{2}$ (and hence $d_{\tilde{N}}(\vartheta(x),f_{\gamma}(x))<\Psi(\delta)$ for every $x$ and every $\gamma\in \Gamma$);
  \item[(2)] for every $\gamma\in \Gamma$, we have $\Upsilon\circ \gamma=\psi(\gamma)\circ \Upsilon$.
\end{description}

Let $p\in\tilde{U}_{3}:=\Upsilon^{-1}(\tilde{N}\times B_{\delta^{-\frac{1}{2n}}\frac{1}{7}L}(0^{s}))$ be fixed.
We will  abuse the notation that $y=(y^{i})$ stands for both a point around $q=\vartheta(p)$ and its coordinate with respect to the normal coordinate $F_{q}$.
Denote by $\bar{\vartheta}_{p}\overset{\Delta}{=}(\vartheta^{1},\ldots,\vartheta^{k})=F_{q}\circ \vartheta:B_{10}(p)\rightarrow \R^{k}$.

By definition, $y=\bar{\vartheta}_{p}=(\vartheta^{1},\ldots,\vartheta^{k})$ satisfies the equations
\begin{align}\label{5.14---1}
\frac{\partial}{\partial y^\alpha }E(x,y)=0,\qquad \alpha=1,\ldots,k.
\end{align}
Analogous to the proof of (\ref{5.18-1}), we can prove
\begin{align}\label{5.18---1}
\Lip \bar{\vartheta}_{p}\leq  C.
\end{align}

By an argument similar to (\ref{5.15}), we can prove that for every $x\in B_{10}(p)$ at where $\bar{\vartheta}_{p}$ is differentiable and around where a chart exists, it holds
\begin{align}\label{5.15-1111}
D\frac{\partial E}{\partial y^\alpha }+\frac{\partial^{2}E}{\partial y^{\beta}\partial y^{\alpha}} D\vartheta^{\beta}=0,
\end{align}
and hence
\begin{align}\label{5.19-1111}
D \vartheta^{\alpha}=-K^{\alpha\beta}D \frac{\partial E}{\partial y^{\beta}},
\end{align}
where $(K^{\alpha\beta})$ is the inverse matrix of $(\frac{\partial^{2}E}{\partial y^{\alpha}\partial y^\beta}\bigl|_{(x,\bar{\vartheta}_{p}(x))})$.
By (\ref{6.12}) and (\ref{5.162}), it is easy to see that
\begin{align}\label{6.9}
\bigl|K^{\alpha\beta}- \delta_{\alpha\beta}\bigr|\leq \Psi(\delta).
\end{align}

For each $\gamma\in\Gamma$, we consider the coordinate transformation matrix $\eta_{\gamma}:F_{f_{\gamma}(p)}(B_{\frac{21}{2}}(q))\rightarrow F_{q}(B_{\frac{21}{2}}(q))$.
i.e. $\eta_{\gamma}(\vec{y})=F_{q}(E_{f_{\gamma}(p)}(\vec{y}))$.
By Appendix \ref{sec-8},
we have
\begin{align}\label{5.43--1}
\left|\frac{\partial \eta_{\gamma}^{\nu}}{\partial y^{\xi}_{\gamma}}\right|\leq C_{5},
\qquad
\left|\frac{\partial^{2}\eta_{\gamma}^{\xi}}{\partial y_{\gamma}^{\omega} \partial y_{\gamma}^{\theta}}\right|\leq C_{5}.
\end{align}
Moreover, the matrix $(\frac{\partial \eta_{\gamma}^{\nu}}{\partial y^{\xi}_{\gamma}})_{\nu\xi}$ can be written in the form $(B^{(\lambda)}+A^{(\lambda)})_{\nu\xi}$, where the matrix $B^{(\lambda)}\in \mathrm{SO}(k)$, while the components of the matrix $A^{(\lambda)}$ satisfy $|A^{(\lambda)}_{\nu\xi}|\leq \Psi_7(\delta)$.

Let $\tilde{f}_{\gamma}= \eta_{\gamma}\circ \bar{f}_{\gamma,p}$.
It is easy to see
\begin{align}
|\tilde{f}_{\gamma}(x)-\bar{f}_{p}(x)|\leq \Psi(\delta)
\end{align}
for every $x\in B_{10}(p)$.
In addition, by  (\ref{5.18-111}), (\ref{5.18-211}) and (\ref{5.43--1}), we have
\begin{align}\label{5.5332--1}
|\Lip \tilde{f}_{\gamma}|\leq C,
\end{align}
\begin{align}\label{5.5333--1}
|\Delta \tilde{f}_{\gamma}^{i}|_{L^{\infty}}\leq C.
\end{align}
Then according to Lemma \ref{lem-6.1}, we have
\begin{align}\label{5.46---1}
\bbint_{B_{10}(p)}|\nabla(\tilde{f}_{\gamma}^{i}-{f}^{i})|^{2}
\leq \Psi(\delta).
\end{align}

On the other hand, analogous to (\ref{5.45}), we have
\begin{align}\label{5.18-311-11}
\bbint_{B_{10}(p)}|\langle\nabla \tilde{f}_{\gamma}^{i}, \nabla \tilde{f}_{\gamma}^{j} \rangle- \delta_{ij}|\leq \Psi(\delta).
\end{align}
Thus, similar to (\ref{5.67}), for any $\gamma,\gamma'\in \Gamma$,
\begin{align}\label{6.31}
\bbint_{B_{10}(p)} |\langle \nabla \tilde{f}_{\gamma}^{i},\nabla \tilde{f}_{\gamma'}^{j}\rangle-\delta_{ij}|\leq \Psi(\delta).
\end{align}

Since $\Gamma$ is finite, for almost every $x\in B_{10}(p)$, $f_{\gamma}$ is differentiable at $x$ for every $\gamma\in\Gamma$.
Then at such an $x$, and for those $y$ with $d_{\tilde{N}}(f(x),y)<50$, $\frac{\partial E}{\partial y^{\beta}}$ is differentiable at $x$, and
\begin{align}\label{6.29}
D \frac{\partial E}{\partial y^{\beta}}=\frac{1}{2}\int_{\Gamma} \frac{\partial^{2}d_{\tilde{N}}^{2}(y_{1},y_{2})}{\partial y_{1}^{\omega}\partial y_{2}^{\beta}}\biggl|_{\substack{y_{1}=\tilde{f}_{\gamma}(x)\\y_{2}=\bar{\vartheta}_{p}(x)}}\cdot D \tilde{f}_{\gamma}^{\omega} d m_{\Gamma}(\gamma).
\end{align}

By (\ref{5.161}),  (\ref{6.9}), (\ref{5.19-1111}), (\ref{6.31}), (\ref{6.29}) and (\ref{5.5332--1}), we have

\begin{align}\label{5.18---3}
&\bbint_{B_{10}(p)}\bigl|\langle\nabla \vartheta^{\alpha}, \nabla \vartheta^{\beta}\rangle-\delta_{\alpha\beta}\bigr| \\
=&\bbint_{B_{10}(p)}\bigl|\langle D \vartheta^{\alpha}, D \vartheta^{\beta}\rangle-\delta_{\alpha\beta}\bigr| \nonumber \\
=&\bbint_{B_{10}(p)}\bigl|K^{\alpha\rho}K^{\beta\theta}\langle D \frac{\partial E}{\partial y^{\rho}}, D\frac{\partial E}{\partial y^{\theta}}\rangle-\delta_{\alpha\beta}\bigr| \nonumber\\
=&\bbint_{B_{10}(p)}\biggl|(\delta_{\alpha\rho}\pm\Psi(\delta))(\delta_{\beta\theta}\pm\Psi(\delta))) \int_{\Gamma\times\Gamma} \bigl[(\delta_{\rho\omega}\pm\Psi(\delta))(\delta_{\theta\tau}\pm\Psi(\delta)) \nonumber\\
&\quad \langle D \tilde{f}_{\gamma}^{\omega}, D
 \tilde{f}_{\gamma'}^{\tau} \rangle  \bigr]
dm_{\Gamma}\otimes dm_{\Gamma}-\delta_{\alpha\beta}\biggr| \nonumber\\
=&\bbint_{B_{10}(p)}\biggl|\int_{\Gamma\times\Gamma} \bigl( \langle \nabla \tilde{f}_{\gamma}^{\alpha},\nabla \tilde{f}_{\gamma'}^{\beta} \rangle
-\delta_{\alpha\beta}\bigr)dm_{\Gamma}\otimes dm_{\Gamma}\pm \Psi(\delta)\biggr| \nonumber\\
\leq &\int_{\Gamma\times\Gamma} \biggl(\bbint_{B_{10}(p)}\bigl|\langle \nabla \tilde{f}_{\gamma}^{\alpha},\nabla \tilde{f}_{\gamma'}^{\beta} \rangle  -\delta_{\alpha\beta}\bigr|\biggr)
dm_{\Gamma}\otimes dm_{\Gamma} + \Psi(\delta) \nonumber\\
\leq & \Psi(\delta). \nonumber
\end{align}

Similar to the proof of Proposition \ref{fibration-blowup}, using the Leibniz rule for the divergence, we can prove that each $\alpha=1,\ldots,k$, $\vartheta^{\alpha}$ belongs to the domain of the local Laplacian on $B_{10}(p)$, and
\begin{align}\label{5.18---2}
|\Delta {\vartheta}^{\alpha}|_{L^{\infty}(B_{10}(p))}\leq C.
\end{align}

Since $d_{\tilde{N}}(\vartheta(x),f(x))<\Psi(\delta)$ holds for every $x\in\tilde{U}_{2}$, and (\ref{5.18-1}), (\ref{5.18-2}), (\ref{5.18---1}), (\ref{5.18---2}) hold, by Lemma  \ref{lem-6.1}, we have
\begin{align}\label{5.3121111}
\bbint_{B_{10}(p)}|\nabla ({\vartheta}^{\alpha}-{f}^{\alpha})|^{2}
\leq\Psi(\delta).
\end{align}
Hence for $\alpha\in\{1,\ldots,k\}$ and $\beta\in\{1,\ldots,s\}$,

\begin{align}\label{6.13123}
&\bbint_{B_{10}(p)}|\langle\nabla {\vartheta}^{\alpha}, \nabla \tilde{v}^{\beta} \rangle|\\
\leq &\bbint_{B_{10}(p)}|\langle\nabla {f}^{\alpha}, \nabla \tilde{v}^{\beta} \rangle| + \bbint_{B_{10}(p)}|\langle\nabla ({\vartheta}^{\alpha}-{f}^{\alpha}), \nabla \tilde{v}^{\beta} \rangle|\nonumber\\
\leq & \Psi(\delta),\nonumber
\end{align}
where we use  (\ref{4.11--100}), (\ref{6.13}), (\ref{5.3121111}), and the Cauchy-Schwarz inequality in the last line.

The proof of Proposition \ref{prop-eq-fibration} can be completed after we rescale back the distances.
\end{proof}

\begin{proof}[Proof of Corollary \ref{cor7.2-smooth}]

If $(X,g)$ is a Riemannian manifold $(X,g)$, then according to Proposition \ref{fibration-smooth}, the map $f:\tilde{U}_{1} \to \tilde{N}$ is smooth.
In particular, for every $\gamma \in \Gamma$, $f_{\gamma}$ is smooth.
Then (\ref{6.29}) holds in the smooth sense at every $x\in B_p(10)$, without assuming $\Gamma$ to be finite.
In addition, by the implicit function Theorem, $\vartheta$ is smooth, and (\ref{5.15-1111}) holds in the smooth sense.

Then the remaining part of the proof of Proposition \ref{prop-eq-fibration} goes through.

Suppose in addition that $U_0$ satisfies the $(\Phi_{0}, 1; k+s,\delta)$-generalized Reifenberg condition. Since $(\bar{\vartheta}_p, \tilde{v})=(F_{q},\id)\circ {\Upsilon}:B_{10}(p)\rightarrow \tilde{N}\times \tilde{B}_{1}$ satisfies (\ref{4.10--100}), (\ref{4.09--100}), (\ref{4.11--100}), (\ref{5.18---1}), (\ref{5.18---2}), (\ref{5.18---3}) and (\ref{6.13123}) , according to Theorem \ref{thm-trans-GRC},
if $\delta$ is sufficiently small (depending on $n$, $L$, $C_{1}$, $\Phi_{0}$ and $N$), then $d\Upsilon$ is non-degenerated on $B_{5}(p)$, and $B_4(0^{k+s})\subset \Upsilon(B_5(p))$.

By an argument similar to the last part of proof of Proposition \ref{fibration-smooth}, we can prove that there exists an open set $\tilde{U}\subset \tilde{U}_{2}$ such that ${\Upsilon}:\tilde{U}\rightarrow \tilde{N}\times B_{\delta^{-\frac{1}{2n}}\frac{1}{10}L}(0^{s})$ is a surjective fibration.
The proof of Corollary \ref{cor7.2-smooth} is completed after rescaling back the metrics.
\end{proof}

\appendix
\section{Normal coordinates}\label{sec-8}

Recall that a metric $g_{ij}(x)dx^{i}dx^{j}$ defined in a ball around the origin is in normal coordinates if every line through the origin is a geodesic (parametrized proportional to arc length) and if $g_{ij}=I_{ij}$ at the origin $x=0$.
The metric $g_{ij}$ is in normal coordinates is equivalent to $g_{ij}x^{i}=I_{ij}x^{i}$.

Let $\nabla$ be the covariant derivative with respect to the metric $g_{ij}$.

\begin{theorem}[\cite{Ham95}*{Theorem 4.5 and Corollary 4.11}]\label{thm-Curvbound-metricbound}
There exist positive constants $c_{1}$ and $\tilde{C}_{0}$ depending only on the dimension $n$ so that the following holds.
Suppose $|\Rm|\leq B_{0}$ in $B_{r}(p)$ and $\inj_{p}\geq r$, where $r\leq \frac{c_{1}}{\sqrt{B_{0}}}$, then in normal coordinates around $p$, at any $x\in B_{r}(p)$,
we have
\begin{align}
|g_{ij}-I_{ij}|\leq \tilde{C}_{0}B_{0}d(x,p)^{2}.
\end{align}
Furthermore, given $q\in \mathbb{Z}^{+}$, suppose $|\nabla^{k}\Rm|\leq B_{k}$ on $B_{r}(p)$, where $r\leq \frac{c_{1}}{\sqrt{B_{0}}}$, $k\leq q$, then
there exists a positive constant $\tilde{C}_{q}$ depending only on $n$, $q$ and $B_{k}$, $1\leq k\leq q$, so that in normal coordinates around $p$, at any $x\in B_{r}(p)$, we have
\begin{align}
\left|\frac{\partial^{q} }{\partial x^{l_{1}}\partial x^{l_{2}}\ldots\partial x^{l_{q}}}g_{ij}\right|\leq \tilde{C}_{q}.
\end{align}
\end{theorem}

\begin{theorem}[\cite{Ham95}*{Theorem 5.1}]\label{prop-trans-matrix}
Let $y = F(x)$ be an isometry from a ball in Euclidean space with a metric $g_{ij}dx^{i}dx^{j}$ to a ball in Euclidean space with a metric $h_{pq}dy^{p}dy^{q}$.
Assume $a_{0}I_{ij}x^{i}x^{j}\leq g_{ij}x^{i}x^{j}\leq A_{0}I_{ij}x^{i}x^{j}$, $a_{0}I_{pq}y^{p}y^{q}\leq h_{pq}y^{p}y^{q}\leq A_{0}I_{pq}y^{p}y^{q}$, and
\begin{align}
\left|\frac{\partial^{k} }{\partial x^{l_{1}}\partial x^{l_{2}}\ldots\partial x^{l_{k}}}g_{ij}\right|\leq A_{k},
\end{align}
\begin{align}
\left|\frac{\partial^{k} }{\partial y^{l_{1}}\partial y^{l_{2}}\ldots\partial y^{l_{k}}}h_{pq}\right|\leq A_{k}
\end{align}
for $1\leq k\leq s$.
Then for $1\leq k\leq s+1$, we have
\begin{align}
\left|\frac{\partial^{k} y^{p}}{\partial x^{l_{1}}\partial x^{l_{2}}\ldots\partial x^{l_{k}}}\right|\leq B_{k},
\end{align}
where $B_{k}$ depends only on $a_{0}$ and $A_{t}$, $0\leq t\leq k-1$.
\end{theorem}

\begin{theorem}[\cite{CCG-I-07}*{Proposition 4.48}]\label{thm-Curvbound-expdbound}
There exists a positive constant $c_{2}$ depending only on the dimension $n$ (we may assume $c_{2}\leq c_{1}$ in Theorem \ref{thm-Curvbound-metricbound}) so that the following holds.
Suppose $\inj_{p}\geq 4\frac{c_{2}}{\sqrt{B_{0}}}$, and $|R_{ijkl}|\leq B_{0}$, $|\nabla^{j}\Rm|\leq B_{j}$, $j\leq q$, $q\in \mathbb{Z}^{+}$, in $B_{r}(p)$, where $r\leq \frac{c_{2}}{\sqrt{B_{0}}}$.
Then for any $x,y\in B_{r}(p)$
such that $x$ is not in the cut locus of $y$,
we have
\begin{align}\label{6.3}
|\nabla^{l_{1}}_{x}\nabla^{l_{2}}_{y}\exp_{y}^{-1}x|\leq \bar{C}_{q}
\end{align}
for  $|l_{1}|+|l_{2}|\leq q-1$,
where $\bar{C}_{q}$ is a positive constant depending only on $n$, $q$ and $B_{j}$, $j\leq q$.
\end{theorem}

Given an $n$-dimensional manifold $(M,g)$ such that $|\Rm|\leq B_{0}\leq \delta$, $|\nabla^{j}\Rm|\leq 1$ for $j=1, 2, 3, 4$, and $\inj \geq \frac{1}{\delta}$.
For any fixed $p\in M$, we consider the normal coordinate around $p$.
According to Theorem \ref{thm-Curvbound-metricbound}, on $B_{\delta^{-\frac{1}{4}}}(p)$ we have
\begin{align}\label{Ap1.1}
|g_{jk}-I_{jk}|\leq \delta^{\frac{1}{2}},
\end{align}
\begin{align}
\left|\frac{\partial^{q} }{\partial x^{j_{1}}\ldots\partial x^{j_{q}}}g_{jk}\right|\leq \tilde{C}_{1}, \quad q=1, 2, 3, 4.
\end{align}
(\ref{Ap1.1}) means $g$ is $(1+2\delta^{\frac{1}{2}})$-bi-Lipschitz to $g_{\mathrm{Eucl}}$, and it is easy to see $B_{\delta^{-\frac{1}{4}}}(p)$ is $\delta^{\frac{1}{4}}$-Gromov-Hausdorff close to $B_{\delta^{-\frac{1}{4}}}(0^{n})$.
Furthermore, by an argument by contradiction, it is easy to prove that, for any given $L>0$,
\begin{align}
\left|\frac{\partial}{\partial x^{i}}g_{jk}\right|\leq \Psi(\delta|n, L)
\end{align}
holds on $B_{L}(p)$.
Thus the Christoffel symbols $\Gamma^{i}_{jk}$ satisfies
\begin{align}\label{6.7}
|\Gamma^{i}_{jk}|\leq \Psi(\delta|n, L)
\end{align}
on $B_{L}(p)$.

By Hessian comparison Theorem and (\ref{Ap1.1}), we have
\begin{align}\label{7.111112222}
\left|\nabla_{\frac{\partial}{\partial x^{i}}}\nabla_{\frac{\partial}{\partial x^{j}}} d^{2}(x,y)-2I_{ij}\right|\leq \Psi(\delta|n),
\end{align}
\begin{align}\label{7.112222}
\left|\nabla_{\frac{\partial}{\partial x^{i}}}\nabla_{\frac{\partial}{\partial y^{j}}}d^{2}(x,y)+2I_{ij}\right|\leq \Psi(\delta|n),
\end{align}
and then by (\ref{6.7}), we have
\begin{align}\label{7.11111}
\left|\frac{\partial^{2}}{\partial x^{i}\partial x^{j}}d^{2}(x,y)-2I_{ij}\right|\leq \Psi(\delta|n, L),
\end{align}
\begin{align}\label{7.11}
\left|\frac{\partial^{2}}{\partial x^{i}\partial y^{j}}d^{2}(x,y)+2I_{ij}\right|\leq \Psi(\delta|n, L).
\end{align}

For any $x,y\in B_{\delta^{-\frac{1}{4}}}(p)$ such that $d(x,y)\leq L$, since $d^{2}(x,y)=|\exp^{-1}_{y}x|^{2}$, by Theorem \ref{thm-Curvbound-expdbound}, we have
\begin{align}\label{6.71}
&|\nabla_{x}\nabla_{y}d^{2}(x,y)|=|\nabla_{x}\nabla_{y}|\exp_{y}^{-1}x|^{2}|\\
\leq& 2|\langle\nabla_{x}\nabla_{y} \exp_{y}^{-1}x, \exp_{y}^{-1}x\rangle|+ 2|\langle\nabla_{x} \exp_{y}^{-1}x, \nabla_{y}\exp_{y}^{-1}x\rangle|\nonumber\\
\leq& C_{1},\nonumber
\end{align}
and similarly, we can prove
\begin{align}\label{6.72}
&|\nabla_{x}^{l_{1}}\nabla_{y}^{l_{2}}d^{2}(x,y)|\leq  C_{2},
\end{align}
for $|l_{1}|+|l_{2}|=3$,
where $C_{1}$, $C_{2}$ are positive constants depending only on $n$ and $L$.

By (\ref{6.71}), (\ref{6.72}), (\ref{6.7}) and the boundedness of the Christoffel symbols (\ref{6.7}),
\begin{align}
\left|\frac{\partial^{3}}{\partial x^{i}\partial x^{j}\partial y^{k}}d^{2}(x,y)\right|\leq C_{3}
\end{align}
for any $x,y\in B_{\delta^{-\frac{1}{4}}}(p)$ satisfying $d(x,y)\leq L$,
with $C_{3}>0$ depends only on $n$ and $L$.

Suppose $B_{50}(p_{3})\subset B_{\delta^{-\frac{1}{4}}}(p_{1})\cap B_{\delta^{-\frac{1}{4}}}(p_{2})$.
On $B_{50}(p_{3})$, we use $\{x^{i}\}$, $\{y^{p}\}$ to denote normal coordinates around $p_{1}$ and $p_{2}$ respectively, and the metric is written in the forms $g_{ij}dx^{i}dx^{j}$ and $h_{pq}dy^{p}dy^{q}$ respectively.

Denote by $G=(g(\frac{\partial}{\partial x^{j}}, \frac{\partial}{\partial x^{k}}))_{1\leq j,k\leq n}$, $H=(h(\frac{\partial}{\partial y^{p}}, \frac{\partial}{\partial y^{q}}))_{1\leq p, q\leq n}$, $S=(S_{kq})_{1\leq k, q\leq n}$, where $S_{kq}=\frac{\partial y^{q}}{\partial x^{k}}$.
Then $G=SHS^{T}$.

At any $B_{50}(p_{3})$, since
$|g_{jk}-I_{jk}|\leq \delta^{\frac{1}{2}}$, there exists a $P\in \mathrm{SO}(n)$ and $\Lambda_{1}=\mathrm{diag}(\lambda_{1}, \ldots,\lambda_{n})$ with $|\lambda_{i}-1|\leq \Psi(\delta)$ so that $G=P^{T}\Lambda_{1}^{2}P$.
Similarly, $H=Q^{T}\Lambda_{2}^{2}Q$ for some $Q\in \mathrm{SO}(n)$ and $\Lambda_{2}=\mathrm{diag}(\mu_{1}, \ldots,\mu_{n})$ with $|\mu_{i}-1|\leq \Psi(\delta)$.

Thus $\Lambda_{1}^{-1}PSQ^{T}\Lambda_{2}^{2}QS^{T}P^{T}\Lambda_{1}^{-1}=\mathrm{Id}$, and hence $L:=\Lambda_{1}^{-1}PSQ^{T}\Lambda_{2}\in \mathrm{SO}(n)$.
Since $S= P^{T}\Lambda_{1}L\Lambda_{2}^{-1}Q$, there exist matrix value functions $B=(b_{kq})_{1\leq k, q\leq n}\in \mathrm{SO}(n)$ and $A=(a_{kq})_{1\leq k, q\leq n}$ with $|a_{kq}|\leq \Psi(\delta)$, so that
\begin{align}
\frac{\partial y^{q}}{\partial x^{k}}=a_{kq}+b_{kq}.
\end{align}

In addition, by Theorem \ref{prop-trans-matrix}, we have
\begin{align}
\left|\frac{\partial^{2} y^{p}}{\partial x^{i}\partial x^{j}}\right|\leq C.
\end{align}

\bibliographystyle{alpha}
\bibliography{ref}

\end{document}